\documentclass[11pt,letterpaper,reqno]{amsart}
\usepackage[
    top=2.7cm, bottom=2.1cm, inner=2.7cm, outer=2.7cm,
    marginparwidth=2cm 
    ]{geometry}

\usepackage{amssymb,latexsym, amsmath, amsxtra, mathrsfs, bm}
\usepackage[alphabetic]{amsrefs}
\usepackage[latin1]{inputenc}

\usepackage[dvips]{graphics}
\usepackage[all]{xy}
\usepackage{xcolor}
\usepackage{subcaption}
\usepackage{enumitem}
\usepackage{verbatim}
\usepackage[abs]{overpic}
\usepackage{hyperref}
\usepackage{mathtools}
\usepackage{todonotes}
\usepackage{tikz}           
\usepgflibrary{arrows.meta}                 
\usetikzlibrary{decorations.pathreplacing}  
\usepackage{tikz-cd}        
\usepackage{appendix}

\DeclareMathAlphabet{\mathbbold}{U}{bbold}{m}{n}

\allowdisplaybreaks[4]

\theoremstyle{plain}
        \newtheorem{theorem}{Theorem}[section]
        \newtheorem*{theorem*}{Theorem}
        \newtheorem*{conj*}{Conjecture}
        \newtheorem{lemma}[theorem]{Lemma}
        \newtheorem{fact}[theorem]{Fact}
        \newtheorem{prop}[theorem]{Proposition}
        \newtheorem{cor}[theorem]{Corollary}

        \newtheorem{thmx}{Theorem}

\theoremstyle{definition}
        \newtheorem{definition}[theorem]{Definition}
        \newtheorem{rem}[theorem]{Remark}

\theoremstyle{remark}

        \newtheorem*{notation}{Notation}

\numberwithin{equation}{section}
\numberwithin{theorem}{section}
\numberwithin{table}{section}
\numberwithin{figure}{section}

\makeatletter
\def\th@plain{%
	\thm@notefont{}
	\itshape 
}
\def\th@definition{%
	\thm@notefont{}
	\normalfont 
}
\makeatother

\newcommand{\id} {\operatorname{id}}

\newcommand{\R}{\mathbb{R}}

\newcommand{\Q}{\mathbb{Q}}
\newcommand{\N}{\mathbb{N}}
\newcommand{\Z}{\mathbb{Z}}

\providecommand{\abs}[1]{\lvert#1\rvert}
\providecommand{\Absbig}[1]{\bigl\lvert#1\bigr\rvert}

\providecommand{\AbsBig}[1]{\Bigl\lvert#1\Bigr\rvert}

\renewcommand{\:}{\colon}

\renewcommand{\le}{\leqslant}

\renewcommand{\ge}{\geqslant}

\newcommand{\X} {\mathbb{X}}

\newcommand{\E} {\mathbf{E}}

\newcommand{\vertiii}[1]{{\left\vert\kern-0.25ex\left\vert\kern-0.25ex\left\vert #1
    \right\vert\kern-0.25ex\right\vert\kern-0.25ex\right\vert}}

\renewcommand{\=}{\coloneqq}

\newcommand{\Cov}{\operatorname{Cov}}

\renewcommand{\E}{\mathbb{E}}

\renewcommand{\P}{\mathbb{P}}

\newcommand{\bL}{\bm{L}}

\newcommand{\cA}{\mathcal{A}}
\newcommand{\cB}{\mathcal{B}}
\newcommand{\cC}{\mathcal{C}}
\newcommand{\cD}{\mathcal{D}}
\newcommand{\cE}{\mathcal{E}}
\newcommand{\cF}{\mathcal{F}}
\newcommand{\cG}{\mathcal{G}}
\newcommand{\cH}{\mathcal{H}}
\newcommand{\cI}{\mathcal{I}}

\newcommand{\cL}{\mathcal{L}}
\newcommand{\cM}{\mathcal{M}}

\newcommand{\cP}{\mathcal{P}}
\newcommand{\cQ}{\mathcal{Q}}
\newcommand{\cR}{\mathcal{R}}
\newcommand{\cS}{\mathcal{S}}

\newcommand{\cU}{\mathcal{U}}

\newcommand{\fC}{\mathfrak{C}}

\newcommand{\fL}{\mathfrak{L}}

\newcommand{\hD}{\widehat{D}}
\newcommand{\hE}{\widehat{E}}

\newcommand{\oB}{\overline{B}}

\newcommand{\tD}{\widetilde{D}}

\newcommand{\trho}{\widetilde{\rho}}

\newcommand{\tphi}{\widetilde{\phi}}

\newcommand{\hd}{\operatorname{hd}}
\newcommand{\Nb}{B}
\newcommand{\Trans}{\operatorname{Trans}}
\newcommand{\wx}{\widehat{\X}}
\newcommand{\Int}{\operatorname{Int}}

\newcommand{\QIE}{\operatorname{QIE}}
\newcommand{\QIS}{\operatorname{QIS}}
\newcommand{\QI}{\operatorname{QI}}
\newcommand{\Ber}{\operatorname{Ber}}
\newcommand*{\boldone}{\text{\usefont{U}{bbold}{m}{n}1}}

\begin{document}
\title[Quasi-isometric rigidity]{Quasi-isometric rigidity\\ for random subsets in products of trees}
\author{Zhiqiang~Li \and Ranfeng~Yu \and Tianyi~Zheng}
\address{Zhiqiang~Li, School of Mathematical Sciences \& Beijing International Center for Mathematical Research, Peking University, Beijing 100871, China.}
\email{zli@math.pku.edu.cn}
\address{Ranfeng~Yu, School of Mathematical Sciences, Peking University, Beijing 100871, China.}
\email{rfyu@pku.edu.cn}
\address{Tianyi~Zheng, Department of Mathematics, University of California, San Diego, San Diego, CA 92093-0112}
\email{tzheng2@ucsd.edu}
	

\begin{abstract}
    In this article, we prove a rigidity result for quasi-isometric embeddings from a random subset $D$ of the product $\X$ of two regular trees into $\X$ itself. This can be seen as an extension of Eskin's quasi-isometric rigidity of higher-rank nonuniform lattices to random subsets. As a consequence, we give a description of the self-quasi-isometric embeddings of a random sample. We also show that two independent samples are almost surely non-quasi-isometric, confirming that such a phenomenon occurs in the higher-rank setting, as suggested by Ab\'ert. This result contrasts with the result on quasi-isometric equivalence between random sequences by Basu and Sly.
\end{abstract}

\maketitle

\section{Introduction}     \label{s:Introduction}
\subsection{Quasi-isometries of point processes}\label{s1.1}

    The study of quasi-isometries between finitely generated groups has been a central topic in geometric group theory since the seminal works of Gromov \cites{gromov1987hyper, gromov1991asym}. An important question in the broad program for studying finitely generated groups as geometric objects is the classification up to quasi-isometry of certain classes of groups. A finitely generated group $\Gamma$ is quasi-isometric to any space on which $\Gamma$ admits an isometric action that is properly discontinuous and cocompact. Thus, natural objects in the program are groups that admit such actions on spaces where a good understanding of the geometry can be attained.  

    A point process $\Pi$ on a locally compact group $G$ is a random closed and discrete subset of $G$, and it is said to be invariant if its law is invariant under translation by $G$. As an important example, a lattice $\Gamma$ in $G$ corresponds to a point process with underlying probability space $(G/\Gamma,m_{G/\Gamma})$, where $m_{G/\Gamma}$ is the normalized Haar measure, under the natural map $x\in G/\Gamma\mapsto \Pi_x=x\Gamma$. That is, a random translate of a lattice $\Gamma$ can be viewed as a point process on $G$.

    In probability theory, the most studied families of point processes are Bernoulli point processes on countable sets and Poisson point processes on non-discrete locally compact spaces. These processes enjoy complete spatial independence: numbers of points in disjoint regions are independent. Such features make them easier to understand than lattice points in some sense. 

    Let $d$ be a left-invariant metric on $G$, and consider a random subset of $G$ equipped with the induced metric. Studying quasi-isometry of such random metric spaces is natural, since quasi-isometries are concerned with large-scale geometry, while ignoring local details. In this work, we focus on the quasi-isometry rigidity properties of Bernoulli point processes.  An interesting question in this context is asked by Ab\'ert \cite{abert2010}:
    \begin{quote}
        \emph{Are two independent random subsets of $\Z$ as metric spaces quasi-isometric almost surely? How about other Cayley graphs, like for $\operatorname{SL}_3(\Z)$?}
    \end{quote}
    
    The reader may find more on the background of this question in the work of Peled \cite{Peled2010}. The conjecture concerning the existence of quasi-isometries between independent Bernoulli percolations on $\Z$ was confirmed by Basu and Sly in \cite{basu2014lipschitz}. Later, it was generalized to higher-dimensional lattice $\Z^n$, $n\ge 2$, by Basu, Sidoravicius, and Sly \cite{basu2018lipschitz}. Recently announced work by Athreya and Troscheit shows that there are quasi-isometries between independent Galton--Watson trees under certain conditions on the offspring distribution.

    Here we consider the other side where the rigidity phenomenon is expected. As suggested by Ab\'ert \cite{abert2010}, lattices in higher-rank simple Lie groups, such as $\operatorname{SL}_3(\Z)$, are natural candidates where independent random subsets are expected to be non-quasi-isometric, as inherited from the rigidity properties of the lattices. We confirm this phenomenon, on the product of two regular trees, which is an example of {\it rank-two} spaces. For many reasons, such as the lack of Kazhdan's property (T), such a product of rank one groups is considered less rigid than $\operatorname{SL}_3(\Z)$, but can clearly demonstrate the distinction between higher-rank and rank-one phenomena. 
    
    Conjecturally the rigidity phenomenon holds for general higher-rank lattices, also Poisson point processes on higher-rank symmetric spaces, but for exposition reasons, we choose to focus on this example, which shows the mechanisms for rigidity, without requiring preliminaries on the structure of Lie groups and their symmetric spaces. 
    
    In percolation theory, an important topic is the large-scale geometry of percolation clusters. Note that here we do not consider connectivity properties of clusters, but take the metric to be the \emph{induced metric} from the ambient space, and ask to what extent the sample of Bernoulli points inherits the geometry of the ambient space. 
\subsection{Splitting structure of QI embeddings}
    In this subsection, we introduce some notation and state the main theorem of the article.
   \begin{notation}\label{assumption}
        We use the following notation throughout the article.
       \begin{enumerate}[label=\rm{(\roman*)}]
           \smallskip     	\item Let $q\ge 3$ be an integer and $0<p<1$ be a constant.
           \smallskip     	\item Let $T_1$ and $T_2$ be the vertex sets of two $(q+1)$-regular trees, each equipped with the graph metric. Let $\X\= T_1\times T_2$ be their product. Let $\pi_1$ and $\pi_2$ be the projections from $\X$ to $T_1$ and $T_2$, respectively. Let $\X$ be equipped with the metric $d(x,y)\=
           \bigl(d_{T_1}(\pi_1(x),\pi_1(y))^2+d_{T_2}(\pi_2(x),\pi_2(y))^2\bigr)^{1/2}$.\footnote{The metric $d$ is bi-Lipschitz equivalent to the natural graph metric $d'(x,y) \coloneqq
           d_{T_1}(\pi_1(x),\pi_1(y))+d_{T_2}(\pi_2(x),\pi_2(y))$ on $\X$.}
           \end{enumerate}
   \end{notation}

   As already mentioned, in this work we always equip a subset $D$ of $\X$ with the induced distance.
   
  Recall the notion of quasi-isometric embedding: let $(X,d_X)$ and $(Y,d_Y)$ be metric spaces and let $\kappa>1$ and $C>0$ be constants. A \emph{$(\kappa,C)$-quasi-isometric embedding} from $X$ to $Y$ is a map $\phi\:X\to Y$ such that
       \begin{equation*}
           \frac{1}{\kappa}d_X(x_1,x_2)-C\le d_Y(\phi(x_1),\phi(x_2))\le \kappa d_X(x_1,x_2)+C
       \end{equation*}
       for all $x_1,\ x_2\in X$. The set of all quasi-isometric embeddings from $X$ to $Y$ is denoted by $\QIE(X,Y)$.
       
       In addition, if there exists $\psi\in \QIE(Y,X)$ such that $\sup_{x\in X} d_X(x,\psi(\phi(x)))<+\infty$ and $\sup_{y\in Y} d_Y(y,\phi(\psi(y)))<+\infty$, we say that $\phi$ is a \emph{quasi-isometry} from $X$ to $Y$ and $\psi$ is a \emph{coarse inverse} of $\phi$ (and vice versa). Equivalently, $\phi$ is a quasi-isometry if there exists some constant $d_0>0$ such that $d_Y(y,\phi(X))<d_0$ holds for every $y\in Y$. Two metric spaces $(X,d_X)$ and $(Y,d_Y)$ are \emph{quasi-isometric} if there exists a quasi-isometry between them.
       
       Let $(X,d_X)$ and $(Y,d_Y)$ be metric spaces. Denote the set of $(\kappa,C)$-quasi-isometric embeddings from $X$ to $Y$ by $\QIE_{\kappa,C}(X,Y)$. Denote by $\QIS_{\kappa,C,d_0}(X,Y)$ the set of $(\kappa,C)$-quasi-isometries $\phi$ from $X$ to $Y$ with $d_Y(y,\phi(X))<d_0$ for every $y\in Y$. 

    Next, we recall the definition of Bernoulli point processes. Let $S$ be a countable set. Denote by $\omega=(\omega_v:v\in S)$ an element of $\{0,\,1\}^{S}$. An element $\omega$ can be identified with its support $D=D(\omega)$, that is, the subset of $S$ consisting of $v$ with $\omega_v=1$. Equip $\{0,\,1\}^{S}$ with the product topology and its Borel $\sigma$-algebra. We refer to a distribution on $\{0,\,1\}^{S}$ as the law of a point process on $S$. Denote by $\P_p$ the product measure on $\{0,\,1\}^{S}$ where in each coordinate the distribution of $\omega_i$ is Bernoulli with parameter $p$. We refer to this model as the \emph{$\Ber(p)$ point process} on $S$.  
    The main result of this article is on the splitting structure of quasi-isometric embeddings: more precisely, for almost every $\Ber(p)$ point process sample $D$ on $\X$, $\phi\in \QIE(D,\X)$ is close to the product of two quasi-isometric embeddings of the factors.     
    
\begin{thmx}\label{t1.1}
     Let $D$ be a random subset of $\X$ having the law of the $\Ber(p)$ point process on $\X$. Then almost surely for each $\epsilon\in (0,1)$, each $\phi\in \QIE(D,\X)$, and each $e=(e_1,e_2)\in \X$, there exist quasi-isometric embeddings $\psi_1\:T_1\to T_1$, $\psi_2\:T_2\to T_2$ (possibly composed with a permutation of the two factors $T_1$ and $T_2$), and a constant $M>0$ such that for every $x=(x_1,x_2)\in D$, 
     \begin{equation}\label{e1.2}
         d(\phi(x),(\psi_1(x_1),\psi_2(x_{2})))\le M+\epsilon d(x,e).
     \end{equation} 
\end{thmx}
    We mention that such a splitting structure is known for quasi-isometries between products of irreducible symmetric spaces, see \cite[Theorem~1.1.2]{kleiner1997rigidity}; and for a quasi-isometry of an irreducible lattice in such a product, see \cite[Proposition~10.1]{eskin1998quasi}.

    Next, we state a refinement where the error term in the statement of Theorem~\ref{t1.1} depends only on the random subset $D$ seen from the point $x$. 
\begin{thmx}\label{t1.2}
    Let $D$ be a random subset of $\X$ having the law of the $\Ber(p)$ point process on $\X$. Given constants $\kappa>1$ and $C>0$, there exists a measurable map $\cR\: \{0,1\}^{\X}\times \X\to [0,+\infty]$ with the following properties:
    \begin{enumerate}[label=\rm{(\roman*)}]
        \smallskip
    	\item The map $\cR$ is translation invariant: for each $\gamma\in\operatorname{Aut}(\X)$, each $D_0\in \{0,1\}^{\X}$, and each $x\in \X$, we have $\cR(D_0,x)=\cR(\gamma(D_0),\gamma(x))$.
    	\smallskip
    	\item $\cR(D,x)$ is almost surely finite for each $x\in \X$. Moreover, for each $y\in \X$ and each pair of subsets $D_1\subseteq D_2\subseteq \X$, we have $\cR(D_1,y)\ge \cR(D_2,y)$.
        \smallskip
        
    	\item Almost surely for each $\phi\in \QIE_{\kappa,C}(D,\X)$, there exist quasi-isometric embeddings $\psi_1 \: T_1 \to T_1$ and $\psi_2\:T_{2}\to T_{2}$ such that, composing $\phi$ with a permutation of the two factors $T_1$ and $T_2$ if necessary, for every $x=(x_1,x_2)\in D$ we have
    	\begin{equation}\label{e1.3}
    		d(\phi(x),(\psi_1(x_1),\psi_2(x_{2})))\le \cR(D,x).
    	\end{equation}
    \end{enumerate}
\end{thmx}

    \begin{rem}
        A careful analysis of \cite{eskin1997quasi} would imply that $\cR(\,\cdot\,,x)$ has the exponential tail distribution given in Lemma~\ref{t5.0b}.
    \end{rem}

Given constants $\kappa$ and $C$, we view $\cR$ as a field of radii that depends only on the environment $D$ seen from the particle $x$. We refer to $\cR$ as the \emph{field of additive-error radii}. It will appear in the description of self QI-embeddings of $D$ in Theorem \ref{t1.4}. 
The field $\cR$ is explicitly given in Corollary~\ref{t6.1}. It is mainly determined by the \emph{irregularity radius} $\rho_0(D,x,\epsilon)$ defined in Definition~\ref{d5.0} for a suitable choice of $\epsilon$ depending only on $\kappa$.

\subsection{Quasi-isometric rigidity properties}
We now derive some results based on the splitting structure given in Theorem~\ref{t1.2}. We first verify the following statement, which answers for the product of two trees the question of Ab\'ert mentioned in Subsection~\ref{s1.1}.

\begin{thmx}\label{t1.3}
    Let $D_1$ and $D_2$ be two independent random subsets of $\X$ having the law of the $\Ber(p)$ point process on $\X$. Then there is almost surely no quasi-isometric embedding from $D_1$ to $D_2$. In particular, $D_1$ and $D_2$ are almost surely not quasi-isometric. 
\end{thmx}

We also prove that a quasi-isometric embedding from $D$ to itself must be close to the identity map in a suitable sense. Note that since $D$ admits arbitrary local patterns, the self-quasi-isometry group of $D$ may not be trivial. For example, suppose that $R>0$, $x,\,y$ are two points in $D$ with $d(x,y)=R$, and $B(x,4\kappa R+C)\cap D=\{x,\,y\}$. Then the map $\phi$ on $D$ that exchanges $x$ and $y$ and maps each other point to itself is a $(\kappa,C)$-quasi-isometry from $D$ to itself with $d(\phi,\operatorname{id})=R$. In words, the next statement shows that a $(\kappa,C)$-quasi-isometric embedding of $D$ to itself is the identity map up to the field of additive-error radii $\cR$. 

\begin{thmx}\label{t1.4}
    Let $D$ be a random subset of $\X$ having the law of the $\Ber(p)$ point process on $\X$. 
    Given constants $\kappa>1$ and $C>0$, there exists a constant $\lambda>1$ such that almost surely for each $\phi\in \QIE_{\kappa,C}(D,D)$ and each $x\in D$, we have
    \begin{equation*}
        d(\phi(x),x)<\lambda+\cR (D,x),
    \end{equation*}
    where $\cR(D,x)$ is from Theorem~\ref{t1.2}.
\end{thmx}

We remark on the difference between the statement of Theorem~\ref{t1.4} with known results on self-quasi-isometries of lattices. Recall that given a metric space $X$, $\QI(X)$ is defined as the set of self-quasi-isometries $X\to X$, modulo the equivalence relation that two maps are equivalent if they are within a bounded distance. A space with a small QI group has a strong form of QI rigidity. Let $\Gamma$ be a cocompact lattice in $\operatorname{Aut}(T_1)\times \operatorname{Aut}(T_2)$, then the QI group of $\Gamma$ is the same as $\QI(\X)$. By the splitting theorem in Kleiner--Leeb~\cite{kleiner1997rigidity}, up to permuting the two factors, $\QI(\X)$ is $\QI(T_1)\times \QI(T_2)$, and $\QI(T_i)$ is the group of quasi-symmetries of the Cantor set $\partial T_i$. In this case, $\QI(\Gamma)$ is a large infinite-dimensional group. 
Random irregularities in the subset $D$ of $\X$ prevent it from having such a large self-QI group.  Instead, we have the field $\cR$ that allows local additive errors.

\subsection{Background and motivations}

   The quasi-isometry rigidity properties of symmetric spaces, Euclidean buildings, and lattices in them have been extensively studied, see e.g., the survey~\cite{Farb1997Survey}. Recall that among rank-one symmetric spaces, real hyperbolic spaces, and complex hyperbolic spaces, each admits plenty of self-quasi-isometries; while Pansu's theorem states that a quasi-isometry of a quaternionic hyperbolic space, that is, the symmetric space of $\operatorname{Sp}(n,1)$, $n\ge 2$, lies within a bounded distance of an isometry \cite{pansu1989}. Kleiner and Leeb \cite{kleiner1997rigidity} proved rigidity theorems for quasi-isometries in the setting of higher-rank symmetric spaces and Euclidean buildings. The work of Kleiner and Leeb also established quasi-isometric classification of cocompact (uniform) lattices in higher-rank Lie groups. For nonuniform lattices, Schwartz~\cite{schwartz1995} obtained the first quasi-isometric classification results in a rank-one semisimple Lie group $G$, where $G$ is not $\operatorname{SL}(2,\mathbb{R})$. Such rigidity results were conjectured by Schwartz to hold for nonuniform lattices in higher rank in full generality; and for the Hilbert modular groups the classification was obtained by Farb and Schwartz in \cite{farbschwartz}. Eskin~\cite{eskin1998quasi} resolved Schwartz's conjecture positively and gave a complete characterization for quasi-isometries between nonuniform lattices in higher-rank Lie groups. A crucial input to Eskin's proof is the quasi-flats theorem from Eskin and Farb~\cite{eskin1997quasi}. The quasi-flats theorem was further extended to the case of reductive groups by Wortman \cite{wortman2006quasiflats}. 

   The notion of Measure Equivalence (ME) for countable groups, introduced by Gromov in \cite{gromov1991asym}, can be considered as a measure-theoretic analog of quasi-isometry between groups. An important example of measure equivalent groups are lattices in the same locally compact group $G$. As a consequence of the Ornstein--Weiss theorem, all countable amenable groups are measure equivalent to $\Z$. The work of Furman \cite{Furman1999ME} establishes ME rigidity for higher-rank lattices: any countable group that is ME to a lattice in a simple Lie group $G$ of higher rank, is commensurable to a lattice in $G$. 

   Recently in \cite{AbertMellick}, Ab\'ert and Mellick initiated the study of invariant point processes on locally compact groups from the point of view of Measure Equivalence.  Among their first results, it is shown that every free p.m.p.~action of $G$ can be realized by an invariant point process, see \cite[Theorem~1.1]{AbertMellick}.    
   
   Cost is a central notion in the study of measure equivalences, see the survey of Gaboriau \cite{GaboriauICM2010} and the references therein. Bernoulli point processes and Poisson point processes play special roles in the theory of the cost of p.m.p.~action in the works of Ab\'ert and Weiss \cite{AbertWeiss}, Ab\'ert and Mellick \cite{AbertMellick}. Poisson point processes play a crucial role in the recent resolution of the fixed price one conjecture for higher rank lattices by Fraczyk, Mellick, and Wilkens~\cite{fraczyk2023}. This important result motivates further investigation into point processes; their properties can provide insight into lattices beyond the reach of classical techniques.     

   We also mention that it is a yet to be explored topic to understand the interactions between QI-rigidity and ME-rigidity phenomena for point processes.

\subsection{Organization of the article} 
The strategy to prove the structural results in Theorems~\ref{t1.1} and~\ref{t1.2} follows the roadmap of the work of Eskin \cite{eskin1998quasi} on nonuniform lattices in higher-rank Lie groups. In our random point process setting, the geometry bears similarity to nonuniform lattices: a random subset $D$ admits arbitrarily large holes. 
    
   In Sections~\ref{step1} and~\ref{s:Step2}, we use quantitative estimates to demonstrate the coarse regularity of holes in Bernoulli samples and construct a well-defined boundary map induced by the QI embedding $\phi$. We first verify the regularity in one flat, in order to apply the quasiflats with holes theorem from Wortman~\cite{wortman2006quasiflats}. Note that in higher rank, a quasi-flat can be near a union of several flats. So to show the flat-to-flat result, one has to consider transverse flats and images of hyperplanes. We perform a careful analysis of the geometry of points in transverse flats and show that the embedding $\phi$ induces a boundary map in Section~\ref{s:Step2}.
    
    After the boundary map is defined, in Section~\ref{s5}, we use the geometry of the product of trees to prove its bi-H\"older continuity. In Section~\ref{s6}, we construct maps, which preserve each of the factor trees, as pullback maps of the boundary map and prove Theorems~\ref{t1.1} and~\ref{t1.2}. 
    
    In Section~\ref{s7}, we derive Theorems~\ref{t1.3} and~\ref{t1.4} via probabilistic arguments. For Theorem~\ref{t1.3}, we use a positive correlation between increasing events along with some other probabilistic estimates to rule out the existence of QI embedding between independent samples. Heuristically, if $\phi\:D_1\to D_2$ is a QI embedding, the splitting structure of $\phi$ imposes an upper bound on the possible number of choices of $\phi$. Under the independence assumption, such counting allows us to estimate the probability that a large region in $D_1$ can be QI embedded in $D_2$, and we conclude by the Borel--Cantelli lemma that the probability that such a map avoids holes in $D_2$ is $0$. A more careful argument along similar lines proves the statement on self QI-embeddings as stated in Theorem~\ref{t1.4}.
    
    \subsection*{Acknowledgments} 
We warmly thank Mikl\'os Ab\'ert for communicating the questions on quasi-isometries of random sets and for introducing us to the higher-rank rigidity aspects of the problem. We also want to thank Jinlong Liu for his careful reading and numerous comments. Z.~L.\ and R.~Y.\ were partially supported by Beijing Natural Science Foundation (JQ25001) and National Natural Science Foundation of China (12471083).

\section{Preliminaries}  \label{s:Preliminaries}

In this section, we recall some standard terminologies in metric geometry. We also recall some notions from symmetric spaces and give their descriptions in the setting of products of trees.

We use $\N$ to denote the set of positive integers. We always use the Euclidean metric to discuss subsets of $\Z^n$ and $\R^n$. Unless explicitly stated otherwise, the notation $B(\,\cdot\,,\,\cdot\,)$ always refers to a ball in $\X$ in this article.

\subsection{Definitions in metric spaces}
   
Graded quasi-isometric embeddings are generalizations of quasi-isometric embeddings where the additive constant is allowed to grow linearly with the specified rate $\epsilon$.

\begin{definition}[Graded quasi-isometric embeddings]
    For metric spaces $(Y_1,d_1),\ (Y_2,d_2)$ and real constants $\kappa>1,\ \epsilon\in (0,1),\ C>0$, we say that a map $\phi\:Y_1\to Y_2$ is a \emph{$(\kappa,\epsilon,C)$-graded quasi-isometric embedding} centered at $z_0\in Y_1$ if for all $x,\,y\in Y_1$ with $d_1(y,z_0)\ge d_1(x,z_0)$,
    \begin{equation*}
        \kappa^{-1}d_1(x,y)-\max\{C,\,\epsilon d_1(x,z_0)\}
        \le d_2(\phi(x),\phi(y))
        \le \kappa d_1(x,y)+\max\{C,\,\epsilon d_1(x,z_0)\}.
    \end{equation*}
    \end{definition}

\begin{definition}[Hausdorff distance and equivalence]\label{d2.4}
    Let $U$ and $V$ be subsets of a metric space $(X,d)$. Their \emph{Hausdorff distance} $\hd(U,V)$ is given by 
    \begin{equation*}
    \hd(U,V)\=\max \{\sup \{d(x,V): x\in U\},\,\sup \{d(U,y) : y\in V \} \}.
    \end{equation*} 
    If $\hd(U,V)<+\infty$, then we say that $U$ is \emph{(Hausdorff) equivalent} to $V$, denoted by $U\sim V$. 
\end{definition}

\begin{definition}[Annuli, conical neighborhoods, and $r$-interior]\label{d2.3}
    Suppose $(X,d)$ is a metric space, $U\subseteq X$, $x\in X$, and $0<r<R$ are constants. The \emph{annulus} $A(r,R ,x )$ is the set $\{y\in X : r<d(x,y)<R \}$. We simply write $A_r(x) \= A(r,2r ,x )$. 
    The \emph{conical neighborhood} of $U$ is defined as     
    \begin{equation*}
    U[r,x]\=\{ y\in X: d(y,U)\le r d(x,y)\}.
    \end{equation*}
    We often write $U[r]$ if the choice of the base point $x$ is clear. 
    The \emph{$r$-interior} of $U$ is defined as
    \begin{equation*}
    \Int_X(U,r)\= \{y\in U: d(y,U^c)\ge r\}.
    \end{equation*}
    We often write $\Int(U,r)$ if the ambient space $X$ is clear from the context.
\end{definition}

\begin{definition}[$\epsilon$-equivalence]\label{d2.5}
   Let $U$ and $V$ be subsets of a metric space $(X,d)$. Fix a constant $\epsilon\in (0,1)$. If there exist $x\in X$ and $R>0$ such that
   \begin{equation} \label{e:d2.5}
   U\smallsetminus B(x,R)\subseteq V[\epsilon]\quad \text{and}\quad V\smallsetminus B(x,R)\subseteq U[\epsilon],
   \end{equation}
   then we say that $U$ is \emph{$\epsilon$-equivalent} to $V$, denoted by $U\sim_{\epsilon} V$. Furthermore, we use the more quantitative notation $U\sim_{\epsilon,x,R}V$ if \eqref{e:d2.5} holds for $x$ and $R$.
\end{definition}
\begin{rem}\label{r2.5}
    Suppose $\epsilon\in(0,1/100)$. Suppose $R_0>0$ and $x\in X$. Suppose $U$ and $V$ are subsets of a metric space $(X,d)$ such that $\hd(U\cap B(x,R),U\cap B(x,R-\epsilon R))\le 2\epsilon R$ for all $R\ge R_0$. Then if $U\sim_{\epsilon,x,R_0}V$, then it is not hard to deduce $\hd(A_R(x_0)\cap U,A_R(x_0)\cap V)\le 6\epsilon R$ by geometric considerations.
\end{rem}

\subsection{Terminology from symmetric spaces}
To stay consistent with related works, we use terminology from the geometry of symmetric spaces. On the other hand, since we work in a special setting, our proof requires little knowledge of Lie groups. In this subsection, we recall these concepts in the setting of the product of two trees $\X\= T_1\times T_2$, equipped with the metric $d((x_1,x_2),(y_1,y_2))=\bigl(d_{T_1}(x_1,y_1)^2+d_{T_2}(x_2,y_2)^2\bigr)^{1/2}$. For general definitions, we refer the interested reader to \cite{eskin1998quasi} and references therein.

\begin{definition}[Flats and hyperplanes]\label{d2.9}
     For bi-infinite geodesics $L_1\subseteq T_1$ and $L_2\subseteq T_2$, we call their product $F \=  L_1\times L_2$ a\emph{ (maximal) flat} of $\X$. All flats are isomorphic to $\Z^2$ as metric spaces by identifying each geodesic with $\Z$. 
    For points $x_1\in T_1$, $x_2\in T_2$, the sets $L_1\times \{x_2\}$ and $\{x_1\}\times L_2$ are called \emph{hyperplanes}. In other words, hyperplanes are horizontal or vertical lines in flats. We say that $L_1\times \{x_2\}$ (resp.\ $\{x_1\}\times L_2$) is a hyperplane in $T_1$ (resp.\ $T_2$).

    With slight abuse of notation, we sometimes do not make a distinction between $L_1$ and $L_1\times \{x_2\}$ when the choice of $x_2$ is irrelevant. 
\end{definition}

\begin{definition}[Chambers and walls]
   Let $L_1\subseteq T_1$ and $L_2\subseteq T_2$ be geodesic rays emitting from $x_1\in T_1$ and $x_2\in T_2$, respectively. Then we call $L_1\times L_2$ a \emph{chamber} $\fC$ with origin at $(x_1,x_2)$. The rays $L_1\times\{x_2\}$ and $\{x_1\}\times L_2$ are called \emph{walls} of $\fC$. Two chambers are called \emph{neighboring chambers} if they share a wall.
\end{definition}

\begin{definition}[Polyhedra]\label{d2.10} 

    Suppose $L_i\subseteq T_i$ is a geodesic (line, ray, or segment) for each $i\in \{1,\,2\}$. Then we call $L_1\times L_2$ a \emph{polyhedron}. If $x_1\in L_1$ (resp.\ $x_2\in L_2$) is an endpoint of $L_1$ (resp.\ $L_2$), then we call $\{x_1\}\times L_2$ (resp.\ $L_1\times\{x_2\}$) a \emph{side} of $L_1\times L_2$.
\end{definition}
    
Take a reference point $e=(e_1,e_2)$ in $D$.
For simplicity of notation, we also refer to the copy $T_1\times \{e_2\}\subseteq \X$ as $T_1$ and $\{e_1\}\times T_2\subseteq \X$ as $T_2$.

\begin{definition}[Boundary]\label{d2.13}
    The \emph{tree boundary} $\partial T_i$ is the set of equivalence classes of geodesic rays in $T_i$ in the sense of Definition~\ref{d2.4}. The \emph{product boundary} $\widehat{\X}$ is the set of the equivalence classes of chambers in $\X$ in the sense of Definition~\ref{d2.4}. 
\end{definition}

Note that each chamber is equivalent to exactly one chamber with origin at $e$. Therefore, $\widehat{\X}$ is also viewed as the set of chambers with apex at $e$, where each chamber is given by a pair of points $(\xi,\eta)\in \partial T_1\times \partial T_2$.

\section{Flats to finite unions of flats}     \label{step1}
Throughout this section until the end of Section~\ref{s6}, we are in the setting of Theorem~\ref{t1.1}. Let $\phi\in \QIE(D,\X)$ be a quasi-isometric embedding. In the first step, we aim to prove that each flat is almost surely mapped near a finite union of flats (see Proposition~\ref{t3.5}).

The cardinality of a finite set $A$ is denoted as $\abs{A}$. In the case of an infinite set $A$, we always set $\abs{A} = +\infty$.

Let $L_1\subseteq T_1$ and $L_2\subseteq T_2$ be two bi-infinite geodesics. Then $F \=  L_1\times L_2$ is a flat in $\X$. 

In the remaining part of this section, we use $B(x,\,\cdot\,)$ to denote a ball in  $\Z^2$ equipped with the Euclidean metric for $x\in \Z^2$.
\begin{definition}\label{d3.1}
Let $0<\epsilon<1$ and $\rho >0$ be constants. For a subset $U\subseteq \Z^2$,  define $U_{(\epsilon,\rho)}$ as the set of $x\in U$ satisfying that for all $y\in \Z^2$ with $d(x,y)> \rho$, 
\begin{equation*}\label{e3.1a}
\abs{B(y,\epsilon d(x,y))\cap U}\ge (1-\epsilon)\abs{B(y,\epsilon d(x,y))}
\end{equation*}
holds. We call $U_{(\epsilon,\rho)}$ the \emph{$(\epsilon,\rho)$-quasi-center} of $U$.
\end{definition}

\begin{notation}
    Given $\epsilon\in (0,1)$, define the choice of $n_\epsilon$ to be the smallest integer that satisfies 
    \begin{equation}\label{e3.0a}
        n_\epsilon^2\ge 11\log_{1/(1-p)}(10/\epsilon).
    \end{equation}
    The reason for such a choice is explained in the proof of Lemma~\ref{t3.3}.
\end{notation}

\begin{notation}
    Let $\epsilon\in (0,1)$ be a constant and $D$ be a subset of $\X$. Let $F$ be a flat and identify $F$ with $\Z^2$. For each pair of $i,\,j\in \Z$, denote
    \begin{equation}\label{e3.0b}
        A_{ij}^{\epsilon} \= \bigl\{(x,y)\in \Z^2: in_\epsilon\le x<(i+1)n_\epsilon,\ jn_\epsilon\le y<(j+1)n_\epsilon\bigr\}.
    \end{equation}
    Define
    \begin{equation}\label{e3.0c}
        U_0^{\epsilon,D}(F) \= \bigcup_{i,j\in \Z:A_{ij}^{\epsilon}\cap D\neq \emptyset}A_{ij}^{\epsilon}\quad \text{ and }\quad V_0^{\epsilon,D}(F) \= \Z^2\smallsetminus U_0^{\epsilon,D}(F).
    \end{equation}
    For a constant $\rho>0$, define
    \begin{equation}\label{e3.0d}
        U_1^{\epsilon,D,\rho}(F) \= \bigl(U_0^{\epsilon,D}(F)\bigr)_{(\epsilon,\rho)} \quad\text{ and }\quad U_2^{\epsilon,D,\rho}(F) \= \bigl(U_1^{\epsilon,D,\rho}(F)\bigr)_{(\epsilon,\rho)}.
    \end{equation}
\end{notation}
\begin{definition}[Grid map]\label{d3.3}
    Let $\kappa>1$, $C>0$, and $\epsilon\in (0,1)$ be constants. Let $F$ be a flat, $D\subseteq\X$, and $\phi\in \QIE_{\kappa,C}(D,\X)$. Let $n_\epsilon$ be the constant defined in \eqref{e3.0a} and $ U_0^{\epsilon,D}(F)$ be the set defined in \eqref{e3.0c}. For each $A_{ij}^{\epsilon}\subseteq U_0^{\epsilon,D}(F)$, choose a point $y_{ij}\in A_{ij}^{\epsilon}\cap D$. Define the \emph{grid map} $\phi_{F,\epsilon}\: U_0^{\epsilon,D}(F)\to \X$ in the following way:
     \begin{enumerate}[label=\rm{(\roman*)}]
        \smallskip     	\item  define $\phi_{F,\epsilon}(x) \=  \phi(y_{ij})$ if $x\in A_{ij}^{\epsilon}\smallsetminus D$ for some $i,\,j\in\Z$;
        \smallskip     	\item define $\phi_{F,\epsilon}(x)\coloneqq\phi(x)$ if $x\in D\cap F$.
    \end{enumerate}     
    \begin{figure}
        \centering\begin{tikzpicture}[scale=0.5]

    \def\ne{2}

\fill[yellow!30] (0*\ne, 0*\ne) rectangle (1*\ne, 1*\ne);
\fill[yellow!30] (1*\ne, 0*\ne) rectangle (2*\ne, 1*\ne);
\fill[yellow!30] (2*\ne, 0*\ne) rectangle (3*\ne, 1*\ne);
\fill[yellow!30] (3*\ne, 0*\ne) rectangle (4*\ne, 1*\ne);

\fill[yellow!30] (1*\ne, 1*\ne) rectangle (2*\ne, 2*\ne); 
\fill[yellow!30] (2*\ne, 1*\ne) rectangle (3*\ne, 2*\ne); 
\fill[yellow!30] (3*\ne, 1*\ne) rectangle (4*\ne, 2*\ne);

\fill[yellow!30] (0*\ne, 2*\ne) rectangle (1*\ne, 3*\ne); 
\fill[yellow!30] (1*\ne, 2*\ne) rectangle (2*\ne, 3*\ne); 
\fill[yellow!30] (3*\ne, 2*\ne) rectangle (4*\ne, 3*\ne); 
\fill[yellow!30] (4.3*\ne, 1.5*\ne) rectangle (4.7*\ne, 1.9*\ne);
\node at (5.7*\ne, 1.7*\ne) {$U_0^{\epsilon,D}(F)$};
\fill[black] (4.5*\ne, 1.1*\ne) circle (1.2pt);
\node at (5.9*\ne, 1.1*\ne) {points in $D$};
    \foreach \x in {0,1,2,3} {
        \foreach \y in {0,1,2} {
            \ifnum\x<3
                \draw[black, thin] (\x*\ne+\ne, \y*\ne) -- (\x*\ne+\ne, \y*\ne+\ne);
            \fi
            \ifnum\y<2
                \draw[black, thin] (\x*\ne, \y*\ne+\ne) -- (\x*\ne+\ne, \y*\ne+\ne);
            \fi
        }
    }

    \draw[<->, thick, black] (0.5*\ne,\ne) -- (0.5*\ne,2*\ne) node[midway, left] {$n_\varepsilon$};

    \fill[black] (1.3*\ne, 1.4*\ne) circle (1.2pt);
    \fill[black] (0.23*\ne, 0.37*\ne) circle (1.2pt);
\fill[black] (0.67*\ne, 0.85*\ne) circle (1.2pt);
\fill[black] (0.45*\ne, 2.73*\ne) circle (1.2pt);
\fill[black] (0.89*\ne, 2.45*\ne) circle (1.2pt);
\fill[black] (1.52*\ne, 0.67*\ne) circle (1.2pt);
\fill[black] (1.78*\ne, 1.23*\ne) circle (1.2pt);
\fill[black] (1.34*\ne, 2.81*\ne) circle (1.2pt);
\fill[black] (1.96*\ne, 2.54*\ne) circle (1.2pt);
\fill[black] (3.12*\ne, 0.45*\ne) circle (1.2pt);
\fill[black] (2.56*\ne, 0.92*\ne) circle (1.2pt);
\fill[black] (3.27*\ne, 1.56*\ne) circle (1.2pt);
\fill[black] (3.78*\ne, 2.67*\ne) circle (1.2pt);
\fill[black] (0.34*\ne, 2.12*\ne) circle (1.2pt);
\fill[black] (2.67*\ne, 1.45*\ne) circle (1.2pt);
\fill[black] (3.89*\ne, 2.23*\ne) circle (1.2pt);

\end{tikzpicture}
        \caption{The set $U_0^{\epsilon,D}(F)$}
        \label{fig3.1}
    \end{figure}
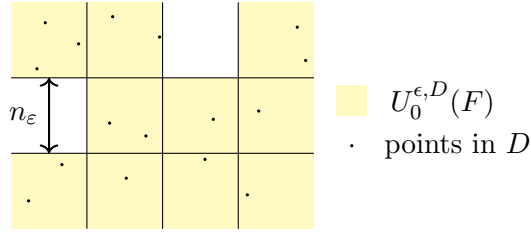
\end{definition}
    It is clear that $\phi_{F,\epsilon}\in \QIE_{\kappa,C+4\kappa n_{\epsilon}}\bigl(U_0^{\epsilon,D}(F),\X\bigr)$, where $U_0^{\epsilon,D}(F)$ is equipped with the induced Euclidean metric.

The purpose of this section is to prove that certain images lie in the neighborhood of the union of finitely many flats, which relies on the quasi-flats with holes theorem from~\cite{wortman2006quasiflats}. The statements are summarized in Proposition~\ref{t3.5}.

\begin{prop}\label{t3.5}   
Let $D$ be a random subset of $\X$. For each constant $\kappa>1$, there exist constants $M_1\in \N$ and $\epsilon_0\in (0,1/100)$ such that for each $\epsilon\in (0,\epsilon_0)$, each constant $C>0$, and each constant $\rho>1/\epsilon$, there exists a constant $\Delta>0$ such that the following hold for each flat $F$:
    \begin{enumerate}[label=\rm{(\roman*)}]
        \smallskip
        \item The sets $U_1^{\epsilon,D,\rho}(F)$ and $U_2^{\epsilon,D,\rho}(F)$ defined in \eqref{e3.0d} are almost surely nonempty.
        \smallskip
        \item Almost surely, for each $x\in F$,
        \begin{equation}\label{e3.3x}
        \liminf_{r\to+\infty} \Absbig{B(x,r)\cap U_2^{\epsilon,D,\rho}(F)} \big/ \abs{B(x,r)\cap F} \ge 1-\epsilon.
        \end{equation}

        \item For each $\phi\in \QIE_{\kappa,C}(D,\X)$, there exist $M_1$ flats $F_1,\,\dots,\,F_{M_1}$ such that 
        \begin{equation}\label{e3.3x1}
            \phi_{F,\epsilon}\bigl(U_1^{\epsilon,D,\rho}(F)\bigr)\subseteq B\biggl(\bigcup_{j=1}^{M_1}F_j,\Delta\biggr).
        \end{equation}

        \item For each $x\in U_2^{\epsilon,D,\rho}(F)$ and each $y\in F$, 
        \begin{equation}\label{e3.3y}
            d\bigl(y,U_1^{\epsilon,D,\rho}(F)\bigr)\le \max\{\rho,\,\epsilon d(x,y)\}.
        \end{equation}
        Consequently, for every $\phi\in \QIE_{\kappa,C}(D,\X)$ there exists a map 
       $\tphi_{F,\epsilon}\colon F\to\phi_{F,\epsilon}\bigl(U_1^{\epsilon,D,\rho}(F)\bigr)\subseteq \Nb\bigl(\bigcup_{j=1}^{M_1}F_j,\Delta\bigr)$
       that agrees with $\phi_{F,\epsilon}$ on $U_1^{\epsilon,D,\rho}(F)$ and is a $(2\kappa,6\kappa\epsilon, 6\kappa\rho+2C+8\kappa n_\epsilon )$-graded quasi-isometric embedding centered at $x$ for each $x\in U_2^{\epsilon,D,\rho}(F)$. 
    \end{enumerate}
\end{prop}

In the following, we will call $\tphi_{F,\epsilon}$ the \emph{regularized grid map} of $\phi$ on $F$.

Note that the grid map $\phi_{F,\epsilon}$ and the regularized grid map $\tphi_{F,\epsilon}$ are constructed by extending the restrictions of $\phi$ to $D\cap F$ and $U_1^{\epsilon,D,\rho}(F)$, respectively. At this point we do not need to assume the existence of a global embedding.

In the remainder of this article, for $m=1,\,2$, we often refer to $U_m^{\epsilon,D,\rho}(F)$ as $U_m(F)$ and $U_0^{\epsilon,D}(F)$ as $U_0(F)$ when the choices of $\epsilon$, $D$, and $\rho$ are clear.

We will use the following straightforward volume bounds in $\Z^2$. 
\begin{fact}\label{t3.0}
    Suppose $y\in \Z^2$. Then for $R>0$ and $r>1$,
    \begin{align}
         R^2\le \abs{B(y,R)}&\le (2R+1)^2 \quad \text{ and}\label{e3.0} \\
         \max\bigl\{3(r-1)^2,\,2r^2\bigr\}\le \abs{B(y,r)}&\le 4r^2+1.\label{e3.0z}
    \end{align}
\end{fact}

\begin{definition}\label{d3.2}
    For a set $W\subseteq \Z^2$, a constant $\alpha\in (0,1)$, and a point $z\in \Z^2$, define the \emph{lower $(1-\alpha)$-quasi-solid radius} of $W$ at $z$ by
    \begin{equation*}
        \cL_W(\alpha,z) \= \inf\{r\in \R:\abs{B(z,R)\cap W}> (1-\alpha )\abs{B(z,R)} \text{ for all } R>r\}.
    \end{equation*}
    Here we follow the convention that the infimum of an empty set is $+\infty$.
\end{definition}
 
 We first show the following deterministic lemma, which will be applied to estimate the sizes of holes in $U_m(F)$.
\begin{lemma}\label{t3.1}
     Let $\epsilon\in (0,1)$, $\alpha\in (0,\epsilon]$, and $\rho>1/\epsilon$ be constants and $W\subseteq \Z^2$. Put $W_1 \=  W_{(\epsilon,\rho)}$. Assume that $\cL_W\bigl(10^{-3}\epsilon^2 \alpha^2,z\bigr)$ is finite for each $z\in \Z^2$. Then for each $z\in \Z^2$, 
    \begin{align}
        \cL_{W_1}(\alpha,z)\le \cL_W\bigl(10^{-3}\epsilon^2 \alpha^2,z\bigr).  \label{e3.1}
    \end{align} 
\end{lemma}
\begin{proof}
    Given $\epsilon\in(0,1)$, $\alpha\in (0,\epsilon]$, and $\rho>1/\epsilon$, let $W\subseteq\Z^2$ be a subset such that $\cL_W\bigl(10^{-3}\epsilon^2 \alpha^2,z\bigr)$ is finite for every $z\in \Z^2$. Write $\cL \=  \cL_W$ throughout the proof. 
    
    We argue by contradiction and assume that there exists a point $z\in \Z^2$ such that $\cL_{W_1}(\alpha,z)> \cL\bigl(10^{-3}\epsilon^2 \alpha^2,z\bigr)$. Then by the definition of $\cL$, there exists a radius $r>\cL\bigl(10^{-3}\epsilon^2 \alpha^2,z\bigr)$ such that 
    \begin{equation}
        \abs{B(z,r)\cap W_1}\le (1-\alpha)\abs{B(z,r)}.\label{e3.1d}
    \end{equation}

    Write $V \= \Z^2\smallsetminus W$ and $V_1 \= \Z^2\smallsetminus W_1$. For each $x\in \Z^2$, by the left continuity of $\abs{B(x,\,\cdot\,)}$ and $\abs{B(x,\,\cdot\,)\cap W}$, we have 
    \begin{equation}
        \abs{B(x,\cL(\alpha,x))\cap V}\ge \alpha\abs{B(x,\cL(\alpha,x))}.\label{e3.1c}
    \end{equation}
    
    For each $x\in \Z^2$ and each $R>0$, define $\overline{B}(x,R)\coloneqq \bigl\{y\in \Z^2:d(x,y)\le R\bigr\}$. Write 
     \begin{equation*}
     H_1 \= \bigcup_{x\in \Z^2:\cL(\alpha,x)\ge \epsilon\rho}\overline{B}\bigl(x,\epsilon^{-1}\cL(\alpha,x)\bigr).    
     \end{equation*}
     By the definition of $\cL$, for every $R\ge r$, we have
    \begin{equation}
        \abs{B(z,R)\cap V}\le 10^{-3}\epsilon^2\alpha^2\abs{B(z,R)}.\label{e3.1b}
    \end{equation}
    
      By the definition of $W_1$ (cf.~Definition~\ref{d3.1}), for each $y\in V_1\smallsetminus V$, there exists a point $x\in \Z^2$ with $d(x,y)>\rho$ such that 
      \begin{equation*}
          \abs{B(x,\epsilon d(x,y))\cap W}< (1-\epsilon)\abs{B(x,\epsilon d(x,y))}.
      \end{equation*}So by the definition of $\cL$ we have $\cL(\alpha,x)\ge \epsilon d(x,y)> \epsilon\rho$ since $\alpha\le \epsilon$, and thus we have $y\in \overline{B}\bigl(x,\epsilon^{-1}\cL(\alpha,x)\bigr)$. This implies 
     \begin{equation}
         V_1\subseteq V\cup H_1.\label{e3.2a}
     \end{equation}

    Let $S\subseteq \Z^2$ be a set with minimum cardinality such that
    \begin{equation}\label{e3.2}
    B(z,r)\cap H_1\subseteq \bigcup_{x\in S:\cL(\alpha,x)\ge \epsilon\rho}\overline{B}\bigl(x,\epsilon^{-1}\cL(\alpha,x)\bigr).
    \end{equation} 
    Then $S$ is a finite set since $\abs{B(z,r)\cap H_1}$ is finite. For each $x\in S$, $\overline{B}\bigl(x,\epsilon^{-1}\cL(\alpha,x)\bigr)\cap B(z,r)\cap H_1\neq \emptyset$ by the minimality of $S$. Hence,
    \begin{align}
        d(x,z)&\le r+\epsilon^{-1}\cL(\alpha,x)\quad \text{ and }\label{e3.3}\\
        \cL(\alpha,x)&\ge \epsilon\rho> 1.\label{e3.3a}
    \end{align}
    We finish the proof by discussing the following two cases:
    
     \smallskip \emph{Case (i).} There exists a point $x\in S$ such that $\cL(\alpha,x)\ge \epsilon r$. By \eqref{e3.3} we have $d(x,z)\le 2\epsilon^{-1}\cL(\alpha,x)$. Then since $B(x,\cL(\alpha,x))\subseteq B(z,d(x,z)+\cL(\alpha,x))$, by \eqref{e3.1c} and \eqref{e3.0} we have
         \begin{equation}  \label{e3.4}
             \abs{B(z,d(x,z)+\cL(\alpha,x))\cap V}
             \ge \abs{B(x,\cL(\alpha,x))\cap V}  \ge  \alpha  \cL(\alpha,x)^2  .  
         \end{equation} 
         
         If $d(x,z)+\cL(\alpha,x)>r$, then by $d(x,z)\le 2\epsilon^{-1}\cL(\alpha,x)$, \eqref{e3.0}, and \eqref{e3.3a}, we have
         \begin{equation*}
             \alpha  \cL(\alpha,x)^2\ge 9^{-1}\alpha\epsilon^2(d(x,z)+\cL(\alpha,x))^2\ge 100^{-1}\alpha\epsilon^2\abs{B(z,d(x,z)+\cL(\alpha,x))},
         \end{equation*}
         which contradicts \eqref{e3.1b} with $R$ set to be $d(x,z)+\cL(\alpha,x)$ by \eqref{e3.4}. 
         
        Assume $d(x,z)+\cL(\alpha,x)\le r$. Then $r\ge \cL(\alpha,x)>1$ by \eqref{e3.3a}. Thus, by \eqref{e3.0z} and $\cL(\alpha,x)\ge \epsilon r$, we have $\alpha \cL(\alpha,x)^2\ge \epsilon^2\alpha r^2\ge 0.1 \epsilon^2\alpha \abs{B(z,r)}$. So by \eqref{e3.4},
         \begin{equation*}
             \abs{B(z,r)\cap V}\ge\abs{B(z,d(x,z)+\cL(\alpha,x))\cap V}\ge\alpha \cL(\alpha,x)^2\ge 0.1 \epsilon^2\alpha \abs{B(z,r)},
         \end{equation*}
         which again contradicts \eqref{e3.1b} with $R$ set to be $r$.
    
    \smallskip \emph{Case (ii).} For all $x\in S$, we have $\cL(\alpha,x)< \epsilon r$. Then
    \begin{equation}\label{e3.3b}
        d(x,z)< 2r\quad \text{ and } \quad B(x,\cL(\alpha,x))\subseteq B(z,3r)
    \end{equation}
    for all $x\in S$ by \eqref{e3.3}. By \eqref{e3.3a} we have $r>\cL(\alpha,x)>1$.
         
         List $S=\{ x_1,\,\dots,\,x_n\}$, where the subscripts are arranged so that $\{\cL(\alpha,x_i)\}_{i=1}^{n}$ is a nonincreasing sequence. We now recursively construct a subset $S'$ of $S$. 
         Define $S_0\coloneqq\emptyset$. At step $k\ge 0$, we choose the smallest integer $1\le i\le n$ such that 
         \begin{equation}\label{e3.3c}
         B(x_i,\cL(\alpha,x_i))\cap B(x_j,\cL(\alpha,x_j))=\emptyset
         \end{equation} 
         for all $x_j\in S_k$ and define $S_{k+1} \=  S_k\cup \{x_i\}$ if such $i$ exists. Set $S' \=  S_k$ if such $i$ does not exist.
         
         By the monotonicity of $\cL(\alpha,x_i)$ in $i$, if $1\le j<i\le n$ and $B(x_i,\cL(\alpha,x_i))\cap B(x_j,\cL(\alpha,x_j))\neq\emptyset$, then
         $\overline{B}\bigl(x_i,\epsilon^{-1}\cL(\alpha,x_i)\bigr)
         \subseteq  \overline{B}\bigl(x_j,\epsilon^{-1}\cL(\alpha,x_i)+d  (x_i,x_j)\bigr)
         \subseteq B\bigl(x_j,3\epsilon^{-1}\cL(\alpha,x_j)\bigr).$
         Thus, by the construction of $S'$,
         \begin{equation}\label{e3.5}
         \bigcup_{x\in S}\overline{B}\bigl(x,\epsilon^{-1}\cL(\alpha,x)\bigr)\subseteq\bigcup_{x\in S'}B\bigl(x,3\epsilon^{-1}\cL(\alpha,x)\bigr).
         \end{equation}
        To complete the proof, we verify the following inequality, which also contradicts \eqref{e3.1b}:
        \begin{equation}\label{e3.5a}
            \abs{B(z,3r)\cap V} \ge \sum_{x\in S'}\abs{B(x,\cL(\alpha,x))\cap V}>\frac{\epsilon^2\alpha}{50} \AbsBig{\bigcup_{x\in S}\overline{B}\bigl(x,\epsilon^{-1}\cL(\alpha,x)\bigr)}>\frac{\epsilon^2\alpha^2}{1000}\abs{B(z,3r)}.
        \end{equation}
             
        By \eqref{e3.3b} and \eqref{e3.3c}, $\abs{B(z,3r)\cap V}\ge \sum_{x\in S'}\abs{B(x,\cL(\alpha,x))\cap V}$. 
        
        By \eqref{e3.5}, \eqref{e3.3a}, \eqref{e3.0z}, and \eqref{e3.1c},
         \begin{align*}
             \AbsBig{\bigcup_{x\in S}\oB\bigl(x,\epsilon^{-1}\cL(\alpha,x)\bigr)}
             &\le \AbsBig{\bigcup_{x\in S'}B\bigl(x,3\epsilon^{-1}\cL(\alpha,x)\bigr)}
             \le 45\epsilon^{-2}\sum_{x\in S'}\abs{B(x,\cL(\alpha,x))}\\
             &<50\epsilon^{-2}\alpha^{-1}\sum_{x\in S'}\abs{B(x,\cL(\alpha,x))\cap V}.
         \end{align*}
         
         By \eqref{e3.2}, \eqref{e3.2a}, \eqref{e3.1d}, \eqref{e3.1b}, and \eqref{e3.0z},
         \begin{align*}
             \AbsBig{\bigcup_{x\in S}\oB\bigl(x,\epsilon^{-1}\cL(\alpha,x)\bigr)}
             &\ge \abs{B(z,r)\cap H_1}
             \ge \abs{B(z,r)\cap V_1}-\abs{B(z,r)\cap V}\\
             &\ge \alpha\abs{B(z,r)}-10^{-3}\epsilon^2 \alpha^2\abs{B(z,r)}
             > 20^{-1}\alpha\abs{B(z,3r)}.
         \end{align*}

        The proof of \eqref{e3.5a} follows immediately from the above inequalities.
        
    Therefore, we conclude the proof by contradiction.
\end{proof}

Using Lemma~\ref{t3.1}, we have the following:
\begin{lemma}\label{t3.3}
    Let $D$ be a random subset of $\X$. For each constant $\epsilon\in (0,1)$, each $\rho>1/\epsilon$, and each flat $F$, the following hold almost surely:
    \begin{enumerate}[label=\rm{(\roman*)}]
        \smallskip     	
        \item The sets $U_1(F)$ and $U_2(F)$ defined in \eqref{e3.0d} are nonempty.
        \smallskip     	
        \item $\cL_{U_2(F)}(\epsilon,z)\le \cL_{U_0(F)}\bigl(10^{-9}\epsilon^{10},z\bigr)<+\infty$ for every $z\in F$. 
    \end{enumerate}
\end{lemma}
\begin{proof}
     Let $\epsilon\in (0,1)$, $\rho>1/\epsilon$, $F$ be a flat, $z\in F$, and $D$ be a random subset of $\X$. For each $r>0$, denote $W_r \=  \bigl\{(i,j)\in \Z^2:A_{ij}^{\epsilon}\cap B(z,r)\neq \emptyset\bigr\}$, where $A_{ij}^{\epsilon}$ is the $n_\epsilon\times n_\epsilon$ square in $\Z^2$ (identified with $F$) defined in \eqref{e3.0b}. Note that   
     $\abs{B(z,r)\cap F}\le n_\epsilon^2\abs{W_r}\le \abs{B(z,r+2n_\epsilon)\cap F}$. By the volume bounds on $\Z^2$ as in \eqref{e3.0}, $\abs{W_r}\ge r^2/n^2_{\epsilon}$; and if $r>5n_{\epsilon}$, then $\abs{W_r}\le 10r^2/n^2_{\epsilon}$. 
     
     Denote by $K_r \=  \Absbig{\bigl\{(i,j)\in W_r:A_{ij}^{\epsilon}\cap D= \emptyset\bigr\}}$. Then $K_r$ is the sum of $\abs{W_r}$ i.i.d.\ Bernoulli$(p_1)$ random variables, where $p_1\=(1-p)^{n_{\epsilon}^2}$. By the Bernstein inequality for Bernoulli random variables, for $a>2p_1$ we have
\begin{equation}\label{e3.8a}
         \P(K_r\ge a\abs{W_r})
         \le \exp (- a\abs{W_r} / 8).
     \end{equation}
     Recall the definition of $U_0(F)$ in \eqref{e3.0c}. For $r>5n_\epsilon$, since $n_{\epsilon}^2\abs{W_r}<10\abs{B(z,r)\cap F}$, we have that if $K_r< a\abs{W_r}$ and $2p_1<a<1/10$, then
     \begin{equation}\label{e3.8b}
          \abs{B(z,r)\cap U_0(F)}\ge \abs{B(z,r)\cap F}-n_\epsilon^2K_r> (1-10a)\abs{B(z,r)\cap F}.
      \end{equation}
      Then for $R>5n_\epsilon$ and $2p_1<a<1/10$, with $M\=\bigl\{r\ge R : r^2\in \N\bigr\}$, by \eqref{e3.8a} and \eqref{e3.8b},
     \begin{align*}
         \P\bigl(\cL_{U_0(F)}(10a,z)\ge R\bigr)
         &\le \sum_{r\in M}\P(\abs{B(z,r)\cap U_0(F)}
         \le (1-10a)\abs{B(z,r)\cap F})
         \le \sum_{r\in M}\P\bigl(K_r\ge a\abs{W_r}\bigr)\\
         &\le \sum_{r\in M} e^{- a\abs{W_r}/8}
         \le \sum_{r\in M} e^{- ar^2  n_\epsilon^{-2}/8}
         \le \frac{e^{- a(R^2-1) n_\epsilon^{-2}/8}}{e^{a n_\epsilon^{-2}/8}-1}.
     \end{align*} 
      Thus, $\lim_{R\to+\infty}\P\bigl(\cL_{U_0(F)}(10a,z)\ge R\bigr)=0$. So almost surely $\cL_{U_0(F)}(10a,z)<+\infty$ for all $z\in F$.

     Applying Lemma~\ref{t3.1} twice with $(W,\alpha)$ taken to be $(U_1(F),\epsilon)$ and $\bigl(U_0(F),10^{-3}\epsilon^4\bigr)$, respectively, we have $\cL_{U_2(F)}(\epsilon,z)\le \cL_{U_1(F)}\bigl(10^{-3}\epsilon^4,z\bigr)\le \cL_{U_0(F)}\bigl(10^{-9}\epsilon^{10},z\bigr)$. To apply the conclusion in the previous paragraph, we take $a\=10^{-10}\epsilon^{10}$. The condition $a>2p_1$ is then equivalent to $10^{-10}\epsilon^{10}>2(1-p)^{n_{\epsilon}^2}$, thus we specify the choice of $n_\epsilon$ to be the smallest positive integer that satisfies $n_\epsilon^2\ge 11\log_{1/(1-p)}(10/\epsilon)$.
\end{proof}

 In the next lemma, we verify that our notion of quasi-centers (cf.~Definition~\ref{d3.1}) is consistent with a related notion in \cite{wortman2006quasiflats}.

\begin{lemma}\label{t3.7}
    Suppose $U$ is a subset of $\Z^2$ with the induced metric, $\epsilon\in (0,1/100)$, and $\rho>1/\epsilon$. Then for each $x\in U_{(\epsilon,\rho)}$ and each $r>\rho/\epsilon$, we have
    \begin{equation*}
         \abs{B(x,r)\cap U} / \abs{B(x,r)} \ge 1-100\epsilon.
    \end{equation*}
\end{lemma}
\begin{proof}   
    Assuming that $U_{(\epsilon,\rho)}\neq \emptyset$, let $r>\rho/\epsilon>10^4$ be a constant and $x\in U_{(\epsilon,\rho)}$ be a point. Write $l(y) \= \epsilon d(x,y)$ for $y \in \Z^2$. We construct a suitable set $S\subseteq B(x,r)$ such that $B(x,r)\subseteq \bigcup_{y\in S}B(y,l(y))$ to utilize the condition $x\in U_{(\epsilon,\rho)}$.
    
    Write $V \= \Z^2\smallsetminus U$. Choose a set $S\subseteq B(x,r)$ with the maximum cardinality such that $B(y_1,0.4 l( y_1))\cap B(y_2,0.4 l( y_2))=\emptyset$ for distinct points $y_1,\,y_2\in S$. This implies by \eqref{e3.0} that
    \begin{equation*}
        \sum_{y\in S}\abs{B(y,0.4 l( y))}\le \abs{B(x,(1+0.4\epsilon)r)}< 5r^2.
    \end{equation*}
    Thus by \eqref{e3.0} and $(2R+1)^2\le 6R^2+3$, we have 
    \begin{equation}\label{e3.6b}
        \sum_{y\in S}\abs{B(y, l(y))}\le 40\sum_{y\in S}\abs{B(y,0.4 l( y))}+3\abs{S}< 250r^2.
    \end{equation}
    Writing $S_> \= \{y\in S:d(x,y)> \rho\}$, an analogous argument with $\rho$ in place of $r$ shows that
    \begin{equation}\label{e3.6c}
        \sum_{y\in S\smallsetminus S_>}\abs{B(y, l( y))} < 40\abs{B(x,(1+0.4\epsilon)\rho)}+15\rho^2<500\rho^2<500\epsilon^2 r^2.
    \end{equation}
    
    Now we show that $B(x,r)\subseteq \bigcup_{y\in S} B(y, l( y))$. Let $y_1$ be an arbitrary point in $B(x,r)\smallsetminus S$. By the maximality of $S$, we have $B(y_1,0.4 l( y_1))\cap B(y_2,0.4 l( y_2))\neq \emptyset$ for some $y_2\in S$. This implies $d(y_1,y_2)\le 0.4   (l(y_1)+l(y_2))$.
    Since $d(y_1,y_2)\ge \abs{d(x,y_2)-d(x,y_1)}$ and $\epsilon\in (0,1/100)$, we can also deduce that $d(x,y_2)\ge 0.9 d(x,y_1)$. So $d(y_1,y_2)< l( y_2)$, which shows $y_1\in B(y_2,l( y_2))$. Thus,
    \begin{equation}\label{e3.6}
        B(x,r)\subseteq \bigcup_{y\in S} B(y, l( y)).
    \end{equation}
    
    Since $x\in U_{(\epsilon,\rho)}$ (cf.~Definition~\ref{d3.1}), for each $y\in S$ with $d(x,y)> \rho$, we have
    \begin{equation}\label{e3.6a}
        \abs{B(y, l( y))\cap V}\le \epsilon\abs{B(y, l( y))}.
    \end{equation}
    By \eqref{e3.6}, \eqref{e3.6a}, \eqref{e3.6b}, \eqref{e3.6c}, and \eqref{e3.0}, we have 
    \begin{align*}
        \abs{B(x,r)\cap V}
        &\le \sum_{y\in S}\abs{B(y,l( y))\cap V}
        = \sum _{y\in S_>}\abs{B(y,l( y))\cap V}+\sum_{y\in S \smallsetminus S_>}\abs{B(y,l( y))\cap V}\\
        &\le \epsilon\sum _{y\in S_>}\abs{B(y,l( y))}+\sum_{y\in S \smallsetminus S_>}\abs{B(y,l( y))}
        < 250 \epsilon r^2+500\epsilon^2 r^2
        < 100\epsilon\abs{B(x,r)}.
    \end{align*}
    This finishes our proof since $V=\Z^2\smallsetminus U$.
\end{proof}

 In the special case of quasiflats with holes in $\X$, \cite[Theorem~1.2]{wortman2006quasiflats} states: 
\begin{theorem}
\label{t3.4}
     For each constant $\kappa>1$, there exist constants $M_1\in \N$ and $\epsilon_1\in (0,1)$ such that for each $\epsilon\in (0,\epsilon_1)$, each $C>0$, and each $\rho>1$, there exists a constant $\Delta>0$ such that for every $U\subseteq \Z^2$ with $U_{(\epsilon,\rho)}\neq \emptyset$ and for every $\phi\in\QIE_{\kappa,C}(U,\X)$, $\phi\bigl(U_{(\epsilon,\rho)}\bigr)$ lies in the $\Delta$-neighborhood of the union of at most $M_1$ flats in $\X$.
\end{theorem}
    In the introduction of \cite{eskin1997quasi}, an example is given where the constant $M_1$ in Theorem~\ref{t3.4} is strictly greater than one.
\begin{proof}[\bf Proof of Proposition~\ref{t3.5}]
    Let $\kappa>1$ and $C>0$ be constants. Let $\epsilon_1\in(0,1)$ and $M_1\in \N$ be the constants given in Theorem~\ref{t3.4}. Define $\epsilon_0\=\min\bigl\{\epsilon_1,\,10^{-5}\bigr\}$ and fix arbitrary constants $\epsilon\in (0,\epsilon_0)$ and $\rho>1/\epsilon$. Now let $\Delta>0$ be the constant given in Theorem~\ref{t3.4} with $(\kappa,\epsilon,C,\rho)$ in Theorem~\ref{t3.4} set to be $(\kappa,\epsilon,C+4\kappa n_{\epsilon},\rho)$. Let $F$ be a flat, $D \subseteq\X$, and $\phi\in \QIE_{\kappa,C}(D,\X)$. 
    
    Parts (i) and (ii) are consequences of Lemma~\ref{t3.3}. Recall that $\phi_{F,\epsilon}\in \QIE_{\kappa,C+4\kappa n_{\epsilon}}(U_0(F),\X)$, so (iii) can be obtained from Theorem~\ref{t3.4} by our choice of constants. 
    
    For (iv), consider $x\in U_2(F)$. By Definition~\ref{d3.1}, for each $y\in F$,
    \begin{equation}\label{e3.7a}
        d(y,U_1(F))\le \max\{\rho,\,\epsilon d(x,y)\}.
    \end{equation}
    For each $y_1\in F$, define $\tphi_{F,\epsilon}(y_1)$ to be $\phi_{F,\epsilon}(y_2)$ for some $y_2$ in $U_1(F)$ closest to $y_1$. We now verify that $\tphi_{F,\epsilon}$ is indeed a desired graded quasi-isometric embedding in (iv). 
    
    Let $y_1,\,z_1$ be points in $F$ with $d(x,y_1)\ge d(x,z_1)$ and $y_2,\,z_2$ be the points in $U_1(F)$ closest to $y_1,\,z_1$ chosen above, respectively. By \eqref{e3.7a} and $d(x,y_1)\le d(x,z_1)+d(y_1,z_1)$, we have
    \begin{align}
        d(y_1,y_2)+d(z_1,z_2)&\le \max\{\rho,\,\epsilon d(x,y_1)\}+\max\{\rho,\,\epsilon d(x,z_1)\}\notag\\
        &\le 0.1d (y_1,z_1)+2\max\{\rho,\,\epsilon d(x,z_1)\}.\label{e3.7b}
    \end{align}
    Thus,
    $d(y_2,z_2)
         \le d(y_1,z_1)+d(y_1,y_2)+d(z_1,z_2) 
         \le 1.1d(y_1,z_1)+2\max\{\rho,\,\epsilon d(x,z_1)\}$.
    Since $\phi_{F,\epsilon}\in \QIE_{\kappa,C+4\kappa n_{\epsilon}}(U_0(F),\X)$, we have
    \begin{align}      
    d\bigl(\tphi_{F,\epsilon}(y_1),\tphi_{F,\epsilon}(z_1)\bigr)&=d(\phi_{F,\epsilon}(y_2),\phi_{F,\epsilon}(z_2))\le \kappa d(y_2,z_2)+C+4\kappa n_{\epsilon}\notag\\
    &\le 2\kappa d(y_1,z_1)+3\kappa\max\{\rho,\,\epsilon d(x,z_1)\}+C+4\kappa n_\epsilon\notag\\
    &\le 2\kappa d(y_1,z_1)+\max\{6\kappa\rho+2C+8\kappa n_\epsilon,\,6\kappa\epsilon d(x,z_1)\}.\label{e3.7c}
    \end{align}
    Here we use the inequality $\max\{a,\,b\}\le a+b \le \max\{2a,\,2b\}$ for $a,b\ge 0$. 
    
    For the remaining inequality, by \eqref{e3.7b} we similarly have 
    \begin{align*}
         d(y_2,z_2)
         \ge d(y_1,z_1)-d(y_1,y_2)-d(z_1,z_2)
         \ge 0.9d(y_1,z_1)-2\max\{\rho,\,\epsilon d(x,z_1)\}.
    \end{align*}
    Thus, 
    \begin{align}      
    d\bigl(\tphi_{F,\epsilon}(y_1),\tphi_{F,\epsilon}(z_1)\bigr)&=d(\phi_{F,\epsilon}(y_2),\phi_{F,\epsilon}(z_2))\ge \kappa^{-1} d(y_2,z_2)-C-4\kappa n_{\epsilon}\notag\\
    &\ge 0.5\kappa^{-1} d(y_1,z_1)-\max\{6\kappa\rho+2C+8\kappa n_\epsilon,\,6\kappa\epsilon d(x,z_1)\}.\label{e3.7d}
    \end{align}
    Since the choices of $y_1$ and $z_1$ are arbitrary, we obtain that $\tphi_{F,\epsilon}$ is a $(2\kappa,6\kappa\epsilon, 6\kappa\rho+2C+8\kappa n_\epsilon )$-graded quasi-isometric embedding centered at $x$ from \eqref{e3.7c} and \eqref{e3.7d}.     
    \end{proof}

\section{Hyperplanes to hyperplanes}      \label{s:Step2}

The aim of this section is to prove that each hyperplane is almost surely mapped near another hyperplane, with uniform control of constants, as stated in Proposition~\ref{t4.5}. This statement is crucial for showing that a boundary map induced by $\phi$ is well defined and for verifying properties of the boundary map, see Section~\ref{s5}. 

The general strategy of this section follows \cite{eskin1998quasi}. We first utilize Proposition~\ref{t3.5} to show that each flat is almost surely mapped near a union of polyhedra. This result implies that each hyperplane is almost surely mapped near another hyperplane, but without uniform control of constants. We then verify that each flat is almost surely mapped near another flat, which implies the uniform control of constants for hyperplanes. Throughout this section, we use several known results on quasi-isometric embeddings between Euclidean spaces.

\subsection{Transverse flats}\label{s4.1}

Note that a linear map from $\Z^2$ into itself can be a quasi-isometry without the property that each hyperplane is mapped close to another hyperplane. This phenomenon suggests that one needs to treat hyperplanes as the intersection of two transverse flats. We first make some choices for such pairs of transverse flats. 

\begin{notation}[Choice of transverse flats]\label{d4.1}    
    
    We call two flats {\it transverse} if their intersection is a hyperplane. For each flat $F$ in $\X$ and each point $x$ in $F$, we now fix a choice of two flats transverse to $F$. Let $F=L_1\times L_2$ be a flat in $\X$ with $L_i\subseteq T_i$ for $i\in\{1,\,2\}$ and $x=(x_1,x_2)$ be a point in $F$. For such a pair $F$ and $x$, we fix a choice of a bi-infinite geodesic $L_2'$ in $T_2$ passing through $x_2$ such that $L_2\cap  L'_2=\{x_2\}$. Then we call the flat $F' \=  L_1\times L_2'$ the \emph{transverse flat for $F$ over $T_2$ at $x$}. The choice, for the pair $F$ and $x$, of the \emph{transverse flat for $F$ over $T_1$ at $x$} is made analogously for a fixed choice of $L'_1$. Denote the set of these two transverse flats of $F=L_1\times L_2$ at $x$ by $\Trans_F(x)$, which consists of two flats $L_1\times L_2'$ and $L_1'\times L_2$.     
    
\end{notation}\label{d4.1z}

    Suppose $D$ is a subset of $\X$, $\epsilon>0$, and $\rho>1/\epsilon$. We consider flats where $D\cap F$ satisfies certain regularity, namely the flats where the relation~\eqref{e3.3x} in Proposition~\ref{t3.5} holds. 
    \begin{definition}\label{d4.1b}
    Define $\mathcal{F}_{\epsilon,D,\rho}$ to be the set of all flats $F$ with the property that for each $x\in F$, 
    \begin{equation} \label{e4.1b}
        \liminf_{r\to+\infty} \Absbig{B(x,r)\cap U_2^{\epsilon,D,\rho}(F)} \big/ \abs{B(x,r)\cap F} \ge 1-\epsilon.
    \end{equation}
    \end{definition}

    Here $U_2^{\epsilon,D,\rho}(F)$ is defined in \eqref{e3.0d}. We now consider a subset of $U_2^{{\epsilon,D,\rho}}(F)$ that consists of the points in $F$ from which the regularity conditions are verified on $F$ and on the chosen transverse flats. 
    
    \begin{definition}\label{d4.1a}
     Suppose $D\subseteq \X$. Let $\epsilon\in (0,1)$ and $\rho>1/\epsilon$ be constants. For $F\in \mathcal{F}_{\epsilon,D,\rho}$, define
    \begin{equation*}
    U_3^{\epsilon,D,\rho}(F) \= 
    \{x\in U_2(F) : F'\in \mathcal{F}_{\epsilon,D,\rho}\text{ and } x\in U_2(F')\text{ for both } F'\text{ in } \Trans_F(x) \}.
    \end{equation*}
\end{definition}

In the remainder of this article, we often refer to $U_3^{\epsilon,D,\rho}(F)$ as $U_3(F)$ when the choices of $\epsilon$, $D$, and $\rho$ are clear.

\begin{rem} \label{r:U3_monotone}
It follows from the above definitions, Definition~\ref{d3.1}, \eqref{e3.0d}, and Proposition~\ref{t3.5} that $U_i^{\epsilon,D,\rho }(F)$ is increasing with respect to $\epsilon\in(0,1)$, $D$, and $\rho>1/\epsilon$ for each $i\in\{1,\,2,\,3\}$.
\end{rem}

We first prove that the set $U_3(F)$ has a large density.
\begin{lemma}\label{t4.1}
   Suppose $F$ is a flat, $\epsilon\in (0,1/100)$, and $D$ is a random subset of $\X$. Then almost surely for each $\rho>1/\epsilon$, we have $F\in \cF_{\epsilon,D,\rho}$, and for each $x\in F$,
    \begin{equation}\label{e4.0}
        \liminf_{r\to+\infty} \abs{B(x,r)\cap U_3(F)}  / \abs{B(x,r)\cap F} \ge 1-3\epsilon.
    \end{equation}
\end{lemma}
    
\begin{proof}
    Since $\cF_{\epsilon,D,\rho}$ and $U_3(F)$ are monotone increasing in $\rho$, it suffices to verify our statement for a fixed $\rho>1/\epsilon$. 
    
    Proposition~\ref{t3.5}~(ii) immediately implies $F\in \cF_{\epsilon,D,\rho}$ almost surely. For each $x\in F$ and each $i\in\{1,\,2\}$, let $F_i(x)\in \Trans_F(x)$ be the chosen transverse flat such that $F_i(x)\cap F$ is a hyperplane in $T_i$ (cf.~Definition~\ref{d2.9}). Write $W \=  \{x\in F: x\notin U_2(F_1(x))\}$ and $W'\=\{x\in F: x\notin U_2(F_2(x))\}$. Then by definition $U_3(F)=U_2(F)\smallsetminus (W\cup W')$ almost surely. Thus, by symmetry and \eqref{e4.1b} it suffices to prove that for each $x$ in $F$, $\limsup_{r\to+\infty}\frac{\abs{B(x,r)\cap W}}{\abs{B(x,r)\cap F}}\le \epsilon$ almost surely.
    
    Fix $x\in F$. By \eqref{e3.0d} and Definition~\ref{d3.1}, $\bigl\{E\subseteq\X:x\in U_2^{\epsilon,E,\rho}(F)\bigr\}$ is measurable. (Note that there are only countably many balls in $\X$.) Write $p_0\=\P(x\in U_2(F))$. By the translation invariance of $\P_p$, $p_0$ is independent of the choices of $F$ and $x$. By Proposition~\ref{t3.5}~(ii), 
    \begin{equation}\label{e4.1}
        1-\epsilon\le \liminf_{r\to+\infty} \E[\abs{B(x,r)\cap U_2(F)}] / \abs{B(x,r)\cap F} =p_0.
    \end{equation}
    In the remaining part of the proof, we parametrize the flat $F$ by $\Z^2$ and identify $x$ with the origin. Suppose $r>10$ is an integer. Write $B_r \=  B(x,r)\cap F$. For each $n\in \Z$, define 
    \begin{equation*}
    K_n \=  W\cap \{(y_1,y_2)\in F:y_2=n\} \quad \text{ and } \quad
    H_n(r) \=  \abs{K_n\cap B_r}.
    \end{equation*} 
    By our definition of transverse flats, $F_1(y_1,y_2)\cap F_1(y_1',y_2')=\emptyset$ for each pair of $(y_1,y_2)$, $(y_1',y_2')\in F$ with $y_2\neq y_2'$. Thus for $n\in\Z$, the random variables $H_n(r)$ are independent.

   Since $0\le H_n(r)\le 2r+1$ and $r>10$, for each $n\in \Z$ we have 
   \begin{equation*}
   \E\bigl[(H_n(r)-\E[H_n(r)])^2\bigr]\le 5r^2\quad \text{ and } \quad
   \E\bigl[(H_n(r)-\E[H_n(r)])^4\bigr]\le 20r^4. 
   \end{equation*}
   Noting that $\abs{B_r\cap W}=\sum_{n=-r+1}^{r-1}H_n(r)$, by the above inequalities we have 
   \begin{align}\label{e4.0b}
       &\E\bigl[(\abs{B_r\cap W}-\E[\abs{B_r\cap W}])^4\bigr]\\
       &\qquad=\sum_{n=-r+1}^{r-1}\E\bigl[(H_n(r)-\E[H_n(r)])^4\bigr] \notag \\
       &\qquad\qquad +6\sum_{n< m}\E\bigl[(H_n(r)-\E[H_n(r)])^2\bigr] \cdot \E\bigl[(H_m(r)-\E[H_m(r)])^2\bigr]
       \le 40r^5+600r^6 
       \le 700r^6. \notag
   \end{align}
   On the other hand, $\E[\abs{B_r\cap W}]=(1-p_0)\abs{B_r}$. So for each $\epsilon'> 1-p_0$, by \eqref{e4.0b} and \eqref{e3.0},
   \begin{equation*}
       \P(\abs{B_r\cap W}\ge \epsilon' \abs{B_r})\le 700r^6(\epsilon'+p_0-1)^{-4}\abs{B_r}^{-4}\le 700(\epsilon'+p_0-1)^{-4}r^{-2}.
   \end{equation*}
  Thus, by the Borel--Cantelli lemma, $
    \limsup_{r\in \Z,\, r\to+\infty}\frac{\abs{B(x,r)\cap W}}{\abs{B(x,r)\cap F}}\le \epsilon'$ almost surely. Since \linebreak$\limsup_{r\to+\infty}\frac{\abs{B(x,r+1)\cap F}}{\abs{B(x,r)\cap F}}=1$, this implies $\limsup_{t\in \R,\,t\to+\infty}\frac{\abs{B(x,t)\cap W}}{\abs{B(x,t)\cap F}}\le \epsilon'$ almost surely.
    This finishes our proof by \eqref{e4.1}.
    \end{proof}
    For each $D\subseteq \X$, define $\cF_{\epsilon,D,\rho}'\subseteq \cF_{\epsilon,D,\rho}$ to be the set of flats $F$ in $\cF_{\epsilon,D,\rho}$ with the property that \eqref{e4.0} holds for all $x\in F$. In the remainder of this section, we consider a flat $F\in \cF_{\epsilon,D,\rho}'$ and a point $x\in U_3(F)$. Lemma~\ref{t4.1} provides a justification for this choice. 
    
   We clarify the notation of singular geodesic rays.
   \begin{notation}
       Consider the infinite geodesic rays that are contained in hyperplanes of flats: these are rays of the form $L_1\times \{x_2\}$ or $\{x_1\}\times L_2$, where $L_i$ is a geodesic ray in $T_i$ for each $i\in \{1,\,2\}$. We refer to them as \emph{singular geodesic rays}, as they are contained in intersections of flats. We say that $L_1\times \{x_2\}$ (resp.\ $\{x_1\}\times L_2$) is a singular geodesic ray in $T_1$ (resp.\ $T_2$). 
   \end{notation}
   In this terminology, for each hyperplane $L\subseteq \X$ and each point $x\in L$, $L$ is equal to the union of two singular geodesic rays emanating from $x$.
   Now we are ready to state the main result of this section. 
   
   \begin{prop}[Hyperplanes to hyperplanes]\label{t4.5}
    Given $\kappa>1$, there exist constants $\epsilon'\in(0,1/100)$ and $\lambda_1>1$ such that if $C>0$, $\epsilon\in (0,\epsilon')$, and $\rho>1/\epsilon$, then there exists a constant $R_0>0$ such that for each $D\subseteq \X$, each $\phi\in \QIE_{\kappa,C}(D,\X)$, each $F\in \cF_{\epsilon,D,\rho}'$, each $x\in U_3(F)$, and each singular geodesic ray $L_1$ in $F$ emanating from $x$, there exists a unique singular geodesic ray $L_1'$ emanating from $\phi_{F,\epsilon}(x)=\tphi_{F,\epsilon}(x)$
    such that 
    \begin{equation}\label{e4.5t2}
    \tphi_{F,\epsilon}(L_1)\sim_{\lambda_1\epsilon,\phi_{F,\epsilon}(x),R_0}L_1'.
    \end{equation}
    Here $\sim_\epsilon$ is the $\epsilon$-equivalence defined in Definition~\ref{d2.5}. Furthermore, suppose $L\subseteq F$ is a hyperplane passing through $x$ and $L=L_1\cup L_2$, where $L_i$ is a singular geodesic ray emanating from $x$ for each $i\in\{1,\,2\}$. Let $L_i'$, $i\in\{1,\,2\}$, be the singular geodesic ray emanating from $\phi_{F,\epsilon}(x)$ given in \eqref{e4.5t2}. Then there exists some $j\in \{1,\,2\}$ such that $L_1'$ and $L_2'$ are both singular geodesic rays in $T_j$, and there exists a unique hyperplane $L'\subseteq L_1'\cup L_2'$ such that 
    \begin{equation}\label{e4.5t1}
    \tphi_{F,\epsilon}(L)\sim_{\lambda_1\epsilon,\phi_{F,\epsilon}(x),R_0}L'.
    \end{equation}
    Moreover, $\epsilon'$ and $\lambda_1$ can be chosen so that the choices of $L'$ and $L_1'$ are independent of the choice of $\rho$.
    Here $\tphi_{F,\epsilon}$ is the regularized grid map defined in Proposition~\ref{t3.5}.
\end{prop}    

The proof of Proposition~\ref{t4.5}, which involves a decent part of geometric considerations along with constant estimates, occupies the remainder of this section. Here we divide the proof into several geometric lemmas in a manner analogous to that of \cite{eskin1998quasi}. Although the proofs of these lemmas are complicated and lengthy, they are based on geometric ideas that are common in the study of QI embeddings and follow a similar pattern. 

In these proofs, we often consider the composition of a nearest-point projection to a nearby flat with the regularized grid map $\tphi_{F,\epsilon}$. By earlier results in this article (e.g., Proposition~\ref{t3.5}), such composition is close to $\tphi_{F,\epsilon}$ and thus also a QI embedding. Then we apply previously known properties of QI embeddings between Euclidean spaces (e.g., Lemma~\ref{t4.0} and Corollary~\ref{t4.0a}) to derive the desired result. 

Throughout these proofs, we also make use of other geometric properties without specifying them. For instance, since QI embeddings are bi-Lipchitz on large scales, additive constants are sometimes safely ignored. Other useful properties include certain distance estimates in $\X=T_1\times T_2$ (e.g., Lemma~\ref{t4.3}) and elementary estimates in $\Z^2$.

\subsection{Flats to finite unions of polyhedra}
    In this subsection, we verify that flats are mapped near finite unions of polyhedra with a detailed proof. The result is stated in Lemma~\ref{t4.2}, which is analogous to \cite[Lemma~3.6]{eskin1998quasi}. We first give a definition used in the statement of Lemma~\ref{t4.2}.
    \begin{definition}\label{d4.2}
        Let $n\in \N$, $x$ be a point in $\X$, and $F_1,\,\dots,\, F_n$ be distinct flats in $\X$. We define a set of disjoint polyhedra $\{P_1,\,\dots,\,P_{n'}\}$ in $\X$ such that $\bigcup_{ \ell=1 }^{n'} P_\ell=\bigcup_{i=1}^{n} F_i$, $P_\ell$ contains no hyperplane, and there exists a flat containing both $x$ and $P_\ell$ for each $1\le \ell\le n'$.

        For each pair of $1\le i<j \le n$, $F_i\cap F_j$ is a polyhedron if it is not empty. For each index $i\in \{1,\,\dots,\, n\}$, we view the flat $F_{i}$ as parametrized by $\Z^2$ as a subset of $\R^2$ and assume that the nearest-point projection of $x$ on $F_i$ is parametrized as $(0,0)$. For each $a\in \R$, denote by $L_v(a)\=\{(a,t):t\in \R\}$ and $L_h(a)\=\{(t,a):t\in \R\}$ the vertical and horizontal lines in $\R^2$, respectively. Let $K_v\=\{L_v(a):a-1/2\in \Z\}$ and $K_h\=\{L_h(a):a-1/2\in \Z\}$. Then, by the definition of polyhedra, for each index $j\in\{1,\,\dots,\, n\}$ with $j\neq i$, $F_i\cap F_j$ is bounded by at most $2$ lines in $K_v$ and $2$ lines in $K_h$ as a subset of $F_i$. All such lines for all possible choices of $j$ together with the lines $L_v(1/2)$ and $L_h(1/2)$ subdivide $F_{i}$ into at most $4n^2$ disjoint polyhedra.  Repeating this construction for all $i\in \{1,\,\dots,\,n\}$, we get a set of distinct nonempty polyhedra $\{P_1,\,\dots,\,P_{n'}\}$ with $n'\le 4n^3$ such that 
    \begin{equation*}
    \bigcup_{ \ell=1 }^{n'} P_\ell=\bigcup_{i=1}^{n} F_i.
    \end{equation*} 
    It is straightforward to verify that any two distinct polyhedra in the above set are disjoint and it satisfies our conditions given above.
    
    Note that the choice of $\{P_1,\,\dots,\,P_{n'}\}$ depends on the parametrization of the flats $F_i$. We call such a set $\{P_1,\,\dots,\,P_{n'}\}$ a \emph{regular partition} of $\{F_1,\,\dots,\, F_n\}$ with respect to $x$.
    \end{definition}
    
\begin{lemma}[Flats to finite unions of polyhedra]\label{t4.2}
     Given $\kappa>1$, there exist constants $\lambda_3>0$ and $\epsilon_3\in(0,1/100)$ such that if $C>0$, $\epsilon\in (0,\epsilon_3)$, and $\rho>1/\epsilon$, then there exists a constant $R_3>0$ such that if $D\subseteq \X$, $\phi\in \QIE_{\kappa,C}(D,\X)$, $F\in \cF_{\epsilon,D,\rho}$, $x\in U_2(F)$, $R>R_3$ is a real number, and $\bigl\{P_1',\,\dots,\,P_{m'}'\bigr\}$ is a regular partition of the set of flats $\{F_1,\,\dots,\,F_{M_1}\}$ given in Proposition~\ref{t3.5} with respect to $\phi_{F,\epsilon}(x)$, then there exist a subset $\sigma_R\subseteq \{1,\,\dots,\,m'\}$ and a collection of polyhedra $\{P_i:i\in\sigma_R\}$ such that
    \begin{enumerate}[label=\rm{(\roman*)}]
       \smallskip
       \item $\Int_{F_i'}\bigl(P_i',\lambda_3\epsilon R\bigr)\subseteq P_i\subseteq P_i'
        \subseteq B(P_i,2\lambda_3\epsilon R)$ for some flat $F_i'$ containing both $P_i' $ and $\phi_{F,\epsilon}(x)$.
       \smallskip
       \item $\tphi_{F,\epsilon}(F)\cap B(\phi_{F,\epsilon}(x),R) 
         \subseteq B\bigl(\bigcup_{i\in\sigma_R}P_i,\lambda_3\epsilon R\bigr)$.
        \smallskip
       \item $\bigcup_{i\in\sigma_R}P_i\cap B(\phi_{F,\epsilon}(x),R)  
        \subseteq B\bigl(\tphi_{F,\epsilon}(F),\lambda_3\epsilon R\bigr)$.
        \smallskip
       \item There exists a constant $R_1''\ge R_3$ such that the set $\theta_R\=\{i\in \sigma_R: P_i'\ \text{is a chamber}\}$ is independent of the choice of $R\ge R_1''$.
    \end{enumerate}
\end{lemma}
\begin{proof}
    Let $\kappa>1$ and $C>0$ be constants, $D$ be a subset of $\X$, and $\phi\in \QIE_{\kappa,C}(D,\X)$. Let $M_1\in \N$ and $\epsilon_0\in(0,1/100)$ be the constants depending only on $\kappa$ given in Proposition~\ref{t3.5}. 

    Let $\epsilon\in(0,\epsilon_0)$ be a constant whose exact choice is determined later and $\rho>1/\epsilon$. Consider a flat $F\in \cF_{\epsilon,D,\rho}$ and a point $x\in U_2(F)$. For $R>0$, let $B_R$ denote the ball $B(\phi_{F,\epsilon}(x),R)$ in $\X$. Let $F_1,\,\dots,\,F_{M_1}$ be the $M_1$ flats given in Proposition~\ref{t3.5} such that 
      \begin{equation}\label{e4.2g}
          \tphi_{F,\epsilon}(F)
          =\phi_{F,\epsilon}(U_1(F))\subseteq\Nb \biggl(\bigcup_{i=1}^{M_1}F_i,\Delta\biggr),
      \end{equation}
    where $\Delta$ is the constant and $\tphi_{F,\epsilon}$ is the regularized grid map given in Proposition~\ref{t3.5}. Let $\{P_{\ell}'\}_{1\le \ell\le m'}$ be a regular partition (cf.~Definition~\ref{d4.2}) of $\{F_1,\,\dots,\,F_{M_1}\}$ with respect to $\phi_{F,\epsilon}(x)$ with
    \begin{align}
        m'&\le 4M_1^3 \quad\text{ and}\label{e4.2b} \\
    \bigcup_{ \ell=1 }^{m'} P_\ell'&=\bigcup_{i=1}^{M_1} F_i.\label{e4.2}
    \end{align}

    \textbf{Step 1.} We first verify that $\tphi_{F,\epsilon}^{-1}\bigl(\tphi_{F,\epsilon}(F)\cap B_{R+\Delta}\bigr)\subseteq B(x,3\kappa R-2)$ for sufficiently large $R$. Thus, the regularized grid map $\tphi_{F,\epsilon}$ restricted on $\tphi_{F,\epsilon}^{-1}\bigl(\tphi_{F,\epsilon}(F)\cap B_{R+\Delta}\bigr)$ is a quasi-isometric embedding to $B\bigl(\bigcup_{ \ell=1}^{m'} P_\ell',\Delta\bigr)$.
    
    Define, for $R>0$,
    \begin{equation}\label{e4.2b2}
        C_\epsilon\coloneqq 2\kappa C+12\kappa\rho+4C+16\kappa n_{\epsilon}+2\Delta \quad \text{ and } \quad
        C'(R) \=  3\bigl(6\kappa^2\epsilon R+C_\epsilon\bigr),
    \end{equation}    
    where $n_\epsilon>0$ is defined in \eqref{e3.0a}.
    Then for $R>2C_{\epsilon}\big/\bigl(\kappa^2\epsilon\bigr)$, 
    we have  
    \begin{equation}\label{e4.2b1}
        C'(R)<20\kappa^2\epsilon R.
    \end{equation}
    If in addition $\epsilon\in\bigl(0,1\big/\bigl(200\kappa^2\bigr)\bigr)$, then $C'(R)<0.1R$. For such $R$ and $\epsilon$, if $y\in \X$ is a point with $d(x,y)\ge 3\kappa R-2$, then 
    \begin{equation}\label{e4.2b3}
        d\bigl(\tphi_{F,\epsilon}(x),\tphi_{F,\epsilon}(y)\bigr)
        \ge d(x,y)/(2\kappa)-6\kappa \epsilon d(x,y)-C_\epsilon
        > R+\Delta
    \end{equation}
    since $\tphi_{F,\epsilon}$ is a $(2\kappa,6\kappa\epsilon, 6\kappa\rho+2C+8\kappa n_\epsilon )$-graded quasi-isometric embedding. Then 
    \begin{equation}\label{e4.2b4}
        \tphi_{F,\epsilon}^{-1}\bigl(\tphi_{F,\epsilon}(F)\cap B_{R+\Delta}\bigr)\subseteq B(x,3\kappa R-2).
    \end{equation}
    So since graded quasi-isometric embeddings are quasi-isometric embeddings on bounded sets, we get from direct calculation that $\tphi_{F,\epsilon}$ restricted to $\tphi_{F,\epsilon}^{-1}\bigl(\tphi_{F,\epsilon}(F)\cap B_{R+\Delta}\bigr)$ is a $(2\kappa,C'(R)-4\Delta)$-quasi-isometric embedding to $B\bigl(\bigcup_{ \ell=1}^{m'} P_\ell',\Delta\bigr)$ by Proposition~\ref{t3.5} and \eqref{e4.2}.
    
    \smallskip
    
    \textbf{Step 2.} We now choose a suitable set $\sigma$ of indices $i$. In the following steps, we show that $P_i'\cap B_R$ is close enough to $\tphi_{F,\epsilon}(F)\cap B_R$ for each $i\in \sigma$. 
    
    Let $R_0>C_{\epsilon}\big/\bigl(\kappa^2\epsilon\bigr)$, $r_0>0$, $R_1\in (R_0,2R_0)$, and $r_1\in (r_0,2r_0)$ be constants whose exact choices are specified later. Let $F_\ell'$ be a flat containing both $\phi_{F,\epsilon}(x)$ and $P_\ell'$ for each $1\le \ell\le m'$. In the remainder of this proof, we denote $\Int_{F_l'}(U,\,\cdot\,)$ by $\Int(U,\,\cdot\,)$ when $U\subseteq F_l'$.
    
    For each index $1\le i\le m'$, consider the pair of sets $V_i' \=  P_i'\cap B_{R_1}$ and $V_i\= \Int(P_i',r_1)\cap B_{R_0}$, where $P_i'$ is always viewed as a subset of the flat $F_i'$ equipped with the Euclidean metric when we consider $\Int(P_i',\,\cdot\,)$ throughout this proof (cf.~Definition~\ref{d2.3}). Note that if $\Int(P_i',r_1)\neq \emptyset$, $\Int(P_i',r_1)$ is also a polyhedron and $P_i'\subseteq B(\Int(P_i',r_1),2r_1)$. A sketch displaying the relations between these sets is given below.
    \begin{figure}
        \centering
        \begin{tikzpicture}[scale=0.8]

    \def\n{1}      
    \def\r{0.35}    
    \def\R{5}

    \draw[->, thick] (-0.5,0) -- (7,0) ;
    \draw[->, thick] (0,-0.5) -- (0,7) ;

    \draw[red, thick, dashed] (0:\R) arc (0:90:\R);

    \draw[blue, thick, dashed] (0:\R+2*\r) arc (0:90:\R+2*\r);

    \fill[green!20, opacity=0.5] (\n,\n) -- (7,\n) -- (7,7) -- (\n,7) -- cycle;
    \draw[green, thick, ->] (\n,\n) -- (7,\n); 
    \draw[green, thick, ->] (\n,\n) -- (\n,7); 
    \node[green] at (0.5,6.2) {$P_i'$};

    \fill[orange!30, opacity=0.7] (\n+\r,\n+\r) -- (7,\n+\r) -- (7,7) -- (\n+\r,7) -- cycle;
    \draw[orange, thick, ->] (\n+\r,\n+\r) -- (7,\n+\r); 
    \draw[orange, thick, ->] (\n+\r,\n+\r) -- (\n+\r,7); 
    \node[orange] at (6.6,1.8) {$\Int(P_i',r_1)$};

    \begin{scope}
        \clip (0,0) -- (0:\R+2*\r) arc (0:90:\R+2*\r) -- cycle;
        \fill[blue!20, opacity=0.4] (\n,\n) -- (7,\n) -- (7,7) -- (\n,7) -- cycle;
    \end{scope}
    \node[blue] at (3,4.5) {$V_i'$};
    \node[blue] at (-0.5,5.6) {$R_1$};

    \begin{scope}
        \clip (0,0) -- (0:\R) arc (0:90:\R) -- cycle;
        \fill[red!20, opacity=0.6] (\n+\r,\n+\r) -- (7,\n+\r) -- (7,7) -- (\n+\r,7) -- cycle;
    \end{scope}
    \node[red] at (2.5,3) {$V_i$};
    \node[red] at (-0.5,5) {$R_0$};

    \node at (-0.8,-0.3) {$\phi_{F,\epsilon}(x)$};

\end{tikzpicture}
        \caption{The sets $P_i'$, $\Int(P_i',r_1)$, $V_i'$, and $V_i$}
        \label{fig4.1}
    \end{figure}
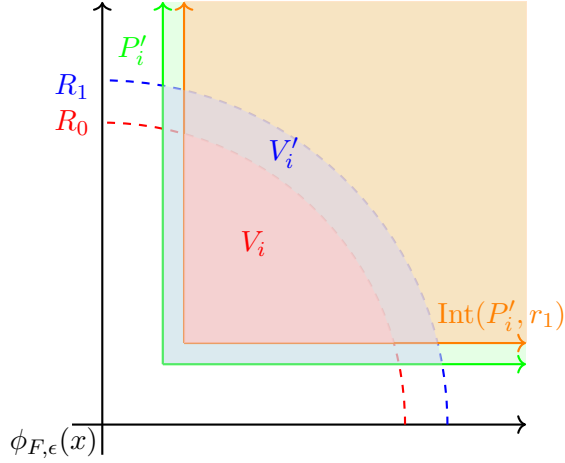
    For each index $1\le i\le m'$, define $\phi_i'\coloneqq \pi_{V_i'}\circ \tphi_{F,\epsilon}$, where $\pi_{V_i'}$ denotes a projection sending each point in $\X$ to its nearest point in $V_i'$. Define 
    \begin{equation}\label{e4.2a1}
        W_i' \= \bigl\{y\in F:d\bigl(\phi_i'(y),\tphi_{F,\epsilon}(y)\bigr)\le \Delta\bigr\}.
    \end{equation}
    Then by \eqref{e4.2b4}, $W_i'\subseteq B(x,3\kappa R_1-2)$ and thus $\phi_i'$ is a $(2\kappa,C'(R_1))$-quasi-isometric embedding on $W_i'$.     
     Consider the set of indices
     \begin{equation}\label{e4.2a}
         \sigma \=  \{i\in \{ 1,\dots,\,m'\}:\,\phi_i'(W_i')\cap \Int(V_i',r_1)\cap B_{R_0}\neq \emptyset\}.
     \end{equation}

     \textbf{Step 3.} We now apply Corollary~\ref{t4.0a} to show that $\bigcup_{j\in \sigma}V_j$ is contained in some neighborhood of $\tphi_{F,\epsilon}(F)$ for suitable choices of $R_0$, $r_0$, $R_1$, and $r_1$. 
     
     We first claim that we may choose $r_1$ in terms of $R_1$ so that 
     \begin{equation}\label{e4.2a2}
     \Int\bigl(V_j',r_1\bigr)\subseteq \Nb \bigl(\phi_j'\bigl(W_j'\bigr),\eta C'(R_1)\bigr)
     \end{equation}
     for each $j\in \sigma$, where $\eta>2\kappa$ is the constant given in Lemma~\ref{t4.0} with $(n,\kappa)$ set to be $(2,2\kappa)$. Suppose $j\in \sigma$ and $y_1$ is a point in $W_j'$ such that $\phi_j'(y_1)\in \Int\bigl(V_j',r_1-\eta C'(R_1)\bigr)$. By \eqref{e4.2g} and \eqref{e4.2}, for each $y_2\in (B(x,3\kappa R_1)\cap F)\smallsetminus W_j'$, there exists some $1\le j_1\le m'$ such that $d\bigl(\tphi_{F,\epsilon}(y_2),P_{j_1}'\bigr)\le \Delta$. Since $y_1\in W_j'$, we have $d\bigl(\tphi_{F,\epsilon}(y_1),\Int\bigl(V_j',r_1-\eta C'(R_1)\bigr)\bigr)\le d\bigl(\tphi_{F,\epsilon}(y_1),\phi_j'(y_1)\bigr)\le \Delta$. Since $ y_2\notin W_j'$, we have $d\bigl(\tphi_{F,\epsilon}(y_2),P_{j_1}'\smallsetminus V_j'\bigr)\le \Delta$. Then we have $d\bigl(\tphi_{F,\epsilon}(y_1),\tphi_{F,\epsilon}(y_2)\bigr)\ge d\bigl(\Int\bigl(V_j',r_1-\eta C'(R_1)\bigr),P_{j_1}'\smallsetminus V_j'\bigr)-2\Delta\ge r_1-\eta C'(R_1)-2\Delta$. Thus, 
    \begin{equation}\label{e4.2h}
        d(y_1,y_2)
        \ge d\bigl(\tphi_{F,\epsilon}(y_1),\tphi_{F,\epsilon}(y_2)\bigr)\big/(2\kappa)-C'(R_1)
        \ge (r_1-\eta C'(R_1)-2\Delta)/(2\kappa)-C'(R_1) .
    \end{equation}
    Define, for each $R>0$,
    \begin{equation}\label{e4.2h1}
        r(R)\= \bigl(8\kappa \eta^2+2\bigr)C'(R)+2\Delta \text{ and take }r_1\=r(R_1).
    \end{equation} Then since $W_j'\subseteq B(x,3\kappa R_1-2)$ and $\eta>2\kappa$, by \eqref{e4.2h}, our choice of $r_1$ implies 
    $d\bigl(y_1,F\smallsetminus W_j'\bigr)
        =d\bigl(y_1,(B(x,3\kappa R_1)\cap F)\smallsetminus W_j'\bigr) 
        \ge \bigl(\bigl(8\kappa \eta^2+2-\eta\bigr)\big/(2\kappa)-1\bigr)C'(R_1)>3\eta^2 C'(R_1)$. Thus,
    \begin{equation}\label{e4.2a3}
            \phi_j'^{-1}\bigl(\Int\bigl(V_j',r_1-\eta C'(R_1)\bigr)\bigr)\cap W_j'\subseteq \Int_{F} \bigl(W_j',3\eta^2 C'(R_1)\bigr).
    \end{equation}
    
    Then since $\Int\bigl(V_j',r_1\bigr)$ is path-connected and $B\bigl(\Int\bigl(V_j',r_1\bigr),\eta C'(R_1)\bigr)\subseteq \Int\bigl(V_j',r_1-\eta C'(R_1)\bigr)$, by Corollary~\ref{t4.0a} with $(\psi,U,V)$ taken to be $\bigl(\phi_j',W_j',\Int\bigl(V_j',r_1\bigr)\bigr)$, \eqref{e4.2a2} now follows from \eqref{e4.2a3}.

   We then choose suitable constants to derive \eqref{e4.2c} from \eqref{e4.2a2}. Take $r_0\coloneqq r(R_0)= \bigl(8\kappa \eta^2+2\bigr)C'(R_0)+2\Delta$ and $R_1\=R_0+4r(R_0)$. The restriction $R_1<2R_0$ now translates into $4r(R_0)<R_0$, that is,
    \begin{equation}\label{e4.2h4}
        12\bigl(8\kappa\eta^2+2\bigr)\bigl(6\kappa^2\epsilon R_0+C_\epsilon\bigr)+8\Delta<R_0.
    \end{equation} 
    Define 
    \begin{equation}\label{e4.2h3}
        \epsilon_3'
        \=\min\bigl\{\epsilon_0,\, \bigl(144\kappa^2\bigl(8\kappa\eta^2+2\bigr)\bigr)^{-1}\bigr\},
    \end{equation} 
    which depends only on $\kappa$. Define
    \begin{equation}\label{e4.2h2}
    R_3\= 2C_\epsilon \big/ \bigl(\kappa^2 \epsilon \bigr).
    \end{equation}
    Note that if $\epsilon\in (0,\epsilon_3')$, then $R_3>144\bigl(8\kappa\eta^2+2\bigr)C_\epsilon$.
    Then since $C_\epsilon>\Delta$, if $\epsilon\in(0,\epsilon_3')$ and $R_0>R_3$, we have $12\bigl(8\kappa\eta^2+2\bigr)6\kappa^2\epsilon R_0< R_0/2$ and $12\bigl(8\kappa\eta^2+2\bigr)C_\epsilon+8\Delta<R_0/2$. Hence \eqref{e4.2h4} follows. Thus $R_1<2R_0$ and $r_1<2r_0$. So $R_1=R_0+4r_0>R_0+r_1$.
    
    Then by our definition of the sets $V_j$ and $V_j'$, we have
    \begin{equation}\label{e4.2c1}
        V_j=\Int\bigl(P_j',r_1\bigr)\cap B_{R_0}=\Int\bigl(P_j',r_1\bigr)\cap \Int(B_{R_1},r_1)\cap B_{R_0}=\Int\bigl(V_j',r_1\bigr)\cap B_{R_0}
    \end{equation} 
    and then by \eqref{e4.2a2}, \eqref{e4.2a1}, and the fact that $r_0>\eta C'(R_1)+\Delta$ (see \eqref{e4.2h1}),
    \begin{equation*}
        V_j\subseteq \Nb \bigl(\phi_j'\bigl(W_j'\bigr),\eta C'(R_1)\bigr)
        \subseteq \Nb\bigl(\tphi_{F,\epsilon}(F),\eta C'(R_1)+\Delta\bigr)
        \subseteq \Nb\bigl(\tphi_{F,\epsilon}(F),r_0\bigr).
    \end{equation*}
    Then by our definition of $V_j$, 
    \begin{equation}\label{e4.2c}
     \bigcup_{i\in \sigma}\Int(P_i',r_1)\cap B_{R_0}= \bigcup_{i\in \sigma}V_i\subseteq \Nb\bigl(\tphi_{F,\epsilon}(F),r_0\bigr).
    \end{equation}

    We will set $P_i\=\Int(P_i',r_1)$, leading to Part~(iii) for suitable choices of constants. For now, we shift our attention to Part(ii).
     
    \smallskip
    \textbf{Step 4.} For the remaining part of the proof, we take $\epsilon\in(0,\epsilon_3')$, $R_0>R_3$, $r_0\=r(R_0)$, $R_1\=R_0+4r_0$, and $r_1\=r(R_1)$. For each $i\in\sigma$, define $P_i\=\Int(P_i',r_1)$.
    For $0<v<R_0/2$, write
    \begin{equation}\label{e4.2c2}
         d_v\=\min\Bigl\{v,\,\max_{y\in \tphi_{F,\epsilon}(F)\cap B_{R_1-v}}\Bigl\{  d\Bigl(y,\bigcup_{i\in \sigma}\Int (P_i',r_1 )\cap B_{R_1}\Bigr)\Bigr\}\Bigr\}.
     \end{equation}
     Here we use the convention that $d(y,\emptyset)=+\infty$.
     
     Let $0<v<R_0/2$ be a constant and $x_0\in \tphi_{F,\epsilon}(F)\cap B_{R_1-v}$ satisfy $d\bigl(x_0,\bigcup_{j\in \sigma}\Int\bigl(P_j',r_1\bigr)\cap B_{R_1}\bigr)\ge d_v$. For each $1\le j\le m'$, recall that $F_j'$ is a flat containing both $P_j'$ and $\phi_{F,\epsilon}(x)$. Define 
     \begin{equation}\label{e4.2c3}
         Y_j\coloneqq B\bigl(\phi_j'\bigl(W_j'\bigr),C'(R_1)\bigr)\cap F_j'.
     \end{equation}
     In this step, we now obtain an upper bound for $d_v$ by estimating $\sum_{i=1}^{m'}\abs{B(x_0,d_v)\cap Y_i}$ in terms of $d_v$.
     
     On the one hand, we obtain an upper bound of $\sum_{i=1}^{m'}\abs{B(x_0,d_v)\cap Y_i}$ in terms of $d_v$. Recall that for each index $1\le j\le m'$ we have $\phi_j'\bigl(W_j'\bigr)\subseteq V_j'$. Hence for $j\in \sigma$, since $B(x_0,d_v)\cap \bigcup_{i\in \sigma}V_i\subseteq B(x_0,d_v)\cap \bigcup_{i\in \sigma}(P_i\cap B_{R_1})=\emptyset$, we have
     \begin{equation*}
         B(x_0,d_v)\cap Y_j\subseteq \bigl(B\bigl(V_j',C'(R_1)\bigr)\cap F_j'\bigr)\smallsetminus V_j.
     \end{equation*}For $j\notin \sigma$, since $\phi_j'\bigl(W_j'\bigr)\cap V_j= \phi_j'\bigl(W_j'\bigr)\cap \Int\bigl(V_j',r_1\bigr)\cap B_{R_0}=\emptyset$ by \eqref{e4.2c1} and \eqref{e4.2a}, we similarly have $
         Y_j\subseteq \bigl(B\bigl(V_j',C'(R_1)\bigr)\cap F_j'\bigr)\smallsetminus \Int(V_j,C'(R_1))$.
     
     Let $1\le j\le m'$ be an index. By $R_1-4r_1<R_0$ and our definition of $V_j$ and $V_j'$, we have $\Int\bigl(V_j',4r_1\bigr)\subseteq V_j$, and thus 
     \begin{equation}\label{e4.2i4}
         B(x_0,d_v)\cap Y_j\subseteq \bigl(B\bigl(V_j',C'(R_1)\bigr)\cap F_j'\bigr)\smallsetminus \Int\bigl(V_j',5r_1\bigr)
     \end{equation}
     by combining the above estimates for $j\in \sigma$ and $j\notin\sigma$ with $C'(R_1)<r_1$ (cf.~\eqref{e4.2h1}). Recalling $V_j'=P_j'\cap B_{R_1}$, we have
     \begin{align}        
        &\bigl(B\bigl(V_j',C'(R_1)\bigr)\cap F_j'\bigr)\smallsetminus \Int\bigl(V_j',5r_1\bigr)\notag\\
         &\qquad\subseteq \bigl(B\bigl(B_{R_1}\cap F_j',C'(R_1)\bigr)\smallsetminus \Int\bigl(B_{R_1}\cap F_j',5r_1\bigr)\bigr)\cup \bigl(B\bigl(P_j',C'(R_1)\bigr)\smallsetminus\Int\bigl(P_j',5r_1\bigr)\bigr)\cap F_j'\notag\\
         &\qquad\subseteq\bigl((B_{R_1+C'(R_1)}\smallsetminus B_{R_1-5r_1})\cap F_j'\bigr)\cup M_j, \label{e4.2i3}
    \end{align}
 where $M_j\=\bigcup_{L\in \cI_j}B(L,5r_1+C'(R_1)+1)\cap F_j'$. Here $\cI_j$ denotes the set of hyperplanes bounding $P_j'$. For each $L\in \cI_j$, the intersection of $B(L,6r_1)\cap F_j'$ and $B(x_0,d_v)$ is contained in a $(12r_1+1)\times (2d_v+1)$ or $(2d_v+1)\times (12r_1+1)$ rectangle. Thus, 
     \begin{equation}\label{e4.2i1}
         \abs{M_j\cap B(x_0,d_v)}\le 4(12r_1+1)(2d_v+1)<100r_1d_v.
     \end{equation}
     Here we first assume that $d_v>100$ and note that $r_1>100$ by our choice of constants. By direct elementary Euclidean area estimates (cf.~Figure~\ref{fig4.2}), we have 
     \begin{equation}\label{e4.2i2}
         \Absbig{\bigl((B_{R_1+C'(R_1)}\smallsetminus B_{R_1-5r_1})\cap F_j'\bigr)\cap B(x_0,d_v)}\le 2\pi (d_v+2)(5r_1+C'(R_1)+2)<50r_1d_v.
     \end{equation}

     Thus, by combining \eqref{e4.2i4}--\eqref{e4.2i2}, we have 
      \begin{equation*}
          \abs{B(x_0,d_v)\cap Y_j}
          \le \Absbig{B(x_0,d_v)\cap\bigl(\bigl(B\bigl(V_j',C'(R_1)\bigr)\cap F_j'\bigr)\smallsetminus \Int\bigl(V_j',5r_1\bigr)\bigr)}<150r_1d_v.
      \end{equation*} 
      Consequently, 
      \begin{equation}\label{e4.2e}
          \sum_{i=1}^{m'}\abs{B(x_0,d_v)\cap Y_i}\le 150m'r_1d_v.
      \end{equation}
     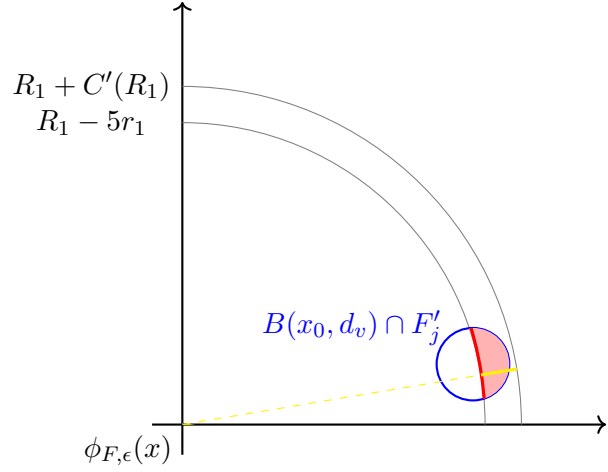
\begin{figure}
        \centering
        \begin{tikzpicture}[scale=0.8]

    \def\R{5}
    \def\Rone{5.6}
    \def\dv{0.6}
    \def\angle{35} 
    \draw[->, thick] (-0.5,0) -- (7,0) ;
    \draw[->, thick] (0,-0.5) -- (0,7) ;

    \draw[thin, gray] (\R,0) arc[start angle=0, end angle=90, radius=\R];
    \draw[thin, gray] (\Rone,0) arc[start angle=0, end angle=90, radius=\Rone];

    \pgfmathsetmacro{\cx}{\Rone - 0.8} 
    \pgfmathsetmacro{\cy}{1}

    \draw[blue, thick] (\cx, \cy) circle (\dv);

    \begin{scope}

        \clip (\cx, \cy) circle (\dv);

        \fill[red!30, even odd rule] 
            (\Rone,0) arc[start angle=0, end angle=90, radius=\Rone] -- (0,0)
            (\R,0) arc[start angle=0, end angle=90, radius=\R] -- (0,0);
    \end{scope}

    \begin{scope}
       \clip (\cx, \cy) circle (\dv);
        \draw[red, very thick] (\R,0) arc[start angle=0, end angle=90, radius=\Rone];
   \end{scope}
    \begin{scope}[even odd rule]
       \clip (\Rone,0) arc[start angle=0, end angle=90, radius=\Rone] -- (0,0)
            (\R,0) arc[start angle=0, end angle=90, radius=\R] -- (0,0)
            ;
        \draw[yellow, very thick] (0,0) -- (6,1);
   \end{scope}
   \begin{scope}
       \clip (\Rone,0) arc[start angle=0, end angle=90, radius=\Rone] -- (0,0) ;
        \draw[yellow, dashed] (0,0) -- (6,1);
   \end{scope}

    \node[below left] at (0,0) {$\phi_{F,\epsilon}(x)$};
   
    \node[blue] at (\cx-2, \cy + 0.6) {$B(x_0,d_v)\cap F_j'$};
    \node at (-1.5, \Rone) {$R_1+C'(R_1)$};
    \node at (-1.5, \R) {$R_1-5r_1$};

    \end{tikzpicture}
        \caption{One of the possible cases}
        \label{fig4.2}
    \end{figure}
     On the other hand, we can also obtain a lower bound of $\sum_{i=1}^{m'}\abs{B(x_0,d_v)\cap Y_i}$ in terms of $d_v$. Suppose $x_1\in F\cap \tphi^{-1}_{F,\epsilon}(x_0)$ and $U_j\=B(x_1,d_v/(2\kappa)-2C'(R_1))\cap W_j'$. Since $\phi_j'$ is a $(2\kappa,C'(R_1))$-quasi-isometric embedding on $W_j'$, we have $\phi_j'(U_j)\subseteq B(x_0,d_v-C'(R_1))$. Thus, by \eqref{e4.2c3},
     \begin{equation*}
         B\bigl(\phi_j'(U_j),C'(R_1)\bigr)\cap F_j'
         \subseteq B(B(x_0,d_v-C'(R_1)),C'(R_1))\cap Y_j
         \subseteq B(x_0,d_v)\cap Y_j.
     \end{equation*}
     Then by Lemma~\ref{t4.0b} with $(\psi,U)$ set to be $\bigl(\phi_j',U_j\bigr)$, writing $\zeta\= 3^{-4}2^{-2}\kappa^{-2}$, we have 
     \begin{equation}\label{e4.2e1}
         \abs{B(x_0,d_v)\cap Y_j}
         \ge \Absbig{B\bigl(\phi_j'(U_j),C'(R_1)\bigr)\cap F_j'}\ge \zeta \abs{U_j}.        
     \end{equation}
     
     Now since $\tphi_{F,\epsilon}$ is a $(2\kappa,C'(R_1)-4\Delta)$-quasi-isometric embedding on $B(x,3\kappa R_1)\cap F$ and $B(x_1,d_v/(2\kappa)-2C'(R_1))\subseteq B(x,3\kappa R_1)$, we have $\tphi_{F,\epsilon}(B(x_1,d_v/(2\kappa)-2C'(R_1))\cap F)\subseteq B(x_0,d_v-C'(R_1)-\Delta)$. Since $x_0\in B_{R_1-v}$, $v\ge d_v$ (cf.~\eqref{e4.2c2}), and $C'(R_1)>0$, we have 
     \begin{equation*}
         B(x_0,d_v-C'(R_1)-\Delta)\subseteq B_{R_1-\Delta}.
     \end{equation*}
     Thus, by \eqref{e4.2g}, \eqref{e4.2}, \eqref{e4.2a1}, and $V_i'=P_i'\cap B_{R_1}$, we have $B(x_1,d_v/(2\kappa)-2C'(R_1))\cap F\subseteq  \tphi_{F,\epsilon}^{-1}(B_{R_1-\Delta})\cap F\subseteq \bigcup_{i=1}^{m'}W_i'$. So by \eqref{e4.2e1} and \eqref{e3.0}, if $d_v\ge 8\kappa C'(R_1)$ and $\zeta'\=3^{-4}2^{-6}\kappa^{-4}$, then
     \begin{equation}\label{e4.2f}
         \sum_{i=1}^{m'}\abs{B(x_0,d_v)\cap Y_i}
         \ge \zeta \sum_{i=1}^{m'} \abs{U_i}
         \ge \zeta(d_v/(2\kappa)-2C'(R_1))^2\ge \zeta'd_v^2.
     \end{equation}

    Thus by combining \eqref{e4.2a}, \eqref{e4.2e}, \eqref{e4.2f}, and our assumption that $d_v>100$, we have
    \begin{equation}\label{e4.2f1}
        d_v\le \max\{100,\,150m'r_1/\zeta',\,8\kappa C'(R_1)\}=150m'r_1/\zeta'.
    \end{equation}

    \textbf{Step 5.} In this step, we complete the proof of Parts~(ii) and (iii). We need to choose suitable constants such that $d_v<v<R_0/2$. For $R>0$, write
    \begin{equation}\label{e4.2f3}
        v(R)\=160m'r(R)/\zeta' \quad \text{ and } \quad 
        v_1\=v(R_1)=160m'r_1/\zeta'.
    \end{equation}
    We now specify the choices of $\epsilon$ and $\lambda_3$ so that $v_1<R_0/3$ and $6 d_{v_1}<\lambda_3 \epsilon R$. Define
    \begin{equation}\label{e4.2f2}
        \lambda_3\= 3^42^810^5M_1^3\kappa^6\bigl(8\kappa\eta^2+4\bigr),
    \end{equation}
    which depends only on $\kappa$. Then by \eqref{e4.2b}, $r_1=r(R_1)$, \eqref{e4.2h1}, \eqref{e4.2b1}, and \eqref{e4.2f2}, we have
    \begin{equation}\label{e4.2d0}
        1000m'r_1/\zeta'\le 3^42^810^3\kappa^4M_1^3r_1<3^42^810^3\kappa^4M_1^3\bigl(8\kappa\eta^2+4\bigr)C'(R_1)<\lambda_3\epsilon R_1.
    \end{equation}
    Hence for $\epsilon\in (0,\epsilon_3)$, where 
    \begin{equation}\label{e4.2d1}
        \epsilon_3\=\min\{\epsilon_3',\,1/\lambda_3\}
    \end{equation}
    depends only on $\kappa$, we have $v_1=160m'r_1/\zeta'<\lambda_3\epsilon R_1/6<R_1/6<R_0/3$. 
    Then by \eqref{e4.2f1}, $d_{v_1}<v_1$, and thus by \eqref{e4.2c2}, $2v(R_1-v_1)>v(R_1)$, $P_i=\Int(P_i',r_1)$, \eqref{e4.2d0}, and $R_1-v_1>R_1-R_0/3>R_0/2$, we have
    \begin{equation}\label{e4.2d}
    \begin{aligned}
        \tphi_{F,\epsilon}(F)\cap B_{R_1-v_1}
        &\subseteq\Nb(P_\sigma\cap B_{R_1},d_{v_1}+1)\\
        &\subseteq\Nb(P_\sigma,150m'r_1/\zeta'+1)
        \subseteq \Nb(P_\sigma,\lambda_3\epsilon(R_1-v_1)),
    \end{aligned}
    \end{equation}
    where $P_\sigma \= \bigcup_{i\in \sigma}P_i$. Since $v_1>4r_0$, we have $R_1-v_1<R_0$, and thus \eqref{e4.2c} implies
    \begin{equation}\label{e4.2c0}
        P_\sigma \cap B_{R_1-v_1}\subseteq \Nb\bigl(\tphi_{F,\epsilon}(F),r_0\bigr)\subseteq\Nb\bigl(\tphi_{F,\epsilon}(F),\lambda_3\epsilon(R_1-v_1)\bigr).
    \end{equation}
    
    Now since $f\:R_0\mapsto R_1-v_1$ is an affine increasing map by $R_1=R_0+4r(R_0)$, \eqref{e4.2h1}, \eqref{e4.2b2}, \eqref{e4.2f3}, and our choice of $\epsilon$, for each constant $R>R_3$ there exists a constant $R_0>R_3$ such that $R_1-v_1=R$. Thus \eqref{e4.2d} and \eqref{e4.2c0} hold with $R_1-v_1$ replaced by $R$ for all $R>R_3$ (cf.~\eqref{e4.2h2}), establishing Parts~(ii) and (iii). Here $\sigma_R$ is defined in \eqref{e4.2a} with $R_0\=f^{-1}(R)$.

\smallskip

    \textbf{Step 6.} We finally show that the set $\theta_R\=\{i\in \sigma_R:P_i'\text{ is a chamber}\}$ is independent of the choice of a sufficiently large $R$.

    \smallskip
    
    \textit{Claim.} There exists a constant $R_1''$ such that if $R\ge R_1''$, $t\in [-1,1]$, $i\in \theta_R$, then $i\in \theta_{(1+t\epsilon)R}$. 

    \smallskip

    If the claim is verified, then $\theta_R$ is independent of the choice of $R\ge R_1''$.

    \smallskip
    
    \textit{Proof of Claim.} By geometry of flats there exists a constant $R_F>0$ such that for each $1\le i\le m$ and each $R>R_F$, $P_i'\cap B(\phi_{F,\epsilon}(x),R)$ contains a ball on $F_i'$ of radius $0.4R$ if $P_i'$ is a chamber. Define 
    \begin{equation}\label{e4.2j1}
        R_1''\=\max\{R_3,\,2R_F/\epsilon\}.
    \end{equation}
    Suppose $R>R_1''$, $t\in [-1,1]$, and $i\in \theta_R$. Let $f$ be the map given above and set $R_0\=f^{-1}(R)$. Write $V^*_i \= \Int(V_i',5r_1)$. Since $x/2<f(x)<x$ for each $x>R_3$, to prove $i\in \theta_{(1+t\epsilon)R}$, by \eqref{e4.2a} and our choices of constants, it suffices to verify 
    \begin{equation}\label{e4.2j}
        \pi_{V_i}\bigl(B(V_i,\Delta)\cap \tphi_{F,\epsilon}(F)\bigr)\cap V^*_i \cap B_{R_1/2}\neq \emptyset.
    \end{equation}
    By our choice of $R_F$, $R_1>100r_1$, and elementary geometric considerations, there exists a point $z'\in P_i'$ such that $B(z',0.1R_1)\cap F_i'\subseteq \Int(V_i',5r_1)\cap B_{R_1/2}$. Since $\eta C'(R_1)<0.1R_1$, by \eqref{e4.2a2}, we have $\phi_i'(W_i')\cap B(z',0.1R_1)\neq \emptyset$. By our definitions of $V_i$ and $W_i'$ and geometric considerations, 
    \begin{equation*}
         \phi_i'(W_i')\cap V^*_i
         =\pi_{V_i'}\bigl(B(V_i',\Delta)\cap \tphi_{F,\epsilon}(F)\bigr)\cap V^*_i
         =\pi_{V_i}\bigl(B(V_i,\Delta)\cap \tphi_{F,\epsilon}(F)\bigr)\cap V^*_i.
    \end{equation*}
    By combining the above relations, we verify \eqref{e4.2j} and the claim holds.
\end{proof}

We have laid out the proof of Lemma~\ref{t4.2} in full as a gentle illustration; the arguments from Lemma~\ref{t4.4} through Proposition~\ref{t4.9a} follow the same thread, albeit at great length, and we shall therefore leave their detailed proofs to the reader's indulgence.

\subsection{Chambers to chambers}
In this subsection, we present Lemma~\ref{t4.4} to Proposition~\ref{t4.9a} as a series of subsequent geometric lemmas. The final goal of these lemmas is to show that chambers are mapped near chambers, as stated in Proposition~\ref{t4.9a}.

Based on Lemma~\ref{t4.2}, it can be shown that hyperplanes are almost surely mapped near finite unions of hyperplanes, as stated in the following lemma: 
\begin{lemma}[Hyperplanes to finite unions of hyperplanes]\label{t4.4}
      Given $\kappa>1$, there exist constants $M'\in \N$, $\epsilon_2\in (0,1/100)$, and $\lambda_2>1$ such that if $C>0$, $\epsilon\in(0,\epsilon_2)$, and $\rho>1/\epsilon$, then there exists a constant $R_2>0$ such that for each $D\subseteq \X$, each $\phi\in \QIE_{\kappa,C}(D,\X)$, each $F\in \cF_{\epsilon,D,\rho}'$, each $x\in U_3(F)$, and each hyperplane $L\subseteq F$ passing through $x$, there exist hyperplanes $L_1,\,\dots,\,L_{M'}$ such that for each $R>R_2 $,  
    \begin{equation}\label{et4.4a}
    \tphi_{F,\epsilon}(L)\cap A_R(\phi_{F,\epsilon}(x))\subseteq \Nb\biggl(\bigcup_{i=1}^{M'}L_i,\lambda_2\epsilon R\biggr)\cap A_R(\phi_{F,\epsilon}(x)),
    \end{equation} 
    and moreover, if Proposition~\ref{t3.5}~(iii) holds for $M_1=1$, then there exists a single hyperplane $L_0$ such that 
    $       \tphi_{F,\epsilon}(L)\cap A_R(\phi_{F,\epsilon}(x))\subseteq \Nb(L_0,\lambda_2\epsilon R)\cap A_R(\phi_{F,\epsilon}(x))$.
\end{lemma}

Recall that $\tphi_{F,\epsilon}$ is the regularized grid map defined in Proposition~\ref{t3.5} and $A_R(\phi_{F,\epsilon}(x))$ is the annular domain defined in Definition~\ref{d2.3}.
    We use the concept of transverse flats and Lemma~\ref{t4.2} to prove Lemma~\ref{t4.4}. Namely, if a hyperplane $L$ is equal to the intersection of two transverse flats $F$ and $F'$, then the image of $L$ is near the union of the intersections of the polyhedra obtained in Lemma~\ref{t4.2} for $F$ and $F'$. Each such intersection is close to a hyperplane since its preimage lies in some neighborhood of $L$. We also rely on negative curvature features of the space; in the setting of this article, we would only need the simple geometric input from Lemma~\ref{t4.3}.

The following statement can be deduced from Lemma~\ref{t4.4} by geometry of singular geodesic rays. It can be viewed as a weaker version of Proposition~\ref{t4.5}.
    \begin{lemma}\label{t4.9}
    Given $\kappa>1$, let $\epsilon_2\in (0,1/100)$ and $\lambda_2>1$ be the constants given in Lemma~\ref{t4.4}. If $C>0$, $D\subseteq \X$, $\phi\in \QIE_{\kappa,C}(D,\X)$, $\epsilon\in (0,\min\{\epsilon_2,\,1/(100\kappa\lambda_2)\})$, $\rho>1/\epsilon$, $F\in \cF_{\epsilon,D,\rho}'$, and $x\in U_3(F)$, then there exists $R_0'>0$ such that for each singular geodesic ray $L\subseteq F$ emanating from $x$, there exists a unique singular geodesic ray $L'$ emanating from $\phi_{F,\epsilon}(x)$ such that 
    \begin{equation*}
    \tphi_{F,\epsilon}(L)\smallsetminus B(\phi_{F,\epsilon}(x),R_0')\subseteq L'[3\lambda_2\epsilon]\smallsetminus B(\phi_{F,\epsilon}(x),R_0').
    \end{equation*}
\end{lemma}
 
We then derive the following lemma from Lemma~\ref{t4.9}. Roughly speaking, it states that the preimages of walls of chambers are coarsely singular geodesic rays. The key geometric input in this lemma is Lemma~\ref{t4.0c}.
\begin{lemma}\label{t4.9b}
    Given $\kappa>1$, there exist constants $\lambda_1'>1$ and $\epsilon_1'\in (0,1/100)$ such that if $C>0$, $D\subseteq \X$, $\phi\in \QIE_{\kappa,C}(D,\X)$, $\epsilon\in \bigl(0,\epsilon_1'\bigr)$, $\rho>1/\epsilon$, $F\in \cF_{\epsilon,D,\rho}'$, and $x\in U_3(F)$, then there exists a constant $R_1'>0$ such that for each $j\in\{1,\,2\}$ and each chamber $\fC'=L_1'\times L_2'$ with apex at $\phi_{F,\epsilon}(x)$ that is Hausdorff equivalent to some polyhedron $P_i',\,i\in\theta_{R_1''}$ given in Lemma~\ref{t4.2}, there exists a singular geodesic ray $L_j\subseteq F$ emanating from $x$ such that 
    \begin{equation*}
        L_j'\smallsetminus B(\phi_{F,\epsilon}(x),R_1')\subseteq \tphi_{F,\epsilon}(L_j)[\lambda_1'\epsilon]\smallsetminus B(\phi_{F,\epsilon}(x),R_1').
    \end{equation*}
\end{lemma}

It can now be verified that chambers are mapped near chambers based on Lemma~\ref{t4.9b}. 
\begin{prop}[Chambers to chambers]\label{t4.9a}
    Given $\kappa>1$, there exist constants $\epsilon_2'\in (0,1/100)$ and $\lambda_1>1$ such that if $C>0$, $D\subseteq \X$, $\phi\in \QIE_{\kappa,C}(D,\X)$, $\epsilon\in (0,\epsilon_2')$, $\rho>1/\epsilon$, $F\in \cF_{\epsilon,D,\rho}'$, $x\in U_3(F)$, and $\fC=L_1\times L_2\subseteq F$ is a chamber with apex at $x$, then there exists a chamber $\fC'=L_1'\times L_2'$ unique up to Hausdorff equivalence (cf.~Definition~\ref{d2.4}) such that the following hold:
    \begin{enumerate}[label=\rm{(\roman*)}]
       \smallskip
       \item $\tphi_{F,\epsilon}(\fC)\sim_{\lambda_1\epsilon} \fC'$.
       \smallskip
       \item There exist $\sigma\in\{\id,\,(1,2)\}$ and $R_2'>0$ such that 
    \begin{equation*}
        \tphi_{F,\epsilon}(L_i)\smallsetminus B(\phi_{F,\epsilon}(x),R_2')\subseteq L_{\sigma(i)}'[\lambda_1\epsilon]\smallsetminus B(\phi_{F,\epsilon}(x),R_2') .
    \end{equation*}     
    \end{enumerate}
\end{prop}

\subsection{Hyperplanes to hyperplanes}
Now we are ready to complete the proofs of Propositions~\ref{t4.5} and \ref{t4.8} based on Proposition~\ref{t4.9a}.
    
By utilizing \cite[Theorem~1.2]{wortman2006quasiflats}, we are now able to deduce the following proposition from Proposition~\ref{t4.9a}. It is a strengthening of Proposition~\ref{t3.5}~(iii).
\begin{prop}\label{t4.8}
    Suppose $\kappa>1$, $C>0$, $\epsilon\in (0,\epsilon_2')$, and $\rho>1/\epsilon$, where $\epsilon_2'\in (0,1/100)$ is the constant given in Proposition~\ref{t4.9a}. Then there exists a constant $\Delta_1>0$ such that for each $D\subseteq \X$, each $\phi\in \QIE_{\kappa,C}(D,\X)$, and each flat $F\in \cF_{\epsilon,D,\rho}'$, there exists a flat $F'$ such that 
    \begin{equation}\label{e4.8t}
        \tphi_{F,\epsilon}(F)=\phi_{F,\epsilon}(U_1(F))\subseteq \Nb(F',\Delta_1).
    \end{equation}
    Here $U_1(F)$ is the subset of $F$ defined in Section~\ref{step1}.
\end{prop}
\begin{proof}
    
    Suppose $D\subseteq\X$, $\phi\in \QIE_{\kappa,C}(D,\X)$, and $F\in \cF_{\epsilon,D,\rho}'$. Suppose $x$ is an arbitrary point in $U_3(F)$. Let $\fC_i,\,1\le i\le 4$, denote the chambers in $F$ with apex at $x$. 
    
    For each $1\le i\le 4$, let $\fC_i'$ be the chamber such that $\tphi_{F,\epsilon}(\fC_i)\sim_{\lambda_1\epsilon} \fC_i'$ given in Proposition~\ref{t4.9a}~(i). By Proposition~\ref{t4.9a}~(ii), there exist singular geodesic rays $L_1', \,L_2'\subseteq T_1$ and $L_3',\, L_4'\subseteq T_2$ emanating from $\phi_{F,\epsilon}(x)$ such that each $\fC'_i$ is equivalent to one of their products. Thus, there exists a single flat $F'$ and a constant $R_3'>0$ such that $\bigcup_{1\le i\le 4}\fC'_i\smallsetminus B(\phi_{F,\epsilon}(x),R_3')= F'\smallsetminus B(\phi_{F,\epsilon}(x),R_3')$. Then $\tphi_{F,\epsilon}(F)\smallsetminus B(\phi_{F,\epsilon}(x),R_3'')\subseteq F'[2\lambda_1\epsilon]\smallsetminus B(\phi_{F,\epsilon}(x),R_3'')$ for some constant $R_3''>0$ by Proposition~\ref{t4.9a}, where $\lambda_1>0$ is the constant depending only on $\kappa$ given in Proposition~\ref{t4.9a}. Then clearly the limit set $\fL_{\tphi_{F,\epsilon},\phi_{F,\epsilon}(x)}(2\lambda_1\epsilon)$ of $\tphi_{F,\epsilon}$ (cf.~Definition~\ref{da.1}) is equal to $\{\fC'_i:1\le i\le 4\}$. Then by Remark~\ref{ra.1}, the flats given in Proposition~\ref{t3.5} can be chosen to be $\{F'\}$ and we conclude our proof. 
\end{proof}

\begin{proof}[\bf Proof of Proposition~\ref{t4.5}]
    
Let $\kappa>1$ and $C>0$ be constants. Let $D$ be a subset of $\X$ and $\phi\in \QIE_{\kappa,C}(D,\X)$. Let $\epsilon_2'\in (0,1/100)$ and $\lambda_2>0$ be the constants given in Proposition~\ref{t4.9a} and Lemma~\ref{t4.4}, respectively. Let $\epsilon\in(0,\epsilon_2')$, $\rho>1/\epsilon$, $F\in \cF_{\epsilon,D,\rho}'$, $x\in U_3(F)$, and $L\subseteq F$ be a hyperplane passing through $x$. Let $L=L_1\cup L_2$, where $L_1$ and $L_2$ are singular geodesic rays emanating from $x$. Let $L_i'$ be the singular geodesic ray emanating from $\phi_{F,\epsilon}(x)$ close to $\tphi_{F,\epsilon}(L_i)$ given in Lemma~\ref{t4.9} for each $i\in \{1,\,2\}$.

By Proposition~\ref{t4.8} and Lemma~\ref{t4.4}, there exist a hyperplane $L''$ such that
\begin{equation*}
        \tphi_{F,\epsilon}(L)\smallsetminus B(\phi_{F,\epsilon}(x),R_2)\subseteq L''[\lambda_2\epsilon]\smallsetminus B(\phi_{F,\epsilon}(x),R_2).
\end{equation*}
Then for each $i\in\{1,\,2\}$, $L_i'\subseteq L''$ up to Hausdorff equivalence and $d(L'',\phi_{F,\epsilon}(x))\le 2R_2$. Thus, there exists a hyperplane $L'\subseteq L_1'\cup L_2'$ such that writing $R_0''\=2\epsilon^{-1}R_2$, we have
\begin{equation*}
        \tphi_{F,\epsilon}(L)\smallsetminus B(\phi_{F,\epsilon}(x),R_0'')\subseteq L'[\lambda_2\epsilon]\smallsetminus B(\phi_{F,\epsilon}(x),R_0'').
\end{equation*}
    We now use Lemma~\ref{t4.0} to show the other side holds as well. Define $\phi_{L'}\=\pi_{L'}\circ \tphi_{F,\epsilon}$, where $\pi_{L'}$ denotes the nearest-point projection from $\X$ to $L'$. Then $\phi_{L'}$ is a $(2\kappa,4\kappa\lambda_2\epsilon R+6\kappa R_0'')$-quasi-isometric embedding on $L\cap B(x,R)$ for each $R>R_0''$. Let $R_0'\=6\kappa R_0''/(\lambda_2\kappa\epsilon)+R_0''$. Then for $R>R_0'$, $\phi_{L'}$ is a $(2\kappa,5\kappa\lambda_2\epsilon R)$-quasi-isometric embedding on $L\cap B(x,R)$. By Lemma~\ref{t4.0} with $(n,\psi)$ set to be $(1,\phi_{L'})$, for $R>R_0'$ we have
    \begin{equation*}
        L'\cap B(\phi_{F,\epsilon}(x),R/\eta)\subseteq B(\phi_{L'}(L\cap B(x,R)),5\eta\kappa \lambda_2\epsilon R)\subseteq B\bigl(\tphi_{F,\epsilon}(L\cap B(x,R)),6\eta \kappa\lambda_2\epsilon R\bigr).
    \end{equation*}
    Write $R_0\=\eta R_0'$ which depends only on $\kappa$, $C$, $\epsilon$, and $\rho$. If $R>R_0$, denoting $\lambda_1\=6\kappa\eta^2 \lambda_2$, we have
    \begin{equation}
        L'\cap B(\phi_{F,\epsilon}(x),R)\subseteq B\bigl(\tphi_{F,\epsilon}(L),\lambda_1\epsilon R\bigr).
    \end{equation}
    Thus, we verify that \eqref{e4.5t2} and \eqref{e4.5t1} hold.
    
    To verify the uniqueness, suppose $\tphi_{F,\epsilon}(L_1)\sim_{\lambda_1\epsilon,\phi_{F,\epsilon}(x),R_0}L_3'$ and $\tphi_{F,\epsilon}(L_1)\sim_{\lambda_1\epsilon,\phi_{F,\epsilon}(x),R_0}L_4'$. Then we have $L_3'\sim_{3\lambda_1\epsilon}L_4'$ and thus $L_3'=L_4'$ for $\epsilon\in(0,\epsilon')$ where $\epsilon'\=\min\{1/(3\lambda_1),\,1/200,\epsilon''\}$ by geometry of geodesic rays (cf.~Fact~\ref{t5.2}). The statement that the choices of $L'$ and $L_1'$ are independent of the choice of $\rho$ immediately follows from the uniqueness.
\end{proof}

\section{The boundary maps}      \label{s5}

The purpose of this section is to construct the boundary map induced by $\phi$ and show that it factors through the boundaries of the two factor trees. We also investigate the regularity of the map. 

To fix notations, given a reference point $e=(e_1,e_2)\in \X$, for each $i\in \{1,\,2\}$, equip the tree boundary $\partial T_i$ with distance 
\begin{equation}\label{e5.0}
    d_{e_i}(L_1,L_2) \=  (q+1)^{-\abs{L_1\wedge L_2}},
\end{equation}
where $L_1$ and $L_2$ are two geodesic rays in $T_i$ emanating from $e_i$ and $\abs{L_1\wedge L_2}$ denotes the length (i.e.,~cardinality) of their common prefix. 
    
    The product boundary $\wx$ can be seen as represented by chambers in $\X$ with apex at $(e_1,e_2)$, which can be identified as $\partial T_1\times \partial T_2$. We equip $\wx$ with the product metric $d_e$ given by 
    \begin{equation*}
        d_e((\xi_1,\xi_2),(\eta_1,\eta_2))^2=d_{e_1}(\xi_1,\eta_1)^2+d_{e_2}(\xi_2,\eta_2)^2.
    \end{equation*}
    
    Next, we fix a choice of a countable set $\cU\subseteq \partial T_1\cup \partial T_2$ that is dense in both $\partial T_1$ and $\partial T_2$, and define $\cC\= \{L_1\times L_2:L_1\in \cU\cap \partial T_1,\,L_2\in \cU\cap \partial T_2\}$. Then $\cC$ is a countable dense set in $\partial T_1\times \partial T_2$. Note that by the continuity of the action of tree automorphisms on $\partial T_i$, for every automorphism $\gamma=(\gamma_1,\gamma_2)$, where $\gamma_i\in \operatorname{Aut}(T_i)$ for $i\in\{1,\,2\}$, the image $\gamma(\cC)$ is a countable dense set of chambers with apex at $\gamma(e_1,e_2)$.

    Choose a countable set $\cF$ of flats such that each chamber in $\cC$ is contained in at least one flat in $\cF$. By taking parallel translates, we assume $e\in F$ for each $F\in \cF$. Denote by $\widetilde{\cU}$ the set of pairs $(L,F)$ where $L\in\cU$ and $F\in \cF$ is a flat containing $L$. 
    
    Let $D$ be a random subset of $\X$. Recall that $\cF_{\epsilon,D,\rho}'$ is defined to be the set of flats $F\in \cF_{\epsilon,D,\rho}$ with the property that \eqref{e4.0} holds for each $x\in F$. Then for each pair of constants $\epsilon\in(0,1/100)$ and $\rho>1/\epsilon$, the event $\bigl\{\cF\subseteq \cF_{\epsilon,D,\rho}'\bigr\}$ is the intersection of a countable collection of events with probability one (cf.~Lemma~\ref{t4.1}). Thus $\cF\subseteq \cF_{\epsilon,D,\rho}'$ almost surely.

\subsection{Choice of constants}
    Recall that the hypotheses of Proposition~\ref{t4.5} require the origin $x$ of the singular geodesic ray to be contained in $U_3(F)$. In this subsection, we prove that we may choose suitable constants to satisfy this restriction. We first define an extended function crucial to the statements of Theorems~\ref{t1.2},~\ref{t1.3},  and~\ref{t1.4}.
    \begin{definition}\label{d5.0}
        Suppose $\epsilon\in (0,1)$, $x\in \X$, and $D\subseteq \X$. Denote by $\cF_x$ the set of flats containing $x$. Define, for each $i\in\{0,\,1\}$, 
        \begin{align*}
            \rho_i(D,x,\epsilon)&\=\sup \bigl\{ \cL_{U_0(F)}\bigl(10^{-5-i}\epsilon^{8},x\bigr) : F\in \cF_x \bigr\},
        \end{align*}
        where $\cL$ is defined in Definition~\ref{d3.2}. Here we follow the convention that the supremum of a set not bounded above is $+\infty$. We call $\rho_0(D,x,\epsilon)$ the \emph{irregularity radius} of $D$ at $x$ with parameter $\epsilon$.
    \end{definition}
    Clearly for each $i\in\{0,\,1\}$, the function $\rho_i(D,x,\epsilon)$ is translation invariant, i.e., $\rho_i(\gamma(D),\gamma(x),\epsilon)=\rho_i(D,x,\epsilon)$ for each $\gamma\in\operatorname{Aut}(\X)$. We first prove the following lemma on the tail distribution of $\rho_i(D,x,\epsilon)$ which verifies that $\rho_i(D,x,\epsilon)$ is almost surely finite. It can be viewed as a strengthening of Lemma~\ref{t3.3}, in the sense that it holds simultaneously for all flats passing through a given point. 
    
    \begin{lemma}\label{t5.0b}
        Let $D$ be a random subset of $\X$. Given $\epsilon\in(0,1)$ and $\beta\in\bigl(10^{-9}\epsilon^{10},1\bigr)$, there exists a constant $\zeta_\beta>0$ such that for each $x\in \X$ and each $\rho>0$, the map defined on $\{0,\,1\}^{\X}$ by $E\mapsto \sup \bigl\{ \cL_{U_0^{E,\epsilon}(F)}\bigl(\beta,x\bigr) : F\in \cF_x \bigr\}$ is measurable and 
        \begin{equation*}
            \P\bigl(\sup \bigl\{ \cL_{U_0^{D,\epsilon}(F)}\bigl(\beta,x\bigr) : F\in \cF_x \bigr\}>\rho\bigr)<\zeta_\beta
            ^{-1}\exp \bigl(-\zeta_\beta\rho^2 \bigr).
        \end{equation*}
        In particular, for each $i\in \{0,\,1\}$, the map defined on $\{0,\,1\}^{\X}$ by $E\mapsto \rho_i(E,x,\epsilon)$ is measurable, and
        \begin{equation}\label{taildis}
            \P(\rho_i(D,x,\epsilon)>\rho)<\zeta_i^{-1}\exp \bigl(-\zeta_i\rho^2 \bigr)
        \end{equation}
        for $\zeta_i\=\zeta_{10^{-5-i}\epsilon^{8}}$. Thus, almost surely $\rho_i(D,x,\epsilon)<+\infty$ for each $i\in\{0,\,1\}$ and each $x\in \X$.
    \end{lemma}
    \begin{proof}
        Let $D$ be a random subset of $\X$. Suppose $\epsilon\in(0,1)$ and $\beta\in\bigl(10^{-9}\epsilon^{10},1\bigr)$. Write $U_0(F)\=U_0^{D,\epsilon}(F)$. Suppose $x=(x_1,x_2)\in \X$ and $F=L_1\times L_2$ is a flat containing $x$. We first show that
        \begin{equation}\label{e5.00}
            \P(\abs{B(x,r)\cap U_0(F)}\le (1-\beta)\abs{B(x,r)\cap F})\le \exp\bigl(-r^2I_1\bigr)
        \end{equation}
        for some constant $I_1>0$ and every sufficiently large $r>0$. For the proof of this estimate, we use a similar argument as in the proof of Lemma~\ref{t3.3}. 
        
        For each $r>0$, denote $W_r \=  \bigl\{(i,j)\in \Z^2:A_{ij}^{\epsilon}\cap B(x,r)\neq \emptyset\bigr\}$, where $A_{ij}^{\epsilon}$ is the $n_\epsilon\times n_\epsilon$ square in $\Z^2$ (as identified with $F$) defined in \eqref{e3.0b}. Note that   
        $\abs{B(z,r)\cap F}\le n_\epsilon^2\abs{W_r}\le \abs{B(z,r+2n_\epsilon)\cap F}$ for each $z\in \Z^2$. By \eqref{e3.0}, $\abs{W_r}\ge r^2\big/n^2_{\epsilon}$. Moreover, if $r>5n_{\epsilon}$, then $\abs{W_r}\le 10r^2\big/n^2_{\epsilon}$. 
     
        Denote $K_r \=  \Absbig{\bigl\{(i,j)\in W_r:A_{ij}^{\epsilon}\cap D= \emptyset\bigr\}}$. Then $K_r$ is the sum of $\abs{W_r}$ i.i.d.\ Bernoulli$(p_1)$ random variables, where $p_1\=(1-p)^{n_{\epsilon}^2}$. Write $a\=\beta/10$. Then $a>2p_1$ by our choice of $n_\epsilon$ (cf.~\eqref{e3.0a}). Thus by the Bernstein inequality for Bernoulli random variables,
    \begin{equation}\label{e5.0d}
         \P(K_r\ge a\abs{W_r})
         \le \exp(- a\abs{W_r} / 8).
     \end{equation}
     For $r>5n_\epsilon$, if $K_r< a\abs{W_r}$, since $n_{\epsilon}^2\abs{W_r}<10\abs{B(x,r)\cap F}$, by \eqref{e3.0} we have
     \begin{equation*}
          \abs{B(x,r)\cap U_0(F)}\ge \abs{B(x,r)\cap F}-n_\epsilon^2K_r> (1-10a)\abs{B(x,r)\cap F}=(1-\beta)\abs{B(x,r)\cap F}.
    \end{equation*}
    Thus, for $r>5n_\epsilon$, by $\abs{W_r}\ge r^2\big/n^2_{\epsilon}$ and \eqref{e5.0d}, denoting $I_1\=10^{-2}n_\epsilon^{-2}\beta$, we establish \eqref{e5.00} as follows:
    \begin{equation}\label{e5.0b}
        \P(\abs{B(x,r)\cap U_0(F)}\le (1-\beta)\abs{B(x,r)\cap F})\le \P(K_r\ge a\abs{W_r})
         \le \exp\bigl(-r^2I_1\bigr).
    \end{equation}
        
        For each $r>0$, define an equivalence relation $\equiv_r$ on the set $\cF_x$ of flats containing $x$ by $F_1\equiv_r F_2$ if and only if $B(x,r+2n_\epsilon)\cap F_1=B(x,r+2n_\epsilon)\cap F_2$. Let $\overline{\cF}_r$ be a complete set of representatives induced by $\equiv_r$. Then simple estimates show that $\abs{\overline{\cF}_r}\le (q+1)^{4(r+2n_\epsilon+1)}$. Again since we are working with $\Z^2$, we only consider $r$ with $r^2\in \Z$. Thus by Definition~\ref{d3.2}, $\sup \bigl\{ \cL_{U_0(F)}\bigl(\beta,x\bigr) : F\in \cF_x \bigr\}>\rho$ if and only if the event $\bigcup_{r^2\in \Z,r>\rho}\bigcup_{F\in \overline{\cF}_r}\{\abs{B(x,r)\cap U_0(F)}
            \le (1-\beta)\abs{B(x,r)\cap F})\}$ occurs. 
        Then the map defined on $\{0,\,1\}^{\X}$ by $E\mapsto \sup \bigl\{ \cL_{U_0^{E,\epsilon}(F)}\bigl(\beta,x\bigr) : F\in \cF_x \bigr\}$ is measurable, and by \eqref{e5.0b}, for $\rho>5n_\epsilon$ we have
        \begin{align*}
            \P\bigl(\sup \bigl\{ \cL_{U_0(F)}\bigl(\beta,x\bigr) : F\in \cF_x \bigr\}>\rho\bigr) 
            & \le \sum_{r^2\in \Z,r>\rho}\sum_{F\in \overline{\cF}_r}\P(\abs{B(x,r)\cap U_0(F)}
            \le (1-\beta)\abs{B(x,r)\cap F})\\
            &\le \sum_{r^2\in \Z,r>\rho}(q+1)^{4(r+2n_\epsilon+1)}\exp\bigl(-r^2I_1\bigr)
             < \zeta_\beta^{-1}\exp \bigl( -\zeta_\beta\rho^2 \bigr)
        \end{align*}
        for some constant $\zeta_\beta>0$ depending only on $\epsilon$ and $\beta$.
    \end{proof}
    
    \begin{lemma}\label{t5.0a}
        Suppose $\epsilon\in(0,1)$ and $\rho>1/\epsilon$ are constants. Suppose $x\in \X$, $D$ is a random subset of $\X$ containing $x$, and $F$ is a flat containing $x$. Almost surely, if $x\notin U_2^{\epsilon,D,\rho}(F)$, then 
        \begin{equation}\label{e5.0a}
            \cL_{U_0(F)}\bigl(10^{-5}\epsilon^{8},x\bigr)\ge \cL_{U_1(F)}\bigl(10^{-1}\epsilon^3,x\bigr)\ge \rho.
        \end{equation}
        Consequently, almost surely $x\in U_2^{\epsilon,D,\rho_2}(F)$ for each $\rho_2\ge \rho_0(D,x,\epsilon)$.
    \end{lemma}
    
    \begin{proof}
    Suppose $x\notin U_2^{\epsilon,D,\rho}(F)$. Then by Definition~\ref{d3.1}, there exists some $y\in F$ with $ d(y,x)\ge \rho$ such that $\abs{B(y,\epsilon d(x,y))\cap U_1(F)}\le (1-\epsilon)\abs{B(y,\epsilon d(x,y))}$. Thus, by \eqref{e3.0z}, we have $\abs{B(x,(1+\epsilon)d(x,y))\smallsetminus U_1(F)}\ge \abs{B(y,\epsilon d(x,y))\smallsetminus U_1(F)}\ge \epsilon\abs{B(y,\epsilon d(x,y))}\ge10^{-1}\epsilon^3\abs{B(x,(1+\epsilon)d(x,y))}$. By Definition~\ref{d3.2}, this implies $\cL_{U_1(F)}\bigl(10^{-1}\epsilon^3,x\bigr)\ge d(x,y)\ge \rho$. By Lemma~\ref{t5.0b}, there exists a full-measure subset $\cD$ of $\{0,\,1\}^\X$ where $\rho_0(E,y,\epsilon)<+\infty$ for each $y\in \X$ and each $E\in\cD$. Now if $D\in\cD$, \eqref{e5.0a} holds due to Lemma~\ref{t3.1}, whose finiteness assumption follows immediately from Definition~\ref{d5.0}.

    Then by Definition~\ref{d5.0} and \eqref{e5.0a}, for each $y\in \X$, each $E\in \cD$ containing $y$, and each $t>\rho_0(E,y,\epsilon)$, we have $y\in U_2^{\epsilon,E,t}(F)$. Since for each $\rho_2>0$, $U_2^{\epsilon,E,\rho_2}(F)=\bigcap_{t>\rho_2}U_2^{\epsilon,E,t}(F)$ by Definition~\ref{d3.1}, we verify that $y\in U_2^{\epsilon,E,\rho_2}(F)$ for each $y\in \X$, each $E\in \cD$ containing $y$, and each $\rho_2\ge \rho_0(E,y,\epsilon)$.
    \end{proof}
    
    The following proposition is an immediate consequence of Lemmas~\ref{t5.0a} and~\ref{t5.0b}.
    \begin{prop}\label{t5.0}
        Suppose $\epsilon\in (0,1/100)$ and $D$ is a random subset of $\X$. Define $\rho_0\=\max\{\rho_0(D,e,\epsilon),\,2/\epsilon\}$. Then almost surely $e\in U_3^{\epsilon,D\cup\{e\},\rho}(F)$ for each flat $F\in \cF$ and each $\rho\ge \rho_0$.
    \end{prop}
    \begin{proof}
        Suppose $\rho\ge \rho_0$. Let $\cF'$ be the set of flats that are the chosen transverse flats at $e$ of flats in $\cF$ (see Subsection~\ref{s4.1}). Since there are only countably many flats in $\cF\cup \cF'$, by Lemma~\ref{t4.1} we have $\cF\cup \cF'\subseteq \cF_{\epsilon,D,\rho}'\subseteq \cF_{\epsilon,D,\rho}$ almost surely. Then the proposition follows from the fact that $\rho_0(D,e,\epsilon)\ge \rho_0(D\cup\{e\},e,\epsilon)$, Lemmas~\ref{t5.0b}, \ref{t5.0a}, Definition~\ref{d4.1a}, and Remark~\ref{r:U3_monotone}.        
    \end{proof}

\subsection{Induced map on the singular rays}\label{s5.2}

     Suppose $\kappa>1$, $C>0$, and $\epsilon\in(0,\epsilon')$ are constants, where $\epsilon'\in(0,1/100)$ is the constant depending only on $\kappa$ given in Proposition~\ref{t4.5}. Suppose $D$ is a random subset of $\X$. We write $\rho_0\=\max\{\rho_0(D,e,\epsilon),2/\epsilon\}$ for the remainder of the section, where $\rho_0(D,e,\epsilon)$ is the irregularity radius defined in Definition~\ref{d5.0}. 
     
     To utilize Proposition~\ref{t5.0}, we define a random subset $D'\subseteq \X$ and a map $\phi'\:D'\to \X$ for the rest of this section. Define $D'\coloneqq D\cup \{e\}$. Let $x$ be some point in $D$ closest to $e$. Suppose $\phi\in \QIE_{\kappa,C}(D,\X)$, define $\phi'\:D'\to \X$ as an extension of $\phi$ by setting $\phi'(e)\=\phi(x)$. Then $\phi'$ is a $(\kappa,C+d(x,e)\kappa)$-quasi-isometric embedding from $D'$ to $\X$.
     
     Then by Propositions~\ref{t4.5} and~\ref{t5.0}, almost surely there exists a singular geodesic ray $L'$ emanating from $\phi'(e)$ such that $\tphi'_{F,\epsilon}(L)\sim_{\lambda_1\epsilon,\phi'(e),R_0}L'$ for each $\rho\ge \rho_0$, each $\phi\in \QIE_{\kappa,C}(D,\X)$, and each pair $(L,F)\in \widetilde{\cU}$. Moreover, the map $\phi_0 \: \widetilde{\cU}\to\partial T_1\cup \partial T_2$ (defined for almost every $D$) given by $(L,F)\mapsto L'$ is independent of the choice of $\rho\ge\rho_0$.

    The purpose of this subsection is to prove the following lemma:
\begin{lemma}\label{t5.1}
    Suppose $\kappa>1$ is a constant and $D$ is a random subset of $\X$. Then there exists a constant $\epsilon_4'\in (0,1/100)$ such that almost surely, for each $C>0$, each $\epsilon\in (0,\epsilon_4')$, each $\rho\ge \rho_0$, and each $\phi\in \QIE_{\kappa,C}(D,\X)$, for the map $\phi_0\:(L,F)\mapsto L'$ defined above, we have $\phi_0(L,F_1)=\phi_0(L,F_2)$ if $F_1,F_2\in \cF$ are two flats both containing $L$. 
    
\end{lemma}
Lemma~\ref{t5.1} shows that the image $L'$ is independent of the choice of the flat $F$ up to equivalence. The proof of the lemma relies on negative curvature features of the space, in the current situation, we would only need the simple geometric input from Fact~\ref{t5.2}.

 The following lemma gives a distance estimate between regularized grid maps on different flats, which is used several times in the remaining proofs of this section:
\begin{lemma}\label{t5.5}
    Suppose $\kappa>1$, $C>0$, $\epsilon\in(0,\epsilon_0)$, and $\rho>1/\epsilon$ are constants, where $\epsilon_0\in(0,1/100)$ is the constant depending only on $\kappa$ given in Proposition~\ref{t3.5}. Then there exists a constant $C'_\epsilon>0$ such that if $x\in \X$, $D\subseteq \X$, $\phi\in \QIE_{\kappa,C}(D,\X)$, $F_1,\,F_2\in \cF'_{\epsilon,D,\rho}$ are two flats containing $x$, and $x\in U_2^{\epsilon,D,\rho}(F_i)$ for each $i\in \{1,\,2\}$, then for each $y_1\in F_1$ and each $y_2\in F_2$, we have
    \begin{equation*}
    \begin{aligned}
         \kappa^{-1} d(y_1,y_2)-9\kappa\epsilon(d(x,y_1)+d(x,y_2))-C_\epsilon'&\le d\bigl(\tphi_1(y_1),\tphi_2(y_2)\bigr)\\
         &\le \kappa d(y_1,y_2)+9\kappa\epsilon(d(x,y_1)+d(x,y_2))+C_\epsilon',
    \end{aligned}
    \end{equation*}
     where $\tphi_i$ denotes the regularized grid map of $\phi$ on $F_i$ from Proposition~\ref{t3.5} for each $i\in\{1,\,2\}$.
\end{lemma}
\begin{proof}
    This proof is an extension of the fact that the regularized grid map is close to the original map $\phi$.

    Let $\kappa>1$, $C>0$, $\epsilon\in(0,\epsilon_0)$, and $\rho>1/\epsilon$ be constants. Suppose $x\in \X$, $D\subseteq \X$, $\phi\in \QIE_{\kappa,C}(D,\X)$, and $F_1,\,F_2\in \cF'_{\epsilon,D,\rho}$ are two flats containing $x$ such that $x\in U_2^{\epsilon,D,\rho}(F_i)$ for each $i\in \{1,\,2\}$. Suppose $y_1\in F_1$ and $y_2\in F_2$. Then for each $i\in\{1,\,2\}$, by \eqref{e3.3y} in Proposition~\ref{t3.5}~(iv), there exists a point $y_i'\in U_1(F_i)$ such that 
    \begin{equation}\label{e5.1a}
        d(y_i,y_i')\le \rho+\epsilon d(y_i,x).
    \end{equation}
   By Definition~\ref{d3.3}, there exists a point $y_i''\in D\cap F_i$ such that $d(y_i',y_i'')\le 2n_{\epsilon}$ and $\tphi_{i}(y_i')=\phi(y_i'')$. Thus, by \eqref{e5.1a},
    \begin{equation}\label{e5.1b}
    \begin{aligned}
        d(y_1'',y_2'')&\le d(y_1'',y_1')+d(y_2'',y_2')+d(y_1',y_1)+d(y_2',y_2)+d(y_1,y_2)\\
        &\le2\rho+4n_\epsilon+\epsilon (d(y_1,x)+d(y_2,x))+d(y_1,y_2).
    \end{aligned}
    \end{equation}
    Thus, by \eqref{e5.1a}, \eqref{e5.1b}, and the property of $\tphi_i$ from Proposition~\ref{t3.5}~(iv), we have
    \begin{align*}
        &d\bigl(\tphi_1(y_1),\tphi_2(y_2)\bigr)\\
        &\qquad\le d\bigl(\tphi_1(y_1),\phi(y_1'')\bigr)+d\bigl(\tphi_2(y_2),\phi(y_2'')\bigr)+d(\phi(y_1''),\phi(y_2''))\\
        &\qquad\le 2\kappa(d(y_1,y_1')+d(y_2,y_2'))+6\kappa\epsilon(d(y_1,x)+d(y_2,x))+12\kappa\rho +5C+16\kappa n_\epsilon+\kappa d(y_1'',y_2'')\\
        &\qquad\le 9\kappa\epsilon(d(y_1,x)+d(y_2,x))+\kappa d(y_1,y_2)+C'_\epsilon,
    \end{align*}
    where $C'_\epsilon\=18\kappa\rho+20\kappa n_\epsilon+5C$ depends only on $\kappa$, $C$, $\epsilon$, and $\rho$. The other half of the inequality can be obtained analogously from \eqref{e5.1a} and \eqref{e5.1b}.
\end{proof}

\begin{proof}[\bf Proof of Lemma~\ref{t5.1}]
    Let $\kappa>1$, $C>0$, and $\epsilon\in(0,\epsilon')$ be constants, where $\epsilon'\in(0,1/100)$ is the constant depending only on $\kappa$ given in Proposition~\ref{t4.5}. Let $D$ be a random subset of $\X$ and $\phi\in \QIE_{\kappa,C}(D,\X)$. 
    
    Suppose $L$ is a hyperplane passing through $e$, and $F_1,\,F_2$ are two flats in $\cF$ both containing $L$. Let $\rho_0>1/\epsilon$ be the constant given in Proposition~\ref{t5.0}. Since $\phi_0$ is independent of the choice of $\rho\ge \rho_0$, we set $\rho\=\rho_0$. Without loss of generality we assume $\cF\subseteq \cF'_{\epsilon,D',\rho}$. Let $L_1'$ and $L_2'$ be the singular geodesic rays emanating from $\phi'(e)$ 
    such that 
    \begin{equation}\label{e5.5a0}
        \tphi'_{F_1,\epsilon}(L)\sim_{\lambda_1\epsilon,\phi'(e),R_0}L_1' \quad\text{ and }\quad \tphi'_{F_2,\epsilon}(L)\sim_{\lambda_1\epsilon,\phi'(e),R_0}L_2'
    \end{equation} given by Proposition~\ref{t4.5} with $\phi$ set to be $\phi'$, respectively. 
    
    We now use distance estimates to show that $L_1'=L_2'$, as shown in Figure~\ref{fig5.1}. Write $C'\=C+\kappa d(e,D)$. Let $R_t\ge R_0$ be a constant whose exact choice will be determined later. Let $y_1$ be an arbitrary point in $L_1'$ with $d(y_1,\phi'(e))\ge 2R_t$. Let $x$ be a point in $L$ such that 
    \begin{equation}\label{e5.5a}
        d \bigl( \tphi'_{F_1,\epsilon}(x),y_1 \bigr)\le \lambda_1 \epsilon d(y_1,\phi'(e))
    \end{equation}
given by Proposition~\ref{t4.5}. 
We first show that for suitable choices of $R_t$ and $\epsilon$,
\begin{equation}\label{e5.5b}
    d(x,e)\le 3\kappa d(y_1,\phi'(e)). 
\end{equation}
We argue by contradiction and assume that $d(x,e)> 3\kappa d(y_1,\phi'(e))$. In the remainder of the proof, we take $\epsilon\in\bigl(0,\min\bigl\{\epsilon',\,1/(10\lambda_1),\,1/\bigl(10^3\kappa^2\bigr)\bigr\}\bigr)$. Then since $\tphi'_{F_1,\epsilon}$ is a $(2\kappa,6\kappa\epsilon, 6\kappa\rho_0+2C'+8\kappa n_\epsilon )$-graded quasi-isometric embedding,
\begin{equation*}
\begin{aligned}
    (1+\lambda_1\epsilon)d(y_1,\phi'(e))
    &\ge d(y_1,\phi'(e)) + d \bigl( \tphi'_{F_1,\epsilon}(x),y_1 \bigr)
    \ge d \bigl( \tphi'_{F_1,\epsilon}(x),\phi'(e)\bigr)\\
    &\ge d(x,e)/(2\kappa)-6\kappa\epsilon d(x,e)- 6\kappa\rho_0-2C'-8\kappa n_\epsilon\\
    &>1.2d(y_1,\phi'(e))- 6\kappa\rho_0-2C'-8\kappa n_\epsilon,
\end{aligned}
\end{equation*}
which is a contradiction if $R_t>10(6\kappa\rho_0+2C'+8\kappa n_\epsilon)$. Thus for such $R_t$ and $\epsilon$, \eqref{e5.5b} holds.

By \eqref{e5.5b} and Lemma~\ref{t5.5}, we have
    \begin{equation}\label{e5.5z}
        d\bigl(\tphi'_{F_1,\epsilon}(x),\tphi'_{F_2,\epsilon}(x)\bigr)
        \le 18\kappa\epsilon d(x,e)+C_\epsilon'<100\kappa^2\epsilon d(y_1,\phi'(e))+C'_\epsilon,
    \end{equation}
    where $C'_\epsilon$ is the constant given in Lemma~\ref{t5.5} with $C$ set to be $C'$.
    
    Thus, if $R_t>2C_\epsilon'+R_0$ and $\epsilon\in\bigl(0,\min\bigl\{\epsilon',\,1/(10\lambda_1),\,1\big/\bigl(10^3\kappa^2\bigr)\bigr\}\bigr)$, by \eqref{e5.5a} and \eqref{e5.5z}, 
    \begin{equation*}
        d\bigl(\tphi'_{F_2,\epsilon}(x),\phi'(e)\bigr)
        \ge d(y_1,\phi'(e))-d\bigl(\tphi'_{F_1,\epsilon}(x),y_1 \bigr)-d\bigl(\tphi'_{F_1,\epsilon}(x),\tphi'_{F_2,\epsilon}(x)\bigr)
        > R_t\ge R_0
    \end{equation*}
    and analogously $d\bigl(\tphi'_{F_2,\epsilon}(x),\phi'(e)\bigr)\le 2d(y_1,\phi'(e))$. Hence by \eqref{e5.5a0},
    \begin{equation*}
        d\bigl(\tphi'_{F_2,\epsilon}(x),L_2'\bigr)
        \le \lambda_1\epsilon d \bigl( \tphi'_{F_2,\epsilon}(x),\phi'(e) \bigr)
        \le 2 \lambda_1\epsilon d(y_1,\phi'(e)).
    \end{equation*}
    Combining these estimates, for $R_t>10(6\kappa\rho_0+2C'+8\kappa n_\epsilon)+2C_\epsilon' + R_0$ and $\lambda_1'\coloneqq 100\kappa^2+4\lambda_1$, we have
    \begin{equation*}
        d(y_1,L_2')
        \le d\bigl(\tphi'_{F_2,\epsilon}(x),L_2'\bigr)+d\bigl(\tphi'_{F_1,\epsilon}(x),\tphi'_{F_2,\epsilon}(x)\bigr)+d\bigl(\tphi'_{F_1,\epsilon}(x),y_1\bigr)
        < (\lambda_1'-\lambda_1)\epsilon d(y_1,\phi'(e))+C'_\epsilon.
    \end{equation*}
    \begin{figure}
        \centering
        \begin{tikzpicture}[>=stealth]

\draw (-2,1.5) -- (4,1.5) node[right] {$L_1'$};
\draw (-2,0) -- (4,0) node[right] {$L_2'$};

\coordinate (y1) at (3.2,1.5);
\coordinate (x1) at (3.0,1.7);
\coordinate (x2) at (2.7,0.3);
\coordinate (y2) at (2.7,0.0);

\filldraw (y1) circle (2pt) node[above right] {$y_1$};
\filldraw (x1) circle (2pt) node[above left] {$\tphi'_{F_1,\epsilon}(x)$};
\filldraw (x2) circle (2pt) node[above left] {$\tphi'_{F_2,\epsilon}(x)$};

\begin{scope}[dashed]

    \draw (x2) -- (y2);

    \draw (x1) -- (x2);

    \draw (x1) -- (y1);
\end{scope}

\end{tikzpicture}
        \caption{}
        \label{fig5.1}
    \end{figure}
    By an analogous argument applied to $d(y_2,L_1')$ for each $y_2\in L_2'$, we deduce $L_1'\sim_{\lambda_1'\epsilon}L_2'$. Thus, for $\epsilon\in (0,\epsilon_4')$ with 
    \begin{equation}\label{e5.3}
        \epsilon_4'\=\min\bigl\{\epsilon',\,1/(10\lambda_1),\,1\big/\bigl(10^3\kappa^2\bigr),\,1/\lambda_1'\bigr\},
    \end{equation}
    $L_1'=L_2'$ follows from Fact~\ref{t5.2}.
\end{proof}
\begin{rem}\label{r5.7}
    By Lemma~\ref{t5.1}, we now view $\phi_0$ as a well-defined map from $\cU$ to $\partial T_1\cup\partial T_2$ given by $\phi_0(L)=L'$, where $L'$ is the equivalence class of singular geodesic rays given in Lemma~\ref{t5.1}.
\end{rem}

\subsection{Continuity of \texorpdfstring{$\phi_0$}{ϕ₀}}
    In this subsection, we verify the H\"older regularity of $\phi_0$ on $\cU$, construct its extension $\overline{\phi}_0$ to $\partial T_1\cup\partial T_2$, and prove that $\overline{\phi}_0$ maps each $\partial T_i$ into itself up to a permutation.
    
    We first show that $\phi_0$ maps each $\partial T_i\cap \cU$ into $\partial T_i$ up to a permutation. It is based on the fact that the regularized grid map $\tphi_{F,\epsilon}$ sends chambers to chambers (cf.~Proposition~\ref{t4.9a}).
\begin{lemma}\label{t5.6}
    Suppose $\kappa>1$, $C>0$, and $\epsilon\in(0,\min\{\epsilon_2',\,\epsilon_4'\})$, where $\epsilon_2'\in(0,1/100)$ is the constant depending only on $\kappa$ given in Proposition~\ref{t4.9a} and $\epsilon_4'\in(0,1/100)$ is the constant depending only on $\kappa$ given in Lemma~\ref{t5.1}. Suppose $D$ is a random subset of $\X$. Then almost surely, for each $\phi\in \QIE_{\kappa,C}(D,\X)$, 
    there exists a permutation $\sigma\in\{\id,\,(1,2)\}$ such that $\phi_0(\partial T_i\cap \cU)\subseteq \partial T_{\sigma(i)}$ for each $i\in\{1,\,2\}$. Here $\phi_0$ is the map on $\cU$ defined above.
\end{lemma}
   
\begin{proof}
    Let $\phi\in \QIE_{\kappa,C}(D,\X)$ and $\rho_0>1/\epsilon$ be the constant given in Proposition~\ref{t5.0}. Fix $\rho\ge \rho_0$. Without loss of generality we assume $\cF\subseteq \cF'_{\epsilon,D',\rho}$.
    
    Let $L_1\in\partial T_1\cap \cU$ and $L_2\in\partial T_2\cap \cU$ be two singular geodesic rays emanating from $e$ and define the chamber $\fC\=L_1\times L_2$. By our choice of $\cF$, there exists a flat $F\in \cF$ such that $\fC\subseteq F$. Then by Proposition~\ref{t5.0}, almost surely $e\in U_3^{\epsilon,D',\rho}(F)$. Thus, by Proposition~\ref{t4.9a}, there exist a chamber $\fC'=L_1'\times L_2'$, a constant $R_2'>0$, and $\sigma\in\{\id,\,(1,2)\}$ such that $\tphi'_{F,\epsilon}(\fC)\sim_{\lambda_1\epsilon}\fC'$ and 
    \begin{equation}\label{e5.8a}
        \tphi'_{F,\epsilon}(L_i)\smallsetminus B(\phi'(e),R_2')\subseteq L_{\sigma(i)}'[\lambda_1\epsilon]\smallsetminus B(\phi'(e),R_2')
    \end{equation} for each $i\in\{1,\,2\}$. Recall that $\tphi'_{F,\epsilon}(L_i)\sim_{\lambda_1\epsilon,\phi'(e),R_0}\phi_0(L_i)$. By Definition~\ref{d2.5}, this implies 
    \begin{equation}\label{e5.8b}
        \tphi'_{F,\epsilon}(L_i)\smallsetminus B(\phi'(e),R_0)\subseteq \phi_0(L_i)[\lambda_1\epsilon]\smallsetminus B(\phi'(e),R_0).
    \end{equation}
    
    We argue by contradiction and assume that $\phi_0(L_i)\neq L_{\sigma(i)}'$. Then since $\phi_0(L_i)$ and $L_{\sigma(i)}'$ are both singular geodesic rays, we have $d \bigl(L_{\sigma(i)}'\smallsetminus B(\phi'(e),R),\phi_0(L_i)\smallsetminus B(\phi'(e),R)\bigr)\ge R-\Absbig{\phi_0(L_i)\cap L_{\sigma(i)}'}-1$ for each $R>0$. Since $5\lambda_1\epsilon_4'<1$ (cf.~\eqref{e5.3}), this implies $\phi_0(L_i)[\lambda_1\epsilon]\cap L_{\sigma(i)}'[\lambda_1\epsilon]$ is bounded by direct calculation, which contradicts \eqref{e5.8a} and \eqref{e5.8b}. Thus, we have $\phi_0(L_i)=L'_{\sigma(i)}$. That is, for each pair of $(L_1,L_2)$, we obtain a permutation $\sigma\in\{\id,\,(1,2)\}$ such that $\phi_0(L_i)\in \partial T_{\sigma(i)}$ for each $i\in \{1,\,2\}$. Since $\phi_0(L_i)$ does not depend on the choice of the other factor $L_{3-i}$, the choice of $\sigma$ does not depend on the pair $(L_1,L_2)$.  
\end{proof}
 We assume $\sigma=\id$ in Lemma~\ref{t5.6} for the remainder of the article.
Recall that $\tphi'_{F,\epsilon}(L)\sim_{\lambda_1\epsilon,\phi'(e),R_0}\phi_0(L)$ as given in Proposition~\ref{t4.5}. Using Lemma~\ref{t5.5}, we now reformulate this relation to verify the bi-H\"older continuity of $\phi_0$, as stated in the following lemma.
\begin{lemma}[Bi-H\"older continuity of $\phi_0$]\label{t5.4}
         Given $\kappa>1$, there exist constants $T>1$ and $\epsilon_4\in (0,1/100)$ such that if $D$ is a random subset of $\X$, then almost surely for each $C>0$, each $\epsilon\in (0,\epsilon_4)$, and each $\phi\in \QIE_{\kappa,C}(D,\X)$, there exists a constant $c>0$ such that if $i\in \{1,\,2\}$ and $L_1,L_2\in \cU\cap\partial T_i$, then
        \begin{equation}\label{e5.4m}
        c^{-1}d_{e_i}(L_1,L_2)^{T}\le d_{\pi_i(\phi'(e))}(\phi_0(L_1),\phi_0(L_2))\le cd_{e_i}(L_1,L_2)^{1/T}.
        \end{equation}
        Here $\phi'\in \QIE_{\kappa,C+\kappa d(e,D)}(D',\X)$ is the map defined at the beginning of Subsection~\ref{s5.2}, and the metrics $d_{e_i}$ and $d_{\pi_i(\phi'(e))}$ are defined in \eqref{e5.0}.
    \end{lemma}
    \begin{proof}
    
       We only prove the statement for $i=1$. Let $C>0$ be a constant, $\lambda_1>0$ be the constant given in Proposition~\ref{t4.5}, and $\epsilon\in(0,\min\{\epsilon_2',\,\epsilon_4'\})$ be a constant, where $\epsilon_2',\,\epsilon_4'\in(0,1/100)$ are the constants depending only on $\kappa$ from Proposition~\ref{t4.9a} and Lemma~\ref{t5.1}, respectively. Define $C'\=C+\kappa d(e,D)$. Let $\rho_0>1/\epsilon$ and $R_0>0$ be the constants given in Propositions~\ref{t5.0} and~\ref{t4.5}, respectively. Suppose $\rho\ge\rho_0$. Then without loss of generality we assume $\cF\subseteq \cF'_{\epsilon,D',\rho}$ and $e\in U_3^{\epsilon,D',\rho}(F)$ for each $F\in\cF$.

        Let $\phi\in \QIE_{\kappa,C}(D,\X)$ and $\phi_0$ be the map defined on $\cU$ in Remark~\ref{r5.7}. Let $L_1$ and $L_2$ be two singular geodesic rays in $\partial T_1$. For each $i\in\{1,\,2\}$, let $L_i'$ be the singular geodesic ray (from Proposition~\ref{t4.5}) emanating from $\phi'(e)=(e_1',e_2')$ such that 
        \begin{equation}\label{e5.4d}
            \tphi'_{F,\epsilon}(L_i)\sim_{\lambda_1\epsilon,\phi'(e),R_0}L_i'.
        \end{equation}
        Recall that $L_1'$ and $L_2'$ are both geodesic rays in $T_1$ (cf.~Lemma~\ref{t5.6}). 
        
        We first consider the first inequality of \eqref{e5.4m}. Let $x_0\in T_1$ be the last common vertex of $L_1$ and $ L_2$ and write $u_0=(x_0,e_2)$. Then $d(u_0,e)=d(x_0,e_1)=\abs{L_1\wedge L_2}-1$. Let $L_3$ be an arbitrary singular geodesic ray in $\cU$ emanating from $e_2$. Let $F_1$ and $F_2$ be flats in $\cF$ containing the chambers $L_1\times L_3$ and $L_2\times L_3$, respectively. Let $\tphi_1$ and $\tphi_2$ denote the regularized grid maps of $\phi'$ defined in Proposition~\ref{t3.5} on flats $F_1$ and $F_2$, respectively. Then by Lemma~\ref{t5.5} with $C$ set to be $C'$, we have 
        \begin{equation}\label{e5.4g}
            d\bigl(\tphi_1(u_0),\tphi_2(u_0)\bigr)\le 18\kappa \epsilon d(u_0,e)+C_\epsilon'.
        \end{equation}
         Write $R_t\=R_0+4(6\kappa \rho +2C'+8\kappa n_\epsilon)$. If $d(u_0,e)\ge 3\kappa R_t$, since $\tphi_i$ is a $(2\kappa,6\kappa\epsilon,6\kappa \rho +2C'+8\kappa n_\epsilon)$-graded quasi-isometric embedding and $180\kappa^2\epsilon_2'<1$
         , we have 
         \begin{equation*} 
             d\bigl(\tphi_i(u_0),\phi'(e)\bigr)\ge d(u_0,e)/(2\kappa)-6\kappa \epsilon d (u_0,e)-6\kappa \rho -2C'-8\kappa n_\epsilon> R_0
         \end{equation*}
         for each $i\in\{1,\,2\}$. Thus by \eqref{e5.4d}, Definition~\ref{d2.5}, and our assumption that $d(u_0,e)\ge 3\kappa R_t$, there exist $(y_1,e_2')\in L_1'$ and $(y_2,e_2')\in L_2'$ such that for each $i\in\{1,\,2\}$,
         \begin{equation}\label{e5.4h}
         \begin{aligned}
             d\bigl(\tphi_i(u_0),(y_i,e_2')\bigr)
             &\le \lambda_1\epsilon d\bigl(\tphi_i(u_0),\phi'(e)\bigr)\\
             &\le\lambda_1\epsilon(2\kappa d(u_0,e)+6\kappa\epsilon d(u_0,e)+6\kappa \rho +2C'+8\kappa n_\epsilon) 
             <3\lambda_1\kappa\epsilon d(u_0,e).
        \end{aligned}
         \end{equation}

        Define $y_3\in T_1$ to be the \emph{confluent} of $y_1$ and $y_2$ with respect to $e_1'$, that is, the last common vertex on the geodesic segments $[e_1',y_1]$ and $[e_1',y_2]$. Then by \eqref{e5.4h} and \eqref{e5.4g},
        \begin{equation}\label{e5.4i}
        \begin{aligned}
            d(y_1,y_3) 
            \le d(y_1,y_2) 
            &\le d\bigl(\tphi_1(u_0),(y_1,e_2')\bigr)+d\bigl(\tphi_2(u_0),(y_2,e_2')\bigr)+d\bigl(\tphi_1(u_0),\tphi_2(u_0)\bigr)\\
            &\le 6\lambda_1\kappa\epsilon d(u_0,e)+18\kappa \epsilon d(u_0,e)+C_\epsilon'. 
        \end{aligned}   
        \end{equation}
        If $\epsilon<\min\bigl\{1 /(36\lambda_1\kappa),1 \big/\bigl(72\kappa^2\bigr)\bigr\}$, by \eqref{e5.4h} and the assumption that $d(u_0,e)\ge 3\kappa R_t$, we have
        \begin{equation}\label{e5.4j}
        \begin{aligned}
            d\bigl(e_1',y_1\bigr)&\ge d\bigl(\tphi_1(u_0),\phi'(e)\bigr)-d\bigl(\tphi_1(u_0),\bigl(y_1,e_2')\bigr)\\
         &\ge d(u_0,e)/(2\kappa)-6\kappa\epsilon d(u_0,e)-(6\kappa \rho +2C'+8\kappa n_\epsilon)-3\lambda_1\kappa\epsilon d(u_0,e)\\
         &\ge d(u_0,e)/(3\kappa)-3\lambda_1\kappa\epsilon d(u_0,e)
         >d(u_0,e)/(4\kappa).
            \end{aligned}
        \end{equation}
        Therefore, by \eqref{e5.4i} and \eqref{e5.4j}, if $d(u_0,e)\ge 3\kappa R_t$, we have
        \begin{align}
            \abs{L_1'\wedge L_2'}
             &=d(y_3,e_1')+1
             \ge d(e_1',y_1)-d(y_1,y_3)+1\ge (1/(4\kappa)-6\lambda_1\kappa \epsilon-18\kappa\epsilon)d(u_0,e)-C'_\epsilon+1\notag\\
             &\ge (1/(4\kappa)-6\lambda_1\kappa \epsilon-18\kappa\epsilon)\abs{L_1\wedge L_2}-C'_\epsilon.\label{e5.4b} 
        \end{align}
        If $d(u_0,e)<3\kappa R_t$, then $\abs{L_1\wedge L_2}=d(u_0,e)+1<3\kappa R_t+1$. Thus the right-hand side of \eqref{e5.4b} is at most $3\kappa R_t+1$. So in both cases, 
        \begin{equation}\label{e5.4l}
        \abs{L_1'\wedge L_2'}=d(y_3,e_1')+1
        \ge (1/(4\kappa)-6\lambda_1\kappa \epsilon-18\kappa\epsilon)\abs{L_1\wedge L_2}-C'_\epsilon-3\kappa R_t.
        \end{equation}

        For the other inequality of \eqref{e5.4m}, we use an analogous argument. Let $z_0\in T_1$ be the last common vertex of $L_1'$ and $L_2'$ farthest from $e_1'$ and write $w_0\=(z_0,e_2')$. Then $d(z_0,e_1')=d(w_0,\phi'(e))=\abs{L_1'\wedge L_2'}-1$. 
        
        We assume first that $d\left(w_0,\phi'(e)\right)>  R_t$. Then by \eqref{e5.4d} there exist points $x_1\in L_1$ and $x_2\in L_2$ such that 
        \begin{equation}\label{e5.4k}
        d\bigl(\tphi_1(x_1),w_0\bigr)\le \lambda_1\epsilon d(w_0,\phi'(e))\quad\text{ and }\quad  d\bigl(\tphi_2(x_2),w_0\bigr)\le \lambda_1\epsilon d(w_0,\phi'(e)).
        \end{equation}
        
        Suppose $d(w_0,\phi'(e))\ge 3\kappa d(x_i,e) $ for some $i\in \{1,\,2\}$. Then since $\tphi_i$ is a graded quasi-isometric embedding, by \eqref{e5.4k} and the assumption that $d\left(w_0,\phi'(e)\right)>  R_t$, for $\epsilon<\min\{1/(20\lambda_1),1/60\}$,
        \begin{equation*}
        \begin{aligned}
            d\bigl(\tphi_i(x_i),\phi'(e)\bigr)&\le 2\kappa d(x_i,e)+6\kappa\epsilon d(x_i,e)+6\kappa \rho_0+8\kappa n_\epsilon+2C'\\
            &\le 0.95d(w_0,\phi'(e))<d(w_0,\phi'(e))-d\bigl(\tphi_i(x_i),w_0\bigr),
        \end{aligned}
        \end{equation*}
        which is clearly a contradiction. This, together with a similar argument as that for \eqref{e5.5b}, shows that for each $i\in\{1,\,2\}$,     
        \begin{equation}\label{e5.4a}
             d(w_0,\phi'(e))/(3\kappa)<d(x_i,e)< 3\kappa d(w_0,\phi'(e)).
        \end{equation}
        
        By Lemma~\ref{t5.5}, \eqref{e5.4k}, and \eqref{e5.4a}, we have
        \begin{equation}\label{e5.4e}
        \begin{aligned}
            d(x_1,x_2)&\le \kappa \bigl(d\bigl(\tphi_1(x_1),w_0\bigr)+d\bigl(\tphi_2(x_2),w_0\bigr)\bigr)+9\kappa^2\epsilon (d(x_1,e)+d(x_2,e))+\kappa C_\epsilon'\\
            &\le \bigl(2\lambda_1\kappa+100\kappa^3\bigr)\epsilon d(w_0,\phi'(e))+\kappa C_\epsilon'.
        \end{aligned}
        \end{equation}
        Combining \eqref{e5.4a} and \eqref{e5.4e} with the fact that $2\abs{L_1\wedge L_2}\ge d(x_1,e)+d(x_2,e)-d(x_1,x_2)$, we use a similar argument as that for \eqref{e5.4l} (with \eqref{e5.4a} and \eqref{e5.4e} replacing \eqref{e5.4i} and \eqref{e5.4h}) to get 
        \begin{equation}\label{e5.4c}
            \abs{L_1\wedge L_2}\ge \bigl(1/(4\kappa)-\bigl(2\lambda_1\kappa+100\kappa^3\bigr)\epsilon \bigr)\abs{L_1'\wedge L_2'}-\kappa C_\epsilon'-R_t.
        \end{equation} 
        Combining \eqref{e5.4l} and \eqref{e5.4c} with~\eqref{e5.0}, we obtain the statement for 
        \begin{equation}\label{e5.4f}
            T\coloneqq 8\kappa,\quad 
            \epsilon_4\coloneqq\min\bigl\{\epsilon',\,\epsilon_4',\,\epsilon_2',\,1 \big/ \bigl(20\lambda_1\kappa^2+10^4\kappa^4\bigr)\bigr\},\quad \text{ and } 
            c\coloneqq (q+1)^{\kappa C_\epsilon'+3\kappa R_t}.\qedhere
        \end{equation}
    \end{proof}
Since $\partial T_i$ is closed for each $i\in \{1,\,2\}$ and $\cU$ is dense in both $\partial T_1$ and $\partial T_2$, we immediately deduce the following proposition.
\begin{prop}\label{t5.7}
    In the setting of Lemma~\ref{t5.4}, almost surely there exists a unique bi-H\"older continuous map $\overline{\phi}_0\: \partial T_1\cup \partial T_2 \to \partial T_1\cup \partial T_2$ that is an extension of $\phi_0$ and maps each tree boundary into a tree boundary.
\end{prop}
Equivalently, we have defined a map $\wx\to\wx$ on the product boundary that sends $(\xi,\eta)\in\wx$ to $(\overline{\phi}_0(\xi),\overline{\phi}_0(\eta))$. That is, we have obtained a boundary map that is H\"older continuous and factors through the projections $\wx\to \partial T_i$. 

\begin{rem}\label{r5.11}
    In the setting of Lemma~\ref{t5.4}, suppose $F$ and $F'$ are two flats such that $\tphi_{F,\epsilon}(F)\subseteq \Nb(F',\Delta_1)$ (cf.~Proposition~\ref{t4.8}), $i\in\{1,\,2\}$, $L\subseteq F$ is a hyperplane in $T_i$, and $x\in U_3(F)\cap L$. Let $L'$ be the hyperplane given in Proposition~\ref{t4.5} such that $\tphi_{F,\epsilon}(L)\sim_{\lambda_1\epsilon,\phi_{F,\epsilon}(x),R_0}L'$. Then  $L'\smallsetminus B(\phi_{F,\epsilon}(x),R_0)\subseteq (B(F',\Delta_1))[\lambda_1\epsilon]$. Since $\epsilon<1/(2\lambda_1)$, this implies $\pi_i(F')=\pi_i(L')$.
\end{rem}

\section{The pullback factor maps}\label{s6}

Recall that without loss of generality, we assume that $\sigma=\operatorname{id}$ in the setting of Lemma~\ref{t5.6}, i.e., $\overline{\phi}_0$ maps the boundary $\partial T_i$ of each factor tree into itself. 

In this section, we complete the proofs of Theorems~\ref{t1.1} and~\ref{t1.2} by constructing factor maps which are pullback maps on the boundary of each tree.

We first fix a choice of the constant $\Delta_1$ given in Proposition~\ref{t4.8}. Given constants $\kappa>1$, $C>0$, $\epsilon\in(0,\epsilon_4)$, and $\rho>1/\epsilon$, denote by $S(\kappa,C,\epsilon,\rho)$ the set of constants $\Delta_1>0$ satisfying \eqref{e4.8t} given in Proposition~\ref{t4.8} (see \eqref{e5.4f} for the fact $\epsilon_4\le\epsilon_2'$). Then for $1/\epsilon<\rho_1<\rho_2$ we have $S(\kappa,C,\epsilon,\rho_2)\subseteq S(\kappa,C,\epsilon,\rho_1)$. Thus, for the rest of the article, we fix the choice of $\Delta_1$ to be 
\begin{equation}\label{e6.00a}
    \Delta_1\=\inf \{ \Delta +1 : \Delta\in S(\kappa,C,\epsilon,\rho) \}.
\end{equation} 
The choice of $\Delta_1$ fixed this way is nondecreasing with respect to $\rho$ and thus measurable in $\rho$.

\subsection{Finding the factor maps}\label{s6.1}
 By Proposition~\ref{t4.8}, each flat is almost surely mapped near another flat. This suggests that we may treat points in $\X$ as intersections of flats. We first prove a geometric lemma.
 \begin{lemma}\label{t6.0}
     Suppose $D$ is a random subset of $\X$, $\kappa>1$, $C>0$, $\epsilon\in(0,\epsilon_4)$, and $\rho>1/\epsilon$ are constants, where $\epsilon_4\in(0,1/100)$ is the constant depending only on $\kappa$ given in Lemma~\ref{t5.4}. Let $\Delta_1>0$ be the constant fixed in \eqref{e6.00a}. Then there exists a constant $\Delta_2>0$ such that the following holds almost surely:

     \smallskip
 
     Suppose $\phi\in \QIE_{\kappa,C}(D,\X)$, $x\in D$, and $F_1$ and $F_2$ are two flats with $F_1\cap F_2=\{x\}$ and $x\in U_3(F_1)\cap U_3(F_2)$. For each $i\in\{1,\,2\}$, let $F_i'$ be the flat given in Proposition~\ref{t4.8} such that
    $\phi(D\cap U_1(F_i))\subseteq \Nb(F_i',\Delta_1)$. Set $Q\=\{y\in F_1':d(y,F_2')=d(F_1',F_2')\}$. 
    Then there exists a point $x'\in Q$ such that
    \begin{equation}\label{e6.0d}
        \phi(x)\in \Nb(F_1',\Delta_1)\cap \Nb(F_2',\Delta_1)\subseteq B(Q,2\Delta_1)\subseteq B(x',\Delta_2). 
    \end{equation}
 \end{lemma}
 \begin{proof}
     We first obtain an upper bound estimate for the size of $Q$.
     
     \smallskip
     
     \textit{Claim.} 
     There exists a constant $\Delta_2>0$ such that for each $i\in\{1,\,2\}$, we have $\abs{\pi_i(Q)}\le \Delta_2-2\Delta_1$.

     \smallskip
     
     If the claim is verified, By Lemma~\ref{t4.3}~(ii),
     \begin{equation*}
         \Nb(F_1',\Delta_1)\cap \Nb(F_2',\Delta_1)\subseteq B(Q,2\Delta_1).
     \end{equation*}
     By Lemma~\ref{t4.3}~(i), $Q$ is a polyhedron and $Q=\pi_1(Q)\times\pi_2(Q)$. Then for each $x'\in Q$ and each $i\in\{1,\,2\}$, $\pi_i(Q)\subseteq B(\pi_i(x'),\abs{\pi_i(Q)})$. Thus, $Q\subseteq B(x',\max\{\abs{\pi_1(Q)},\,\abs{\pi_2(Q)}\})$ and we verify \eqref{e6.0d}.

     \smallskip
     
     \textit{Proof of Claim.}
     By Lemma~\ref{t4.3}~(i), we only need to prove $\abs{\pi_i(F_1')\cap\pi_i (F_2')}\le \Delta_2-2\Delta_1$. Now fix a choice of $i\in\{1,\,2\}$. 
     For each $j\in\{1,\,2\}$, let $L_j\subseteq F_j$ be the hyperplane in $T_i$ passing through $x$, and $L_j'$ be the hyperplane given in \eqref{e4.5t1} in Proposition~\ref{t4.5} such that 
     \begin{equation}\label{e6.0a}
         \tphi_{F,\epsilon}(L_j)\sim_{\lambda_1\epsilon,\phi(x),R_0}L'_j.
     \end{equation}
     Then by  Propositions~\ref{t4.5} and \ref{t5.7}, we have $\pi_{3-i}(L_1')=\pi_{3-i}(\phi(x))=\pi_{3-i}(L_2')$ and thus 
     \begin{equation}\label{e6.0a2}
         \abs{L_1'\cap L_2'}=\abs{\pi_i(L_1')\cap \pi_i(L_2')}.
     \end{equation}
     Let $z$ be a point in $L_1'\cap L_2'$ and write $R_t\=R_0+4\kappa (6\kappa\rho+2C+8\kappa n_\epsilon)$. Suppose $d(z,\phi(x))>R_t$. Then by \eqref{e6.0a} for each $j\in\{1,\,2\}$ there exists a point $x_j\in L_j$ such that 
     \begin{equation}\label{e6.0a1}
         d\bigl(\tphi_{F_j,\epsilon}(x_j),z\bigr)\le\lambda_1\epsilon d(z,\phi(x)).
     \end{equation}
     Recall that $\tphi_{F_j,\epsilon}$ is a $(2\kappa,6\kappa\epsilon, 6\kappa\rho+2C+8\kappa n_\epsilon)$-graded quasi-isometric embedding. Since $\epsilon<\epsilon_4\le 1\big/\bigl(20\lambda_1\kappa^2+10^4\kappa^4\bigr)$ (see \eqref{e5.4f}), we can use a similar argument as that for \eqref{e5.4a} to get
     \begin{equation}\label{e6.0b1}
         d(z,\phi(x))/(3\kappa)<d(x_j,x)<3\kappa d(z,\phi(x)).
     \end{equation} Thus, since $F_1\cap F_2=\{x\}$, we have
     \begin{equation}\label{e6.0b}
         d(x_1,x_2)=d(x_1,x)+d(x,x_2)>2d(z,\phi(x))/(3\kappa).
     \end{equation}

     On the other hand, by Lemma~\ref{t5.5}, \eqref{e6.0a1}, \eqref{e6.0b1}, and $\epsilon< 1\big/\bigl(40\lambda_1\kappa^2+10^4\kappa^4\bigr)$, we have
     \begin{align}
         d(x_1,x_2)&\le \kappa d\bigl(\tphi_{F_1,\epsilon}(x_1),\tphi_{F_2,\epsilon}(x_2)\bigr)+9\kappa^2\epsilon ((d(x_1,x)+d(x_2,x))+\kappa C_\epsilon'\notag\\
         &\le \bigl(2\kappa \lambda_1+54\kappa^3\bigr)\epsilon d(z,\phi(x))+\kappa C_\epsilon'\le d(z,\phi(x))/(10\kappa)+\kappa C'_\epsilon.\label{e6.0c} 
     \end{align}
     By \eqref{e6.0b} and \eqref{e6.0c}, for $d(z,\phi(x))>R_t$ we have $d(z,\phi(x))<2\kappa^2 C'_\epsilon$. Then for a general $z'\in L_1'\cap L_2'$, we have $d(z',\phi(x))\le\max\big\{2\kappa^2 C'_\epsilon,\, R_t\big\}<2\kappa^2 C'_\epsilon+R_t$. 
     
     So by Remark~\ref{r5.11} and \eqref{e6.0a2}, we have $\abs{\pi_i(F_1')\cap\pi_i (F_2')}=\abs{\pi_i(L_1')\cap\pi_i (L_2')}=\abs{L_1'\cap L_2'}< 4\kappa^2 C'_\epsilon+2R_t+1$. Thus, denoting
     \begin{equation}\label{e6.00c}
         \Delta_2\=2\Delta_1+4\kappa^2 C_\epsilon'+2R_t+1=2\Delta_1+4\kappa^2 C_\epsilon'+2(R_0+4\kappa (6\kappa\rho+2C+8\kappa n_\epsilon))+1,
     \end{equation}
     we verify our claim. 
 \end{proof}
 
  We now construct the factor maps of $\phi$ on the subset of $D$ where $\rho_0(D,x,\epsilon)$ is bounded uniformly (cf.~Definition~\ref{d5.0}). We first state a fact as a direct consequence of the definition of $\rho_0(D,x,\epsilon)$.
\begin{fact}\label{f7.1}
    Suppose $\epsilon\in (0,1)$ and $R>1/\epsilon$ are constants. Let $x$ be a point in $\X$ and fix $i\in\{0,\,1\}$. Suppose $D_1$ and $D_2$ are two subsets of $\X$ such that $\rho_i(D_j,x,\epsilon)\le R$ for each $j\in\{1,\,2\}$ and $D_1\cap B(x,R+2n_\epsilon)=D_2\cap B(x,R+2n_\epsilon)$. By the definition of $U_0(F)$ (cf.~\eqref{e3.0c}), we have $U_0^{\epsilon,D_j}(F)\cap B(x,R)=U_0^{\epsilon,D_j\cap B(x,R+2n_\epsilon)}(F)\cap B(x,R)$ for each flat $F$ containing $x$ and each $j\in\{1,\,2\}$. Thus by Definition~\ref{d5.0}, we obtain $\rho_i(D_1,x,\epsilon)=\rho_i(D_2,x,\epsilon)$.    
\end{fact}
 \begin{definition}\label{d6.4}
     Suppose $D\subseteq\X$, $\epsilon\in (0,1)$, and $\rho>0$. Define $D'_{\epsilon,\rho}\=\{x\in D:\rho_0(D,x,\epsilon)\le \rho\}$ and $\widehat D'_{\epsilon,\rho}\=\{x\in D:\rho_1(D,x,\epsilon)\le \rho\}$. We call the set $D'_{\epsilon,\rho}$ the \emph{regular-core} of $D$ with parameter $(\epsilon,\rho)$. 
\end{definition}
     Note that $\widehat D'_{\epsilon,\rho}\subseteq D'_{\epsilon,\rho}$ by Definitions~\ref{d5.0} and~\ref{d3.2}.

     \smallskip
     
    Suppose $x\in\X$ and $D$ is a random subset of $\X$. We verify that for each $\epsilon\in(0,1)$ and each $\rho>0$, the events $\bigl\{x\in D'_{\epsilon,\rho}\bigr\}$ and $\bigl\{x\in \hD'_{\epsilon,\rho}\bigr\}$ have positive probability. Fix $\epsilon\in (0,1)$. By \eqref{taildis} in Lemma~\ref{t5.0b}, there exists a constant $\rho_2>1/\epsilon$ such that $\P(\rho_1(D,x,\epsilon)\le \rho_2)>0$. By Fact~\ref{f7.1}, we know that conditional on $\rho_1(D,x,\epsilon)\le \rho_2$, if $B(x,\rho_2+2n_\epsilon)\subseteq D$, then $\rho_1(D,x,\epsilon)=0$. Thus, by the Fortuin--Kasteleyn--Ginibre inequality (see e.g.,~\cite[Chapter~2.2]{grimmett2012percolation}), for each $\rho>0$,
    \begin{equation}\label{e6.00}
    \begin{aligned}
    \P\bigl(x\in D'_{\epsilon,\rho}\bigr)\ge\P\bigl(x\in \hD'_{\epsilon,\rho}\bigr)
    &\ge \P(\rho_1(D,x,\epsilon)\le \rho_2,\,B(x,\rho_2+2n_\epsilon)\subseteq D)\\
    &\ge\P(\rho_1(D,x,\epsilon)\le \rho_2)\P(B(x,\rho_2+2n_\epsilon)\subseteq D)
    >0.
    \end{aligned}
    \end{equation}
    
    Based on the argument above, conditional on $x\in D'_{\epsilon,\rho}$, almost surely we have that if $F$ is a flat containing $x$, then $x\in U_3(F)$ by Lemma~\ref{t5.0a}, Definitions~\ref{d4.1b}, \ref{d4.1a}, and Proposition~\ref{t3.5}. 

    \smallskip
    
    Let $\epsilon\in (0,1)$ be a constant. For each $\rho>0$, each $D\subseteq\X$, and each $\phi\in\QIE(D,\X)$, define the nonempty set $\cP_{\epsilon}(\rho,D,\phi)$ which consists of the pairs $(\psi_1,\psi_2)$ of maps $\psi_1 \:T_1\to T_1$ and $\psi_2\:T_2\to T_2$ which satisfies that for each $i\in\{1,\,2\}$ and each $z_i\in\pi_i\bigl(D_{\epsilon,\rho}'\bigr)$, there exists $z\in D_{\epsilon,\rho}'\cap \pi_i^{-1}(z_i)$ such that $\psi_i(z_i)=\pi_i(\phi(z))$. The following corollary can be derived from Lemma~\ref{t6.0}.
\begin{cor}\label{t6.1}
    In the setting of Lemma~\ref{t6.0}, almost surely for each $\phi\in\QIE_{\kappa,C}(D,\X)$, each pair $(\psi_1,\psi_2)\in\cP_{\epsilon}(\rho,D,\phi)$, and each $x=(x_1,x_2)\in D'_{\epsilon,\rho}$, we have
    \begin{equation}\label{e6.3}
        d((\psi_1(x_1),\psi_2(x_2)),\phi(x))\le N.
    \end{equation}
    Here $N\=6\Delta_2$, where $\Delta_2>0$ is the constant defined in \eqref{e6.00c}.
\end{cor}    

\begin{proof}
   Suppose $\phi$ is a $(\kappa,C)$-quasi-isometric embedding and $(\psi_1,\psi_2)\in\cP_{\epsilon}(\rho,D,\phi)$. By symmetry it suffices to verify that $d(\psi_1(x_1),\pi_1(\phi(x)))\le 3\Delta_2$ for each $x=(x_1,x_2)\in D'_{\epsilon,\rho}$. Since for each $x_1\in \pi_1\bigl(D'_{\epsilon,\rho}\bigr)$, $\psi_1(x_1)=\pi_1(\phi(x'))$ for some $x'=(x_1,x_2')\in D'_{\epsilon,\rho}$, it suffices to prove that if $x=(x_1,x_2)$ and $y=(x_1,y_2)$ are two points in $D'_{\epsilon,\rho}$, then
    \begin{equation*}
        d(\pi_1(\phi(x)),\pi_1(\phi(y)))\le 3\Delta_2.
    \end{equation*}
    
    Let $L_1$ and $L_2$ be two bi-infinite geodesics in $T_1$ such that $L_1\cap L_2=\{x_1\}$. Let $L_3,\,L_4,\,L_5$, and $L_6$ be  bi-infinite geodesics in $T_2$ such that $L_3\cap L_4=\{x_2\}$ and $L_5\cap L_6=\{y_2\}$. Write $(F_1,F_2,F_3,F_4)\=(L_1\times L_3,L_2\times L_4,L_1\times L_5,L_2\times L_6)$. For each $i\in\{1,\,2,\,3,\,4\}$, let $F_i'$ be the flat given in Proposition~\ref{t4.8} such that $\phi(D\cap U_1(F_i))\subseteq \Nb(F_i',\Delta_1)$ and write $L_i'\=\pi_1(F_i')\subseteq T_1$. 
    
    Then since $x\in U_3(F_1)\cap U_3(F_2)$ and $y\in U_3(F_3)\cap U_3(F_4)$ almost surely, by applying Lemma~\ref{t6.0} to $(x,F_1,F_2)$ and $(y,F_3,F_4)$, there exist $x_1'\in F_1'$ and $x_2'\in F_3'$ with $d(x_1',F_2')=d(F_1',F_2')$ and $d(x_2',F_4')=d(F_3',F_4')$ such that
    \begin{equation}\label{e6.3a}
        \phi(x)\in B(x_1',\Delta_2)\quad\text{ and }\quad \phi(y)\in B(x_2',\Delta_2).
    \end{equation}
    Then by Lemma~\ref{t4.3}~(i), we have $\pi_1(x_1')\in \{z\in L_1':d(z,L_2')=d(L_1',L_2')\}$ and $\pi_1(x_2')\in \{z\in L_3':d(z,L_4')=d(L_3',L_4')\}$. Note that by Proposition~\ref{t5.7} and Remark~\ref{r5.11}, $L_i'$ only depends on $\pi_1(F_i)$. Thus, $L_1'=L_3'$ and $L_2'=L_4'$. Then by Lemma~\ref{t4.3}~(i) and \eqref{e6.0d}, we have 
    \[\pi_1(x_1')\in\{z\in L_1':d(z,L_2')=d(L_1',L_2')\}= \pi_1(Q)\subseteq B(\pi_1(x_2'),\Delta_2)\]
    for $Q\=\{y\in F_1':d(y,F_2')=d(F_1',F_2')\}$, 
    and thus $d(\pi_1(x_1'),\pi_1(x_2'))\le \Delta_2$. Therefore, we have 
    $d(\pi_1(\phi(x)),\pi_1(\phi(y))\le 3\Delta_2$ by \eqref{e6.3a}.
\end{proof}
\subsection{Properties of the factor maps}
We now prove, based on Corollary~\ref{t6.1}, that $\psi_1$ and $\psi_2$ are indeed quasi-isometric embeddings for a suitable choice of $\rho$, as stated in Proposition~\ref{t6.2}. The property is based on the following definition.

\begin{definition}\label{d6.1}
    Let $\epsilon\in (0,1)$ be a constant. Suppose $D$ is a random subset of $\X$ and $x\in\X$. By Lemma~\ref{t5.0b} with $\beta$ set to be $10^{-6}\epsilon^8$, the set $S_\epsilon\=\{\rho>0:\P(\rho_1(D,x,\epsilon)>\rho)<\epsilon\}$ is nonempty. Let $\rho_\epsilon'\=\inf_{\rho\in S_\epsilon}\rho$.  Define $\rho_\epsilon\=\max\bigl\{\rho_\epsilon'+1,\,2/\epsilon\bigr\}$. 
\end{definition}
Suppose $D$ is a random subset of $\X$. In the setting of Definition~\ref{d6.1}, consider $D_{\epsilon,\rho}'$ and $\widehat D'_{\epsilon,\rho}$ as defined in Definition~\ref{d6.4}. Then by Definition~\ref{d6.1} and the FKG inequality, for each $\rho\ge \rho_\epsilon$ and each $y\in\X$, we have
   \begin{equation}\label{e6.1}
   \begin{aligned}
        \P\bigl(y\in D_{\epsilon,\rho}'\bigr)\ge\P\bigl(y\in \widehat D_{\epsilon,\rho}'\bigr)&=\P(y\in D,\,\rho_1(D,x,\epsilon)\le \rho)\\&\ge \P\bigl(y\in D\bigr)\P(\rho_1(D,x,\epsilon)\le\rho)\ge p(1-\epsilon).
        \end{aligned}
    \end{equation}

We now state the main result of this subsection.

\begin{prop}\label{t6.2}
    Suppose $\kappa>1$, $C>0$, and $\epsilon\in(0,\epsilon_4)$ are constants, where $\epsilon_4\in(0,1/100)$ is the constant given in Lemma~\ref{t5.4}. Suppose $D$ is a random subset of $\X$. Let $\rho_\epsilon$ be the constant defined in Definition~\ref{d6.1}. Then the following almost surely hold for each $\phi\in\QIE_{\kappa,C}(D,\X)$:

    \begin{enumerate}[label=\rm{(\roman*)}]
    \smallskip
    \item If $(\psi_1,\psi_2)\in\cP_{\epsilon}(\rho_\epsilon,D,\phi)$, then $(\psi_1,\psi_2)\in\cP_{\epsilon}(\rho,D,\phi)$ for each $\rho\ge\rho_\epsilon$.
    \smallskip
    \item Let $N_\epsilon$ be the constant $N$ in Corollary~\ref{t6.1} with $\rho$ set to be $\rho_\epsilon$. Then for all $i\in\{1,\,2\}$, $x_i,\,y_i\in T_i$, and $(\psi_1,\psi_2)\in\cP_{\epsilon}(\rho_\epsilon,D,\phi)$, we have
    \begin{equation}\label{QIEsplitting}
        \kappa^{-1}d(x_i,y_i)-C-2N_\epsilon\le d(\psi_i(x_i),\psi_i(y_i))\le \kappa d(x_i,y_i)+C+2N_\epsilon.
    \end{equation}
    \end{enumerate}
\end{prop}
 To prove Proposition~\ref{t6.2}, we first prove some properties of $D_{\epsilon,\rho}'$. 

\begin{lemma}\label{t7.0}
    Suppose $\epsilon\in (0,1)$, $\rho>1/\epsilon$, $R_1>0$,  and $\zeta\in (0,1)$ are constants. Suppose $D$ is a random subset of $\X$, Then there exists a constant $R_2>R_1$ such that if $R> R_2$, $x,\,y\in \X$, $d(x,y)>R$, and $S\subseteq B(x,R_1)$, then
    \begin{equation}\label{e7.0a}
        \P\bigl(y\in D'_{\epsilon,\rho}\,\big|\,D\cap B(x,R_1)=S \bigr)\ge (1-\zeta)\P\bigl(y\in D_{\epsilon,\rho}'\bigr).
    \end{equation}

\end{lemma}
\begin{proof}
   Let $\cA$ be the event $\{D\cap B(x,R_1)=S\}$. We choose a constant $R_\zeta>\rho$ such that 
    \begin{equation}\label{e6.7a}
        \P(\rho_0(D,y,\epsilon)>R_\zeta)<\zeta \P(\cA)(1-p)
    \end{equation}
    by \eqref{taildis} in Lemma~\ref{t5.0b}. Define $\cS\=\bigl\{E\cap B(y,R_\zeta+2n_\epsilon):E\subseteq \X,\,\rho_0(E,y,\epsilon)\le R_\zeta,\,y\notin E'_{\epsilon,\rho}\bigr\}$. Then by Fact~\ref{f7.1}, given $E\subseteq\X$ with $\rho_0(E,y,\epsilon)\le R_\zeta$, we have that
    \begin{equation}\label{e6.7b}
        y\notin E'_{\epsilon,\rho}\quad \text{ if and only if }\quad E\cap B(y,R_\zeta+2n_\epsilon)\in \cS.
    \end{equation}
    Thus, by \eqref{e6.7b} and Definition~\ref{d6.4}, we deduce that
    \begin{align}
        &\P(D\cap B(y,R_\zeta+2n_\epsilon)\in \cS)\notag\\
        &\qquad=\P(D\cap B(y,R_\zeta+2n_\epsilon)\in \cS,\,\rho_0(D,y,\epsilon)\le R_\zeta)+\P(D\cap B(y,R_\zeta+2n_\epsilon)\in \cS,\,\rho_0(D,y,\epsilon)> R_\zeta)\notag\\
        &\qquad\le\P\bigl(y\notin D_{\epsilon,\rho}',\,\rho_0(D,y,\epsilon)\le R_\zeta\bigr)+\P(\rho_0(D,y,\epsilon)> R_\zeta)=\P\bigl(y\notin D_{\epsilon,\rho}'\bigr).\label{e6.7c}
    \end{align}
    
    If $d(y,x)>R_\zeta+R_1+2n_\epsilon$, we have that $D\cap B(x,R_1)$ and $D\cap B(y,R_\zeta+2n_\epsilon)$ are independent. This implies that the events $\cA$ and $\{D\cap B(y,R_\zeta+2n_\epsilon)\in \cS\}$ are independent. Thus, by \eqref{e6.7b}, \eqref{e6.7a}, \eqref{e6.7c}, and the fact that $\P\bigl(y\in D_{\epsilon,\rho}'\bigr)\le\P(y\in D)\le p$, we have
    \begin{align*}
        \P\bigl(y\in D'_{\epsilon,\rho},\cA \bigr)
        &=\P(\cA)-\P\bigl(y\notin D'_{\epsilon,\rho},\,\cA,\,\rho_0(D,y,\epsilon)\le R_\zeta \bigr)-\P\bigl(y\notin D'_{\epsilon,\rho},\,\cA,\,\rho_0(D,y,\epsilon)> R_\zeta \bigr)\\
        &\ge \P(\cA)- \P(D\cap B(y,R_\zeta+2n_\epsilon)\in \cS,\,\cA)-\P(\rho_0(D,y,\epsilon)>R_\zeta)\\
        &\ge\P(\cA)-\P(D\cap B(y,R_\zeta+2n_\epsilon)\in \cS)\P(\cA)-\zeta\P(\cA)(1-p)\\
        &\ge \P(\cA)-\P\bigl(y\notin D_{\epsilon,\rho}'\bigr)\P(\cA)-\zeta\P(\cA)(1-p)\ge (1-\zeta)\P\bigl(y\in D_{\epsilon,\rho}'\bigr)\P(\cA).
    \end{align*}
    Therefore, \eqref{e7.0a} holds with $R_2\= R_\zeta+R_1+2n_\epsilon$.
\end{proof}
\begin{cor}\label{t6.5}
    Suppose $\epsilon\in(0,1/100)$ and $\rho\ge\rho_\epsilon$ are constants, where $\rho_\epsilon$ is the constant defined in Definition~\ref{d6.1}. Suppose $D$ is a random subset of $\X$. Suppose $x_1$ and $x_2$ are two points in $T_1$. Then almost surely there exists a point $y$ in $T_2$ such that $(x_1,y)\in D_{\epsilon,\rho}'$ and $(x_2,y)\in D_{\epsilon,\rho}'$. The same statement holds for $T_2$.
\end{cor}
\begin{proof}
    It suffices to prove the statement for distinct $x_1$ and $x_2$ in $T_1$. Write $S_y\=\{(x_1,y),\,(x_2,y)\}$ for each $y\in T_2$. Let $y_0$ be an arbitrary point in $T_2$ and write $u_0\= (x_1,y_0)\in \X$. 
    
    \textit{Claim.} There exists a sequence of points $\{y_j\}_{j\in \N}$ in $T_2$ such that for each $j\in \N_0$,
    \begin{equation}\label{e6.5}
        \P\bigl(S_{y_{j+1}}\subseteq D'_{\epsilon,\rho}\,\big|\,D\in\cA_j\bigr)\ge p^2(1-6\epsilon),
    \end{equation}
    where $\cA_j\=\bigl\{E\subseteq\X:S_{y_k}\not \subseteq E'_{\epsilon,\rho}\text{ for each $k\in\{0,\,1,\,\dots,\,j\}$}\bigr\}$. We note that since $D'_{\epsilon,\rho}\subseteq D$,
    \begin{equation}\label{positiveprobforAj}
        \P(D\in\cA_j)\ge\P(S_{y_k}\not\subseteq D \text{ for each }k\in \{0,\,1,\,\dots ,\, j\})>0.
    \end{equation}
    
    \textit{Proof of Claim.} We verify the claim by giving a recursive construction of $\{y_i\}_{i\in \N}$. Assume that the points $y_0,\,y_1,\,\dots,\,y_j$ have been chosen for some $j\in \N_0$. Write $d_j\=\max_{k\in \{0,\,1,\,\dots,\,j\}}d(u_0,(x_2,y_k))$. By \eqref{taildis} in Lemma~\ref{t5.0b} and \eqref{positiveprobforAj}, there exists a constant $R_j>\rho+d_j$ such that
    \begin{equation}\label{e6.5b}
        \P(D\in\cE_j)\le 2(j+1)\zeta_0^{-1}\exp \bigl(-\zeta_0(R_j-d_j)^2 \bigr)<p^2\epsilon \P(D\in\cA_j),
    \end{equation}
    where $B_j\=B(u_0,R_j+2n_\epsilon)$ and $\cE_j\=\{E\subseteq\X:\rho_0(E,(x_i,y_k),\epsilon)>R_j-d_j\text{ for some } i\in \{1,\,2\} \text{ and some } k\in \{0,\,1,\,\dots ,\, j\}\}$. Define $\cS_j\=\{E\cap B_j:E\in\cA_j\smallsetminus\cE_j\}$. 
    
    Then by Fact~\ref{f7.1}, given $E\subseteq\X$ with $E\notin\cE_j$, we have that
    \begin{equation}\label{e6.5c}
        E\in\cA_j\quad \text{if and only if}\quad E\cap B_j\in \cS_j.
    \end{equation}
    
    By separately considering the cases $D\in\cE_j$ and $D\notin\cE_j$, and by \eqref{e6.5c} and \eqref{e6.5b}, we have
    \begin{align}\label{e6.5a}
        &\P(D\cap B_j\in \cS_j)\ge \P(D\in \cA_j)-\P(D\in\cE_j)\ge (1-\epsilon)\P(D\in\cA_j).
    \end{align}
    
    By Lemma~\ref{t7.0}, there exists $R_j'>R_j+2n_\epsilon$ such that  for each $w\in \X$ with $d(w,u_0)>R_j'$,
    \begin{equation}\label{e6.000}
        \P\bigl(w\in D'_{\epsilon,\rho}\,\big|\,D\cap B_j\in \cS_j \bigr)\ge (1-\epsilon)\P\bigl(w\in D_{\epsilon,\rho}'\bigr).
    \end{equation}
    Fix $w\in\X$ with $d(w,u_0)>R_j'>R_j+2n_\epsilon$. Recall that $\P\bigl(w\in D_{\epsilon,\rho}'\bigr)\ge p(1-\epsilon)$ by \eqref{e6.1}. Thus, by \eqref{e6.000} and the independence between the events $\{w\in D\}$ and $\{D\cap B_j\in\cS_j\}$, we have
    \begin{align}
        &\P(\rho_0(D,w,\epsilon)>\rho\,|\,w\in D,\,D\cap B_j\in \cS_j )\notag=1-\frac{\P\bigl(w\in D_{\epsilon,\rho}',\,D\cap B_j\in \cS_j\bigr)}{\P(w\in D,\,D\cap B_j\in \cS_j)}\nonumber\\
        &\qquad\le1-\frac{(1-\epsilon)\P\bigl(w\in D_{\epsilon,\rho}'\bigr)\P(D\cap B_j\in \cS_j)}{\P(w\in D)\P(D\cap B_j\in \cS_j)}\le 1-\frac{p(1-\epsilon)^2}{p}\le 2\epsilon. \label{e6.5d}
    \end{align}
    Since $d(w,u_0)>R_j+2n_\epsilon$, we have $w\notin B_j$. Thus, by \eqref{e6.5d}, 
    \begin{equation}\label{e6.5h}
    \P(\rho_0(D\cup \{w\},w,\epsilon)>\rho\,|\,D\cap B_j\in \cS_j)=\P(\rho_0(D,w,\epsilon)>\rho\,|\,w\in D,\,D\cap B_j\in \cS_j)
    \le 2\epsilon.
    \end{equation}
    
    Choose $y_{j+1}
    \in T_2$ so that $d(y_{j+1},y_0)>R_j'$. Let $C_1$ and $C_2$ be the events $\{\rho_0(D,(x_1,y_{j+1}),\epsilon)>\rho\}$ and $\{\rho_0(D,(x_2,y_{j+1}),\epsilon)>\rho\}$. Thus, as in \eqref{e6.5h}, we similarly have that for each $i\in\{1,\,2\}$,
    \begin{equation}\label{e6.5i}
    \P\bigl(C_i\,\big|\,S_{y_{j+1}}\subseteq D,\,D\cap B_j\in \cS_j \bigr)=\P\bigl(\rho_0\bigl(D\cup S_{y_{j+1}},(x_i,y_{j+1}),\epsilon\bigr)>\rho\,\big|\,D\cap B_j\in \cS_j\bigr).
    \end{equation}
     By the independence between $\bigl\{S_{y_{j+1}}\subseteq D\bigr\}$ and $\{D\cap B_j\in\cS_j\}$, \eqref{e6.5i}, and \eqref{e6.5h}, we have
    \begin{equation}\label{e6.5f}
    \begin{aligned}
        &\P\bigl(S_{y_{j+1}}\subseteq D'_{\epsilon,\rho}\,\big|\,D\cap B_j\in \cS_j \bigr)\\
        &\qquad= \P\bigl(S_{y_{j+1}}\subseteq D\,\big|\,D\cap B_j\in \cS_j \bigr)\bigl(1- \P\bigl(C_1\cup C_2\,\big|\,S_{y_{j+1}}\subseteq D,\,D\cap B_j\in \cS_j \bigr)\bigr)\\
        &\qquad\ge \P\bigl(S_{y_{j+1}}\subseteq D\bigr) \biggl(1-\sum_{i=1}^{2}\P\bigl(C_i\,\big|\,S_{y_{j+1}}\subseteq D,\,D\cap B_j\in \cS_j \bigr)\biggr)\\
        &\qquad= p^2 \biggl(1-\sum_{i=1}^{2}\P\bigl(\rho_0\bigl(D\cup S_{y_{j+1}},(x_i,y_{j+1}),\epsilon\bigr)>\rho\,\big|\,D\cap B_j\in \cS_j\bigr)\biggr)
        \ge p^2(1-4\epsilon),
    \end{aligned}
    \end{equation}
    where the last inequality utilizes the monotonicity of $\rho_0$ and \eqref{e6.5h}. By \eqref{e6.5c}, \eqref{e6.5a}, \eqref{e6.5f}, and \eqref{e6.5b},
    \begin{align*}
        &\P\bigl(S_{y_{j+1}}\subseteq D'_{\epsilon,\rho},\, D\in\cA_j\bigr)
        \ge \P\bigl(S_{y_{j+1}}\subseteq D'_{\epsilon,\rho},\,D\cap B_j\in \cS_j,\,D\notin\cE_j\bigr)\\
        &\qquad\ge\P(D\cap B_j\in \cS_j)\P\bigl(S_{y_{j+1}}\subseteq D'_{\epsilon,\rho}\,\big|\,D\cap B_j\in \cS_j \bigr)-\P(D\in \cE_j)\\
        &\qquad\ge (1-\epsilon)\P(D\in \cA_j)p^2(1-4\epsilon)-\P(D\in \cE_j)\ge p^2(1-6\epsilon)\P(D\in \cA_j).
    \end{align*}
    This verifies \eqref{e6.5}, and consequently the claim follows. Then
    $
        \lim_{i\to+\infty}\P\bigl(\bigcup_{j=0}^i\bigl\{S_{y_{j}}\subseteq D'_{\epsilon,\rho}\bigr\}\bigr)
        =1-\P(D\in\cA_0)\prod_{j=0}^{+\infty}\P\bigl(S_{y_{j+1}}\not \subseteq D'_{\epsilon,\rho}\,\big|\,D\in \cA_j\bigr)
        =1$.
\end{proof}
\begin{cor}\label{surjective}
In the setting of Corollary~\ref{t6.5}, almost surely $\pi_i\bigl(D_{\epsilon,\rho}'\bigr)=T_i$ for each $i\in \{1,\,2\}$.
\end{cor}
\begin{proof}
    By Corollary~\ref{t6.5}, for each $i\in\{1,\,2\}$ and each $x_i$ in $T_i$, there exists a full-measure subset $\cD_{x_i}\subseteq\{0,\,1\}^{\X}$ such that if $E\in\cD_{x_i}$, then $x_i\in\pi_i\bigl(E_{\epsilon,\rho}'\bigr)$. Thus, $\cD\=\bigcap_{i\in\{1,\,2\}}\bigl(\bigcap_{x_i\in T_i}\cD_{x_i}\bigr)$ has full measure. Therefore, for each $E\in\cD$ and each $i\in\{1,\,2\}$, $\pi_i\bigl(E_{\epsilon,\rho}'\bigr)=T_i$.
\end{proof}

Now we are ready to prove Proposition~\ref{t6.2}.

\begin{proof}[\bf Proof of Proposition~\ref{t6.2}]
     Statement~(i) follows from Corollary~\ref{surjective} and the fact that $D'_{\epsilon,\rho_\epsilon}\subseteq D'_{\epsilon,\rho}$ for each $\rho\ge \rho_\epsilon$. 
     
     To prove statement~(ii), let $x_1,\,y_1\in T_1$ be distinct points in $T_1$. By Corollary~\ref{t6.5}, almost surely there exists a point $z\in T_2$ such that $x\=(x_1,z)$ and $y\=(y_1,z)$ are both in $D_{\epsilon,\rho_\epsilon}'$. Then by Corollary~\ref{t6.1}, almost surely we have
    \begin{equation*}
        d((\psi_1(x_1),\psi_2(z)),\phi(x))\le N_\epsilon\quad\text{ and }\quad 
        d((\psi_1(y_1),\psi_2(z)),\phi(y))\le N_\epsilon.
    \end{equation*}
    Since $\kappa^{-1}d(x,y)-C\le d(\phi(x),\phi(y))\le \kappa d(x,y)+C$, we conclude our proof for $i=1$ by the triangle inequality. The proof for $i=2$ is the same.
\end{proof}
 
\subsection{Rigidity of the QI embedding \texorpdfstring{$\phi$}{φ}}\label{s6.3}
In this subsection, we finish the proofs of Theorems~\ref{t1.1} and~\ref{t1.2}. 

\begin{proof}[\bf{Proof of Theorem~\ref{t1.2}}]
Fix $\epsilon\=\epsilon_4/2$, where $\epsilon_4$ is the constant in Lemma~\ref{t5.4}. Let $\rho_\epsilon$ be the constant in Definition~\ref{d6.1}. For each $\rho\ge\rho_\epsilon$, denote $N(\rho)$ to be the constant $N$ in Corollary~\ref{t6.1} with respect to $\rho$. Denote $\lceil\rho\rceil\=\min\{n\in\Z:n\ge\rho\}$, and let $\cD(\rho)$ be the set of $E\subseteq \X$ with the property that for each $\phi\in\QIE_{\kappa,C}(E,\X)$, if $(\psi_1,\psi_2)\in\cP_{\epsilon}(\rho_\epsilon,E,\phi)$, then for each $z=(z_1,z_2)\in E_{\epsilon,\rho}'$, we have $d((\psi_1(z_1),\psi_2(z_2)),\phi(z))\le N(\lceil\rho\rceil).$

By Proposition~\ref{t6.2}~(i) and Corollary~\ref{t6.1}, $\cD(\rho)$ has full measure for each $\rho\ge\rho_\epsilon$. Define $\cE\=\bigcap_{\rho\ge\rho_\epsilon}\cD(\rho)$. Note that for each pair $(\rho_1,\rho_2)$ of constants with $\rho_1>\rho_2\ge\rho_\epsilon$ and $\lceil\rho_1\rceil=\lceil\rho_2\rceil$, we have $\cD(\rho_1)\subseteq\cD(\rho_2)$. Then $\cE=\bigcap_{\rho\ge\rho_\epsilon}\cD(\rho)=\bigcap_{k\ge\rho_\epsilon,k\in\Z}\cD(k)$. Thus, $\cE$ has full measure. 

We define $\cR(D,x)$ for each $x\in\X$ and each $D\subseteq \X$ as follows. Set $\trho\=\max\{\rho_\epsilon,\,\rho_0(D,x,\epsilon)\}$, and define 
\begin{equation}\label{e6.b}
    \cR(D,x)\=N(\trho+1)=6\Delta_2(\trho+1),
\end{equation}
where $\Delta_2(\trho+1)$ is the constant $\Delta_2$ in Lemma~\ref{t6.0} with $\rho$ set to be $\trho+1$ (see \eqref{e6.00c}), and $N(\trho+1)=6\Delta_2(\trho+1)$ follows from Corollary~\ref{t6.1}. It follows from Lemma~\ref{t5.0a} that $\cR(D,x)$ is almost surely finite for each $x\in \X$.

By Definition~\ref{d6.4}, we have $x\in D_{\epsilon,\trho}'$. Then $d((\psi_1(x_1),\psi_2(x_2)),\phi(x))\le N(\lceil\trho\rceil)\le N(\trho+1)=\cR(D,x)$ by our definition of $\cE$. Thus, the third statement follows from Proposition~\ref{t6.2}~(ii) and the fact that $\cE$ has full measure. 

By Lemma~\ref{t5.0b}, for each $x\in\X$, the map $E\mapsto\rho_0(E,x,\epsilon)$ is measurable. By Definitions~\ref{d5.0} and~\ref{d3.2}, $\rho_0(E,x,\epsilon)$ is translation invariant, and for each $x\in\X$, $\rho_0(E,x,\epsilon)$ is nonincreasing with respect to $E$. Then the measurability, the translation invariance, and the monotonicity follow from \eqref{e6.b} and \eqref{e6.00c}.
\end{proof}

Next we verify that points in $\X$ are close to $D'_{\epsilon,\rho}$, which implies Theorem~\ref{t1.1}, as precisely stated in the following lemma.
\begin{lemma}\label{t6.3}
    Let $D$ be a random subset of $\X$. Suppose $e\in\X$, $\epsilon\in(0,1)$ and $\rho\ge \rho_\epsilon $ are constants, where $\rho_\epsilon$ is the constant defined in Definition~\ref{d6.1}. Then almost surely there exists a constant $R_e>0$ such that for each $x\in \X$ with $d(x,e)\ge R_e$, we have $B(x,\epsilon d(e,x))\cap D'_{\epsilon,\rho}\neq\emptyset$.
\end{lemma}
We first state a fact on the cardinality of certain maximal subsets.
\begin{fact}\label{f6.0}
    Suppose $t>0$ and $x=(x_1,x_2)\in \X$. Then for the balls $B(x_1,t)$ in $T_1$, $B(x_2,t)$ in $T_2$, and $B(x,t)$ in $\X$, we have $q^{t-1}\le\abs{B(x_1,t)}=\abs{B(x_2,t)}\le q^{t+1}$ and $q^{t-1}\le\abs{B(x,t)}\le q^{2t+2}$. Moreover, suppose $A\subseteq\X$ and $S$ is a maximal subset of $A$ with pairwise distance at least $t$. Then $A\subseteq B(S,t)$ by maximality. Thus $\abs{S}\ge \abs{A}/\abs{B(x,t)}$. 
\end{fact}
\begin{proof}[\bf Proof of Lemma~\ref{t6.3}]
    Let $D$ be a random subset of $\X$. Suppose $x\in\X$ and $r>20n_\epsilon$. Let $S_{x,r}$ be a maximal subset of $B(x,r)$ with pairwise distances at least $0.4r$ and $\cA_{x,r}$ be the event $\{\rho_0(D,z,\epsilon)\le 0.1r\text{ for each $z\in S_{x,r}$}\}$.  Then we have $\abs{S_{x,r}}\ge q^{0.2r-3}$ by Fact~\ref{f6.0} and $\P(\cA_{x,r})>0$ by the FKG inequality and \eqref{e6.00}. By Fact~\ref{f7.1}, for each $y\in S_{x,r}$, conditional on $\cA_{x,r}$, the event $\bigl\{y\in D'_{\epsilon,\rho}\bigr\}$ is independent of $D\smallsetminus B(y,0.1r+2n_\epsilon)$. Thus, conditional on $\cA_{x,r}$, since $r>20n_\epsilon$, the events $\bigl\{y\in D'_{\epsilon,\rho}\bigr\}$, $y\in S_{x,r}$ are independent. Then
    \begin{equation*}
        \P\bigl(B(x,r)\cap D'_{\epsilon,\rho}=\emptyset\,\big|\,\cA_{x,r}\bigr)\le \P\bigl(S_{x,r}\cap D'_{\epsilon,\rho}=\emptyset\,\big|\,\cA_{x,r}\bigr)=\prod_{y\in S_{x,r}}\P\bigl(y\notin D'_{\epsilon,\rho}\,\big|\,\cA_{x,r}\bigr)
    \end{equation*}
    Combining this with the FKG inequality and \eqref{e6.1}, we have
    \begin{equation}\label{e6.2a}
    \P\bigl(B(x,r)\cap D'_{\epsilon,\rho}=\emptyset\,\big|\,\cA_{x,r}\bigr)\le\prod_{y\in S_{x,r}}\P\bigl(y\notin D'_{\epsilon,\rho}\bigr)\le (1-p+p\epsilon)^{\abs{S_{x,r}}}\le (1-p+p\epsilon)^{q^{0.2r-3}}.
    \end{equation}

    We then estimate the probability of $\cA_{x,r}$. By Lemma~\ref{t5.0b}, there exists $\zeta>0$ such that for each $z\in S_{x,r}$, $\P(\rho_0(D,z,\epsilon)>0.1r)< \zeta^{-1}\exp\bigl(-10^{-2}\zeta r^2\bigr)$. Thus, since $\abs{S_{x,r}}\le q^{2r+2}$ by Fact~\ref{f6.0},
    \begin{equation}\label{e6.2b}
        \P(\cA_{x,r}^c)< \abs{S_{x,r}}\zeta^{-1}\exp\bigl(-10^{-2}\zeta r^2\bigr)\le q^{2r+2}\zeta^{-1}\exp\bigl(-10^{-2}\zeta r^2\bigr).
    \end{equation}
    By separately considering $\cA_{x,r}$ and $\cA_{x,r}^c$, and also by \eqref{e6.2a}, \eqref{e6.2b}, and the fact that $\abs{\{y\in\X:d(y,e)=R\}}\le q^{2R+2}$ for each $R> 0$ (cf.~Fact~\ref{f6.0}), we have
    \begin{align*}
        &\sum_{R>0:R^2\in \Z}\,\sum_{y\in\X:d(y,e)=R} \P\bigl(B(y,\epsilon d(e,y))\cap D'_{\epsilon,\rho}=\emptyset\bigr)\\
        &\qquad \le \sum_{R>0:R^2\in \Z}\,\sum_{y\in\X:d(y,e)=R}\bigl(\P\bigl(B(y,\epsilon d(e,y))\cap D'_{\epsilon,\rho}=\emptyset\,\big|\,\cA_{y,\epsilon R}\bigr)+\P\bigl(\cA_{y,\epsilon R}^c\bigr)\bigr)\\
        &\qquad\le \sum_{R>20\epsilon^{-1} n_\epsilon :R^2\in\Z} q^{2R+2}\Bigl((1-p+p\epsilon)^{q^{0.2\epsilon R-3}}+q^{2\epsilon R+2}\zeta^{-1}\exp\bigl(-10^{-2}\zeta \epsilon^2 R^2\bigr)\Bigr)+K<+\infty, 
    \end{align*}
    where $K\=\sum_{0<R\le 20\epsilon^{-1} n_\epsilon:R^2\in\Z}\sum_{y\in\X:d(y,e)=R}\bigl(\P\bigl(B(y,\epsilon d(e,y))\cap D'_{\epsilon,\rho}=\emptyset\,\big|\,\cA_{y,\epsilon R}\bigr)+\P\bigl(\cA_{y,\epsilon R}^c\bigr)\bigr)$. Then we conclude our proof by the Borel--Cantelli lemma.
\end{proof}

Now Theorem~\ref{t1.1} follows from Corollary~\ref{t6.1}, Proposition~\ref{t6.2}, and Lemma~\ref{t6.3}.

\begin{proof}[\bf Proof of Theorem~\ref{t1.1}]
    Let $\kappa>1$, $C>0$, and $\epsilon\in(0,\epsilon_4)$ be constants, where $\epsilon_4\in(0,1/100)$ is the constant depending only on $\kappa$ given in Lemma~\ref{t5.4}. Let $D$ be a random subset of $\X$. Fix $e\in \X$. Let $\rho_\epsilon$, $N_\epsilon$, and $R_e$ be the constants from Definition~\ref{d6.1}, Proposition~\ref{t6.2}~(ii), and Lemma~\ref{t6.3}. Suppose $\phi\in\QIE_{\kappa,C}(D,\X)$ and $(\psi_1,\psi_2)\in\cP_{\epsilon}(\rho_\epsilon,D,\phi)$.
     
    Suppose $x=(x_1,x_2)\in D$. Then there exists a point $z\in \X$ with $d(z,x)\le R_e+1$ and $d(z,e)>R_e$. By Lemma~\ref{t6.3}, almost surely there exists a point $y=(y_1,y_2)\in D'_{\epsilon,\rho_\epsilon}$ such that $d(y,z)\le \epsilon d(z,e)$. Hence 
    \begin{equation}\label{e6.a1}
    d(x,y)\le d(x,z)+d(y,z)\le R_e+1+\epsilon (R_e+1+d(x,e))<2R_e+2+\epsilon d(x,e).
    \end{equation}
    Thus, since $\phi$ is a $(\kappa,C)$-quasi-isometric embedding, almost surely we have
    \begin{equation}\label{e6.a2}
        d(\pi_1(\phi(x)),\pi_1(\phi(y)))\le \kappa d(x,y)+C<\kappa\epsilon d(x,e)+\kappa (2R_e+2)+C.
    \end{equation}
    By Corollary~\ref{t6.1} and Proposition~\ref{t6.2}~(ii), almost surely we have $d(\pi_1(\phi(y)),\psi_1(y_1))\le N_\epsilon$ and $d(\psi_1(y_1),\psi_1(x_1))\le \kappa d(x,y)+C+2N_\epsilon$. Thus, by \eqref{e6.a2} and \eqref{e6.a1}, almost surely
    \begin{align*}
        d(\pi_1(\phi(x)),\psi_1(x_1))&\le d(\pi_1(\phi(x)),\pi_1(\phi(y)))+d(\pi_1(\phi(y)),\psi_1(y_1))+d(\psi_1(y_1),\psi_1(x_1))\\
        &\le 2\kappa \epsilon d(x,e) +2\kappa (2R_e+2)+2C+3N_\epsilon.
    \end{align*}
    By symmetry, therefore, almost surely
    \begin{equation}\label{e6.a3}
        d(\phi(x_1,x_2),(\psi_1(x_1),\psi_2(x_2)))\le 8\kappa R_e+4C+6N_\epsilon+8\kappa +4\kappa \epsilon d(x,e).
    \end{equation}
    Let $\cE_e$ be the full measure subset of $\{0,\,1\}^\X$ of $E\subseteq\X$ with the property that \eqref{e6.a3} holds for all $\epsilon\in (0,\epsilon_4)\cap \Q$. Then \eqref{e1.2} holds for each $E\in \cE_e$ and each $\epsilon\in (0,1)$.
\end{proof}
\begin{rem}\label{r6.1}
Let $\kappa>1$, $C>0$ be constants, and $D$ be a random subset of $\X$. Let $\rho_\epsilon$ be the constant defined in Definition~\ref{d6.1}. Then the proofs of Theorems~\ref{t1.1} and~\ref{t1.2} show that almost surely for each $\phi\in\QIE_{\kappa,C}(D,\X)$, if $(\psi_1,\psi_2)\in\cP_{\epsilon}(\rho_\epsilon,D,\phi)$, then $(\psi_1,\psi_2)$ satisfies \eqref{e1.2} in Theorem~\ref{t1.1} and \eqref{e1.3} in Theorem~\ref{t1.2} for corresponding $D$ and $\phi$.
\end{rem}
\section{Nonexistence of QI embeddings between independent samples}\label{s7}

Now we conclude the proof of Theorems~\ref{t1.3} and~\ref{t1.4} based on the preceding results. 

For each $D\subseteq \X$, recall that $\rho_0(D,x,\epsilon)$ and $\rho_1(D,x,\epsilon)$ are defined in Definition~\ref{d5.0} which characterize the irregularity of $D$ seen from $x$. The field $\cR(D,x)$ in Theorem~\ref{t1.2} is mainly determined by $\rho_0(D,x,\epsilon)$. In this section, we focus our attention on the properties of $D'_{\epsilon,\rho}=\{x\in D:\rho_0(D,x,\epsilon)\le \rho\}$ and $\widehat D'_{\epsilon,\rho}=\{x\in D:\rho_1(D,x,\epsilon)\le \rho\}$ (see Definition~\ref{d6.4}). These sets consist of the points in $D$ where regularity is observed. We first prove a probabilistic estimate on the cardinality of $D_{\epsilon,\rho}'$ and $\widehat D'_{\epsilon,\rho}$ used in the proof of Theorems~\ref{t1.3} and~\ref{t1.4}. 

In the remainder of this section, we use $B_i(\,\cdot\,,\,\cdot\,)$ to denote balls in $T_i$ for each $i\in \{1,\,2\}$.

\begin{lemma}\label{t7.0a}
    Suppose $\kappa>1$, $C>0$, $\epsilon\in(0,1)$, and $\rho\ge\rho_\epsilon$ are constants, where $\rho_\epsilon$ is given in Definition~\ref{d6.1}. Let $D$ be a random subset of $\X$. Then for each pair of constants $R_1>r>0$ and each $\xi\in (0,1)$, there exists a constant $R_2>0$ such that if $x=(x_1,x_2)\in \X$, $R\ge R_2$, and $I$ is a maximal subset of either $B_1(x_1,R_1)\times B_2(x_2,R)$ or $B_1(x_1,R)\times B_2(x_2,R_1)$ such that distinct points in $I$ have pairwise distance at least $r$, then
    \begin{equation*}
        \P\bigl(\Absbig{I\cap D'_{\epsilon,\rho}}\ge p(1-2\epsilon)\abs{I}\bigr)\ge 1-\xi\quad\text{ and }\quad \P\bigl(\Absbig{I\cap \widehat D'_{\epsilon,\rho}}\ge p(1-2\epsilon)\abs{I} \bigr)\ge 1-\xi.
    \end{equation*}
\end{lemma}
\begin{proof}
    Recall that $\widehat D'_{\epsilon,\rho}\subseteq D'_{\epsilon,\rho}$ (see Definition~\ref{d6.4}). Then it suffices to prove the lemma for $\widehat D'_{\epsilon,\rho}$ and $I\subseteq B_1(x_1,R_1)\times B_2(x_2,R)$, and we perform a variance estimate. Let $R_1>r>0$ and $\xi\in (0,1)$ be constants. For each $x\in \X$, define a random variable $a_x\=\boldone_{\{x\in \widehat D'_{\epsilon,\rho}\}}$ as the indicator of the event $\bigl\{x\in \widehat D'_{\epsilon,\rho}\big\}$. Write $p_1\=\P(a_x=1)$. By the transition invariance of $\rho_1(D,x,\epsilon)$, $p_1$ is independent of the choice of $x\in\X$, and by \eqref{e6.1}, we have $p_1\ge p(1-\epsilon)$. 
    
    For each pair of points $x,\,y\in \X$, let $\cB_{xy}$ be the event $\{\max\{\rho_1(D,x,\epsilon),\,\rho_1(D,y,\epsilon)\}\le d(x,y)/2\}$ and write $p_{xy}\=\P(\cB_{xy})$. Then by Fact~\ref{f7.1}, if $x$ and $y$ satisfy $d(x,y)>2\rho+4n_\epsilon$, then conditional on $\cB_{xy}$, $a_x$ and $a_y$ are independent. We then have by direct computation that $\E[(a_x-p_1/p_{xy})(a_y-p_1/p_{xy})\,|\,\cB_{xy}]=0$. Thus, by separately considering $\cB_{xy}$ and $\cB_{xy}^c$, we have
    \begin{equation}\label{covariance}
        \abs{\Cov (a_x,a_y)}\le \E[(a_x-p_1)(a_y-p_1)\,|\,\cB_{xy}]+(1-p_{xy})=(p_1/p_{xy}-p_1)^2+(1-p_{xy}).
    \end{equation}
    By \eqref{taildis} in Lemma~\ref{t5.0b}, we have $p_{xy}>1-2\zeta^{-1}e^{-\zeta d(x,y)^2/4}$ for some constant $\zeta>0$. Then for each $\xi\in (0,1)$, there exists a constant $N_1>0$ such that for each pair of points $x,y\in \X$ with $d(x,y)\ge N_1$, we have $\abs{\Cov (a_x,a_y)}<2^{-1}\xi\epsilon^2p^2$. Then by Fact~\ref{f6.0}, we now have
    \begin{align*}
        &\E\bigl[\bigl(\Absbig{I\cap\widehat D'_{\epsilon,\rho}}-\E\bigl[\Absbig{I\cap \widehat D'_{\epsilon,\rho}}\bigr]\bigr)^2\bigr]\\
        &\qquad =\sum_{x,y\in I,d(x,y)\le N_1}\E[(a_x-p_1)(a_y-p_1)]+\sum_{x,y\in I,d(x,y)>N_1}\E[(a_x-p_1)(a_y-p_1)]\\
        &\qquad \le \abs{I}q^{2N_1+3}+ \abs{I}^2 2^{-1}\xi\epsilon^2p^2 .
    \end{align*}
    By Fact~\ref{f6.0} and our hypothesis on $I$, we have $\abs{I}\ge q^{R_1+R-2r-4}$. Thus, $\E\bigl[\bigl(\Absbig{I\cap\widehat D'_{\epsilon,\rho}}-\E\bigl(\Absbig{I\cap \widehat D'_{\epsilon,\rho}}\bigr)\bigr)^2\bigr]\le \xi \epsilon^2p^2\abs{I}^2$ for sufficiently large $R_2$. Then since $\E\bigl(\Absbig{I\cap\widehat D'_{\epsilon,\rho}}\bigr)=p_1\abs{I}\ge p(1-\epsilon)\abs{I}$, by Chebyshev's inequality (see e.g.\ \cite[Theorem~1.6.4]{Durrett_2019}), we have 
    \begin{equation*}
        \P\bigl(\Absbig{I\cap \widehat D'_{\epsilon,\rho}}< p(1-2\epsilon)\abs{I}\bigr)\le (\epsilon p\abs{I})^{-2}\E\bigl[\bigl(\Absbig{I\cap\widehat D'_{\epsilon,\rho}}-\E\bigl[\Absbig{I\cap \widehat D'_{\epsilon,\rho}}\bigr]\bigr)^2\bigr]\le \xi.\qedhere
    \end{equation*}
\end{proof}

Next we prove the statement of Theorem~\ref{t1.3}. The heuristic is that the splitting structure of $\phi$ given in Theorem~\ref{t1.2} preserves the product structure of $\X$. For a given finite area $A \=  B_1(x_1,R_1)\times B_2(x_2,R_2)$, there are not so many choices for $\phi|_A$'s when $\phi$ varies in $\QIE(D_1,\X)$, compared with the probability of $\{\phi(D_1\cap A)\subseteq D_2\}$, whose decay rate is exponential with respect to $\abs{A}$. Then the Borel--Cantelli lemma allows to conclude nonexistence of QI embedding almost surely.
\begin{proof}[\bf Proof of Theorem~\ref{t1.3}]

    By Proposition~\ref{t5.7}, every $\phi\in\QIE(D_1,\X)$ either preserves the boundary of each factor tree or exchanges them. By symmetry, we assume that $\phi$ preserves the boundary of each factor tree.
    
    Let $\kappa>1$, $C>0$, $\epsilon\in(0,\epsilon_4)$, and $\rho>\rho_\epsilon$ be constants, where $\epsilon_4$ and $\rho_\epsilon$ are the constants given in Lemma~\ref{t5.4} and Definition~\ref{d6.1}, respectively. Let $N_\epsilon$ be the constant given in Proposition~\ref{t6.2}. Let $D_1$ and $D_2$ be independent random subsets of $\X$, and write $D_1'\=(D_1)'_{\epsilon,\rho}$. Consider a point $x=(x_1,x_2)\in D_1'$. For each $y=(y_1,y_2)\in\X$, let $\cM(y)$ be the set of maps $\phi:D_1\to \X$ with $\phi(x)=y$. Since $\QIE_{\kappa,C}(D_1,\X)\subseteq \bigcup_{y\in\X}\cM(y)$, and $\X$ is a countable set, to prove Theorem~\ref{t1.3}, it suffices to verify that for each $y=(y_1,y_2)\in\X$,
    \begin{equation*}
        \P(\QIE_{\kappa,C}(D_1,D_2)\cap \cM(y)=\emptyset)=1.
    \end{equation*}
    
    Fix $y=(y_1,y_2)\in\X$. Let $R_2>R_1>12\kappa N_\epsilon+4\kappa C+2$ be two constants whose exact choices are to be determined later. Write $r\=12\kappa N_\epsilon+4\kappa C+2$. Define $\widetilde{D_2} \=  B(D_2,N_\epsilon)$. Consider a large fixed area $A \=  B_1(x_1,R_1)\times B_2(x_2,R_2)\subseteq \X$. We will estimate the probability of the event that there exists $\phi\in \QIE_{\kappa,C}(D_1,\X)\cap \cM(y)$ such that $\phi(A\cap D_1')\subseteq D_2$. 

    Let $\cS$ be the set of pairs $(\varphi_1,\varphi_2)$ of maps satisfying the following properties:
    \begin{enumerate}[label=\rm{(\roman*)}]
    \smallskip\item $\varphi_i\:T_i\rightarrow T_i$ is a $(\kappa,C+2N_\epsilon)$-quasi-isometric embedding for each $i\in\{1,\,2\}$.
    \smallskip\item $d(\varphi_i(x_i),y_i)\le N_\epsilon$ for each $i\in\{1,\,2\}$.
    \end{enumerate}
    
    Suppose $\phi\in\QIE_{\kappa,C}(D_1,\X)\cap \cM(y)$. By Corollary~\ref{t6.1} and Proposition~\ref{t6.2}~(ii), if $(\psi_1,\psi_2)\in\cP_{\epsilon}(\rho_\epsilon,D_1,\phi)$, then almost surely $(\psi_1,\psi_2)\in\cS$.
    
    Suppose $(\varphi_1,\varphi_2)\in\cS$ and $G\subseteq\X$. Let $I$ be a maximal subset of $A$ such that distinct points in $I$ have pairwise distances at least $r$. Then by property (i) of $\varphi_1$ and $\varphi_2$, the map $\varphi_1\times \varphi_2$ is a $(2\kappa,2C+4N_\epsilon)$-quasi-isomeric embedding, and then it is injective on $I$ and distinct points in $(\varphi_1\times \varphi_2)\bigl(G'_{\epsilon,\rho_\epsilon}\cap I\bigr)$ have distance no less than $r(2\kappa)^{-1}-2C-4N_\epsilon>2N_\epsilon$. Thus, the events of the inclusion in $\widetilde{D_2}$ of each individual point in $(\varphi_1\times \varphi_2)\bigl(G'_{\epsilon,\rho_\epsilon}\cap I\bigr)$ are independent. Then, since $\P\bigl(z\in \widetilde{D_2}\bigr)\le p_2\=1-(1-p)^{q^{2N_\epsilon+2}}$ for each point $z\in\X$ (cf.~Fact~\ref{f6.0}), we have
    \begin{equation}\label{e7.1a}
        \P\bigl((\varphi_1\times \varphi_2)\bigl(G'_{\epsilon,\rho_\epsilon}\cap A\bigr)\subseteq \widetilde{D_2}\bigr)\le\P\bigl((\varphi_1\times \varphi_2)\bigl(G'_{\epsilon,\rho_\epsilon}\cap I\bigr)\subseteq \widetilde{D_2}\bigr)\le p_2^{\abs{G'_{\epsilon,\rho_\epsilon}\cap I}}.
    \end{equation}
    Let $\cE$ denote the event that $\abs{I\cap D_1'}\ge p(1-2\epsilon)\abs{I}$. By Lemma~\ref{t7.0a}, if $R_2>0$ is sufficiently larger than $R_1$, then
    \begin{equation}\label{e7.1b}
        \P(\cE)\ge 1-\epsilon.
    \end{equation}
    
   On the other hand, we estimate the cardinality of the set $\Phi\=\{(\varphi_1\times\varphi_2)|_A: (\varphi_1,\varphi_2)\in\cS\}$. We only need to count the possible choices of $\varphi_i(z_i)$ for each $i\in\{1,\,2\}$ and each $z_i\in B_i(x_i,R_i)$ when $(\varphi_1,\varphi_2)\in\cS$. By property (ii) and Fact~\ref{f6.0}, there are at most $q^{N_\epsilon+1}< q^{\kappa+C+2N_\epsilon+1}$ choices of $\varphi_i(x_i)$. Since $\varphi_i\in \QIE_{\kappa,C+2N_\epsilon}(T_i,T_i)$, it maps neighboring points to points with distance at most $\kappa+C+2N_\epsilon$. Then the image of each point in $B(x_i,R_i)$ under $\varphi_i$ has no more than $q^{\kappa+C+2N_\epsilon+1}$ choices if the image of one of its neighbors has been determined. Then by Fact~\ref{f6.0}, there are at most $q^{(\kappa+C+2N_\epsilon+1)q^{R_i+1}}$ choices of the map $\varphi_i|_{B_i(x_i,R_i)}$. Thus, the cardinality of $\Phi$ is at most $M\=q^{(\kappa+C+2N_\epsilon+1) (q^{R_1+1}+q^{R_2+1})}$. Moreover, by Fact~\ref{f6.0}, $\abs{I}\ge \xi q^{R_1+R_2-2}$ for $\xi\=q^{-2r-2}$. Then by Corollary~\ref{t6.1}, Proposition~\ref{t6.2}~(ii), \eqref{e7.1a}, and \eqref{e7.1b}, we have
    \begin{align*}
        &\P(\QIE_{\kappa,C}(D_1,D_2)\cap \cM(y)\neq\emptyset)\le\P(\exists \phi\in\QIE_{\kappa,C}(D_1,\X)\cap \cM(y)\text{ such that } \phi(A\cap D_1')\subseteq D_2)\\
        &\quad \le \P(\cE^c) 
         +\P\bigl(\exists \phi\in\QIE_{\kappa,C}(D_1,\X)\cap \cM(y)\text{ such that } \phi(A\cap D_1')\subseteq D_2\,\big|\,\cE\,\bigr) \\
        &\quad\le \P(\cE^c) 
         +\P\bigl(\,\text{$\exists(\varphi_1,\varphi_2)\in\cS$},\, (\varphi_1\times \varphi_2)(D_1'\cap A)\subseteq \widetilde{D_2}\,\big|\,\cE\,\bigr)\\
        &\quad\le \P(\cE^c)+M\max\bigl\{ \P\bigl((\varphi_1\times \varphi_2)(D_1'\cap A)\subseteq \widetilde{D_2}\,\big|\,\cE\bigr) : (\varphi_1,\varphi_2)\in\cS \bigr\}\\
        &\quad\le \epsilon+Mp_2^{p(1-2\epsilon)\abs{I}} 
        \le \epsilon+q^{(\kappa+C+2N_\epsilon+1) (q^{R_1+1}+q^{R_2+1})}p_2^{p(1-2\epsilon)\xi q^{R_1+R_2-2}}
        \le 2\epsilon,
    \end{align*}
   where the last inequality holds for sufficiently large $R_2$ and $R_1$. Since the choice of $\epsilon\in(0,\epsilon_4)$ is arbitrary, $\QIE_{\kappa,C}(D_1,D_2)$ is empty almost surely. 
\end{proof}
We now prove Theorem~\ref{t1.4} on the structure of self QI embeddings. The proof differs from that of Theorem~\ref{t1.3} in the sense that for a self QI embedding $f:D\to D$, we no longer have the independence between two samples $D_1$ and $D_2$. Instead, we utilize our definition of $\cR(D,x)$ to obtain a similar probabilistic estimate when $d(x,\phi(x))$ is large.
\begin{proof}[\bf{Proof of Theorem~\ref{t1.4}}]
    
    For simplicity of exposition, we assume that the map $\phi$ does not interchange the two factor trees in Theorem~\ref{t1.2} ~ (iii). The proof for the other half is analogous.
    
    Let $\kappa>1$, $C>0$, $\epsilon\=\epsilon_4/2$, $\rho\=\rho_\epsilon$, and $N\=N_\epsilon$ be constants, where $\epsilon_4$ is the constant depending only on $\kappa$ in Lemma~\ref{t5.4}, and $\rho_\epsilon$ and $N_\epsilon$ are from Definition~\ref{d6.1} and Proposition~\ref{t6.2}~(ii). For each $x=(x_1,x_2)\in \X$ and each $y=(y_1,y_2)\in \X$, let $\cS(x,y)$ be the set of pairs $(\varphi_1,\varphi_2)$ satisfying the following properties :
    \begin{enumerate}[label=\rm{(\roman*)}]
    \smallskip\item $\varphi_i\:T_i\rightarrow T_i$ is a $(\kappa,C+2N)$-quasi-isometric embedding for each $i\in\{1,\,2\}$.
    \smallskip\item $\varphi_i(x_i)=y_i$ for each $i\in\{1,\,2\}$.
    \end{enumerate}
    
    Define $\epsilon_a\= 1.1\epsilon$ and $\epsilon_b\= 1.2\epsilon$. For each $E\subseteq\X$, write $\widetilde{E}\=B(E,N)$, $E_\ddagger\=\hE'_{\epsilon,\rho}$, and $E_\dagger\=E'_{\epsilon,\rho}$ (see Definition~\ref{d6.4}), and define $E_c\=\bigl\{x\in E:\sup\bigl\{\cL_{U_0^{\epsilon,E}(F)}\bigl(10^{-6}\epsilon_c^8,x\bigr):F\in \cF_x\bigr\}\le \rho\bigr\}$ for $c\in\{a,\,b\}$. Then by Definition~\ref{d5.0} $E_\ddagger\subseteq E_a\subseteq E_b\subseteq E_\dagger$ for each $E\subseteq\X$.
    
    Let $D$ be a random subset of $\X$. Let $\cQ(x,y)$ denote the event that there exists $(\varphi_1,\varphi_2)\in \cS(x,y)$ such that $(\varphi_1\times \varphi_2)(D_\dagger)\subseteq \widetilde{D}$.

    \smallskip
    
    \textit{Claim.}
        There exists $\lambda>0$ such that if $x,\,y\in\X$ satisfies $d(x,y)\ge\lambda$, then $\P(\cQ(x,y))=0$.

    \smallskip

    If the claim is proved, let $\cD$ be the full-measure subset of $\{0,\,1\}^{\X}$ where $\cQ(x,y)$ does not happen for each pair of $x,\,y\in \X$ with $d(x,y)\ge \lambda$. By Corollary~\ref{t6.1} and Proposition~\ref{t6.2}~(ii), for almost every $E$ in $\cD$, if $\phi\in \QIE_{\kappa,C}(E,E)$ and $(\psi_1,\psi_2)\in\cP_{\epsilon}(\rho_\epsilon,E,\phi)$, then $(\psi_1\times\psi_2)(E_\dagger)\subseteq\widetilde{E}$ and $(\psi_1,\psi_2)\in\cS(z,(\psi_1(z_1),\psi_2(z_2)))$ for each $z=(z_1,z_2)\in E$, and thus $d(z,(\psi_1(z_1),\psi_2(z_2)))< \lambda$.
    
    Therefore, by Remark~\ref{r6.1} and Theorem~\ref{t1.2}, almost surely for all $\phi\in\QIE_{\kappa,C}(D,D)$, $z\in D$, and $(\psi_1,\psi_2)\in\cP_{\epsilon}(\rho_\epsilon,D,\phi)$, we have $d(\phi(z),z)\le d(z,(\psi_1(z_1),\psi_2(z_2))+d(\phi(z),(\psi_1(z_1),\psi_2(z_2))\le\lambda+ \cR(D,z)$. This verifies Theorem~\ref{t1.4}.
    
    \smallskip
    
    \textit{Proof of Claim.}
        Let $\lambda>0$ be a constant whose exact choice is determined later. Suppose $x=(x_1,x_2),\,y=(y_1,y_2)\in\X$ with $d(x,y)\ge \lambda $. By symmetry, assume $d(x_1,y_1)\ge 2^{-1}\lambda $. Let $R_2>R_1>r>12\kappa N+4\kappa C+2$ be constants whose exact choices are determined later. Consider $A\=B_1(x_1,R_1)\times B_2(x_2,R_2)\subseteq\X$. Then as in the proof of Theorem~\ref{t1.3}, the cardinality of the set $\Phi\=\{(\varphi_1\times\varphi_2)|_A: (\varphi_1,\varphi_2)\in\cS\}$ is at most $M\=q^{(\kappa+C+2N+1)(q^{R_1+1}+q^{R_2+1})}$.

     Suppose $(\varphi_1,\varphi_2)\in \cS(x,y)$. We first estimate $\P\bigl((\varphi_1\times\varphi_2)(I\cap D_b)\subseteq \widetilde{D}\cap I'\bigr)$. 
     
     Let $I$ be a maximal subset of $A$ such that distinct points in $I$ have pairwise distances at least $r$, and write $I'\=(\varphi_1\times \varphi_2)(I)$. By Fact~\ref{f6.0}, we deduce $\abs{I}\ge q^{R_1+R_2-2r-4}$. As in the proof of Theorem~\ref{t1.3}, the map $\varphi_1\times \varphi_2$ is injective on $I$ and points in $I'$ have pairwise distances at least $r (2\kappa)^{-1}-2C-4N>2N$. Moreover,
     \begin{equation}\label{distanceIandimage}
       d(I,I')\ge d(B(x_1,R_1),\varphi_1(B(x_1,R_1))> 2^{-1}\lambda-2\kappa R_1-C-2N
    \end{equation}
     by our assumption that $(\varphi_1,\varphi_2)\in \cS(x,y)$ and $d(x_1,y_1)\ge2^{-1}\lambda$. Note that $p_2\= \P\bigl(z\in \widetilde{D}\bigr)\le 1-(1-p)^{q^{2N+2}}$ for each $z\in \X$ (cf.~Fact~\ref{f6.0}). Then for each $K\subseteq I$, by the independence of the inclusion in $\tD$ of each individual point in a collection of points with pairwise distances $>2N$,
     \begin{equation}\label{eqthmd1}
         \P\bigl((\varphi_1\times\varphi_2)(K)\subseteq\widetilde{D}\bigr)= p_2^{\abs{K}}.
     \end{equation}
    
    For each $J\subseteq \X$, let $\cG_J^{a}$, $\cG_J^b$, and $\cH_J$ denote the events $\{D_a\cap I\subseteq J\}$, $\{D_b\cap I\subseteq J\}$, and $\bigl\{\widetilde{D}\cap I'=J\bigr\}$, respectively. Let $\cE_1$ and $\cE_2$ denote the events $\{\abs{D_\ddagger\cap I}\ge p(1-2\epsilon)\abs{I}\}$ and $\{\abs{D_a\cap I}\ge p(1-2\epsilon)\abs{I}\}$, respectively. Fix $\zeta\in(0,1)$. By Lemma~\ref{t7.0a}, for sufficiently large $R_2$,
    \begin{equation}\label{eqthmd2}
        \P(\cE_1)\ge 1-\zeta.
    \end{equation}
    
    To apply an analogous argument as in the proof of Theorem~\ref{t1.3}, we need to remove the conditional $\cH_{J'}$ by altering the constants. Thus, we now verify that for suitable choices of constants and each $J\subseteq I$, 
    \begin{equation}\label{e7.2c0}
        \P\bigl(\cG_{J}^b,\,\cE_1\,\big|\,\cH_{J'}\bigr)\le \P(\cG_{J}^a,\,\cE_2),
    \end{equation}
    where $J'\=(\varphi_1\times\varphi_2)(J)$. Define independent random subsets $D_\alpha\=D\cap B(I',N)$ and $D_\beta\=D\smallsetminus B(I',N)$. Then for each $J\subseteq I$, we have $\cH_{J'}=\bigl\{\widetilde{D_\alpha}\cap I'=J'\bigr\}$ and
    \begin{equation}\label{eqthmd3}
    \begin{aligned}
        \P\bigl(\cG_{J}^b,\,\cE_1\,\big|\,\cH_{J'}\bigr)&\le \P\bigl((D_\beta)_b\cap I\subseteq J,\,\abs{(D_\beta\cup B(I',N))_\ddagger\cap I}\ge  p(1-2\epsilon)\abs{I}\,\big|\,\widetilde{D_\alpha}\cap I'=J'\bigr).\\
        &=\P((D\smallsetminus B(I',N))_b\cap I\subseteq J,\,\abs{(D\cup B(I',N))_\ddagger\cap I}\ge p(1-2\epsilon)\abs{I}).
    \end{aligned}
    \end{equation}
    We choose $r\=10^6\epsilon^{-8}\kappa N$ and $\lambda\=2r+4\kappa R_1+4\kappa C$ for the remainder of the proof. Then direct estimate on the sizes of balls yields $\abs{B(z,R)\cap B(I',N)\cap F}\le 10^{-6}\epsilon^8\abs{B(z,R)\cap F}$ for each $z\in I$, each flat $F$ containing $z$, and each $R>0$, by \eqref{distanceIandimage} and the fact that points in $I'$ have pairwise distances greater than $r(2\kappa)^{-1}-2C-4N$. This, with simple estimates, implies $(D\cup B(I',N))_\ddagger\cap I\subseteq D_a\cap I\subseteq (D\smallsetminus B(I',N))_b\cap I$ using Definitions~\ref{d3.2},~\ref{d5.0}, and~\ref{d6.4}. Thus, combining \eqref{eqthmd3}, we get \eqref{e7.2c0}. 

    Thus, by \eqref{e7.2c0} and \eqref{eqthmd1}, 
    \begin{equation*}
    \begin{aligned}
    &\P\bigl((\varphi_1\times\varphi_2)(I\cap D_b)\subseteq \widetilde{D}\cap I',\,\cE_1\bigr)\\
    &\qquad=
        \sum_{J\subseteq I}\P\bigl(\cG_{J}^b,\, \cE_1\,\big|\,\cH_{J'}\bigr)\P(\cH_{J'})\le \sum_{J\subseteq I}\P(\cG_{J}^a,\, \cE_2)\P(\cH_{J'})\\
        &\qquad=\sum_{J\subseteq I}\,\sum_{K\subseteq J:\abs{K}\ge p(1-2\epsilon)\abs{I}}\P(D_a\cap I=K)\P(\cH_{J'})\\
        &\qquad=\sum_{K\subseteq I:\abs{K}\ge p(1-2\epsilon)\abs{I}}\P(D_a\cap I=K)\P\bigl((\varphi_1\times \varphi_2)(K)\subseteq \widetilde{D}\bigr)\le p_2^{p(1-2\epsilon)\abs{I}}.
    \end{aligned}
    \end{equation*}
    
    Therefore, by \eqref{eqthmd2},
    \begin{align*}
        \P(\cQ(x,y))&\le\P(\cE_1^c)+\P\bigl(\exists (\varphi_1,\varphi_2)\in\cS(x,y)\text{ such that } (\varphi_1\times\varphi_2)(I\cap D_b)\subseteq \widetilde{D}\cap I',\,\cE_1\bigr)\\
        &\le \zeta+ Mp_2^{p(1-2\epsilon)\abs{I}}\le\zeta+Mp_2^{p(1-2\epsilon)q^{R_1+R_2-2r-4}}\le2\zeta,
    \end{align*}
    where the last inequality holds if $R_1>r$ satisfies $-\ln(p_2)p(1-2\epsilon)q^{R_1-2r-5}\ge(\kappa+C+2N+2)\ln(q)$ and $R_2$ is sufficiently large. Since $\zeta\in(0,1)$ is arbitrary, the claim follows.
    \end{proof}
\appendix
\section{Geometric estimates}\label{app:a}
This section lists several geometric lemmas and facts used in Section~\ref{s:Step2}. The first three lemmas are discrete versions of well-known results on QI embeddings between Euclidean spaces.
\begin{lemma}[Local packing]\label{t4.0}
    For each $n\in \N$ and each $\kappa>1$, there exists a constant $\eta>\kappa$ such that if $C>1$, $R>0$, $x\in \Z^n$, and $\psi\in \QIE_{\kappa,C}(B(x,R)\cap \Z^n,\Z^n)$, then
    \begin{equation*}
        \Nb(\psi(x),R/\eta)\subseteq \Nb (\psi(B(x,R)\cap \Z^n),\eta C),
    \end{equation*}
    where $B(x,R)$ denotes the ball of radius $R$ in $\R^n$. 
\end{lemma}
The following corollary is an immediate consequence of Lemma~\ref{t4.0}.
\begin{cor}[Packing of metric interiors]\label{t4.0a}
Let $n\in\N$, $\kappa>1$, and $C>1$ be constants. Let $\eta$ be the constant given for $n$ and $\kappa$ in Lemma~\ref{t4.0}. Let $U,\, V$ be subsets of $\Z^n$ and $\psi \: U \to \Z^n$ be a $(\kappa,C)$-quasi-isometric embedding. Suppose $V$ equipped with the induced graph structure (of $\Z^n$) is connected, and $\psi^{-1}(B(V,\eta C)\cap \Z^n)$ is a nonempty subset of $\mathrm{Int}_{\Z^n}\bigl(U, 3\eta^2C\bigr)$. Then $V \subseteq B(\psi(U), \eta C)$.
\end{cor}
We also need the following lemma as a discrete version of the volume bounds for quasi-isometric embeddings between Euclidean spaces.
\begin{lemma}\label{t4.0b}
    Let $\kappa>1$ and $C>1$ be constants. Let $U$ be a finite subset of $\Z^2$ and let $\psi \: U \to \Z^2$ be a $(\kappa,C)$-quasi-isometric embedding. Then
    $\abs{B(\psi(U),C)}\ge 3^{-4}\kappa^{-2} \abs{U}$.
\end{lemma}
The next lemma is about the geometry of polyhedra in $\X$ (see Definition~\ref{d2.10}).
\begin{lemma}\label{t4.3}
    Suppose $P_1$ and $P_2$ are nonempty polyhedra in $\X$. Define $Q_1 \= \{y\in P_1: d(y,P_2)=d(P_1,P_2)\}$. Then the following hold for each $i\in \{1,\,2\}$:
    \begin{enumerate}[label=\rm{(\roman*)}]
        \smallskip     	\item  $Q_1=\pi_1(Q_1)\times \pi_2(Q_1)$ and $Q_1$ is a nonempty polyhedron. Moreover, $\pi_i(Q_1)=\{y\in \pi_i(P_1):d(y,\pi_i(P_2))=d(\pi_i(P_1),\pi_i(P_2))\}$, and if $\abs{\pi_i(Q_1)}>1$ then $\pi_i(Q_1)=\pi_i(P_1)\cap \pi_i(P_2)$.
        \smallskip     	\item For each $R>0$, $B(\pi_i(P_1),R)\cap B(\pi_i(P_2),R)\subseteq B(\pi_i(Q_1),R)$, and consequently,
        $B(P_1,R)\cap B(P_2,R)\subseteq B(Q_1,2R)$.
    \end{enumerate}
\end{lemma}
\begin{proof}
    For each $i\in\{1,\,2\}$, we have $P_i=\pi_1(P_i)\times \pi_2(P_i)$, where $\pi_j(P_i)$ is a nonempty path-connected subset of a geodesic in $T_j$ for each $j\in \{1,\,2\}$. So $d(P_1,P_2)^2=d_1(\pi_1(P_1),\pi_1(P_2))^2+d_2(\pi_2(P_1),\pi_2(P_2))^2$. Thus for each $i\in\{1,\,2\}$, we have $\pi_i(Q_1)\subseteq L_i$,
    where $L_i\coloneqq \{y\in \pi_i(P_1):d(y,\pi_i(P_2))=d(\pi_i(P_1),\pi_i(P_2))\}$. Moreover, $L_i$ is a path-connected subset of a geodesic in $T_i$ by geometry of trees. Note that for each point $x\in L_1\times L_2$, we have $d(x,P_2)=d(P_1,P_2)$. Thus, 
    \begin{equation}\label{e4.3a}
        Q_1=L_1\times L_2=\pi_1(Q_1)\times \pi_2(Q_1)
    \end{equation}is a polyhedron. We also have $\pi_i(Q_1)=\pi_i(P_1)\cap \pi_i(P_2)$ if $\abs{\pi_i(Q_1)}>1$ by geometric considerations.

    To prove property~(ii), it suffices to prove that for each $i\in \{1,\,2\}$ and each $R>0$, we have 
    \begin{equation*}
        B(\pi_i(P_1),R)\cap B(\pi_i(P_2),R)\subseteq B(\pi_i(Q_1),R).
    \end{equation*} 
    
    Fix $i\in\{1,\,2\}$ and $R>0$. Suppose $x_1\in B(\pi_i(P_1),R)\cap B(\pi_i(P_2),R)$ and let $x_2$ be the point in $\pi_i(P_1)$ closest to $x_1$. If $x_2\in \pi_i(Q_1)$, we have $d(x_1,\pi_i(Q_1))\le d(x_1,x_2)<R$. If $x_2\notin \pi_i(Q_1)$, then by \eqref{e4.3a} and the simple connectivity of $T_i$, it is not difficult to see that $d(x_1,\pi_i(P_2))=d(x_1,x_2)+d(x_2,\pi_i(Q_1))+d(\pi_i(Q_1),\pi_i(P_2))$, and consequently $d(x_1,\pi_i(Q_1))\le d(x_1,\pi_i(P_2))<R$. In both cases we have $x_1\in B(\pi_i(Q_1),R)$ and we finish our proof.
\end{proof}
\begin{fact}[Geometry of singular geodesic rays]\label{t5.2}
    Let $L_1$ and $L_2$ be two singular geodesic rays such that $L_1\sim_{a}L_2$ for some constant $a\in(0,1)$. Then $L_1$ and $L_2$ are Hausdorff equivalent.
\end{fact}
The following geometric lemma is a discrete version of \cite[Lemma~3.10]{eskin1998quasi}.

\begin{lemma}\label{t4.0c}
    Given $\kappa>1$, there exists a constant $\beta>1$ such that the following holds:

    \smallskip
    
    Suppose $U$ and $V$ are subsets of $\Z^2$. Let $B(\,\cdot\,,\,\cdot\,)$ denote a ball in $\R^2$. Suppose $\psi$ is a $(\kappa,C)$-quasi-isometry from $U$ to $V$ with $C>1$. Suppose $V\subseteq B(\psi(U),C)$, $V'\subseteq V$, $U'\=\psi^{-1}(V')$, and $B(V',9\kappa+C)$ is a simply connected subset of $\R^2$. Suppose there exist constants $m_2>0$ and $C'>1$ such that for every hyperplane $L\subseteq \Z^2$, there exist hyperplanes $L_1',\,\dots,\,L_{m_2}'$ such that
    \begin{equation}\label{ea.60}
        \psi(L\cap U)\subseteq B\bigl(L_1'\cup\cdots\cup L_{m_2}',C'\bigr).
    \end{equation}
    Then for every hyperplane $L\subseteq\Z^2$, there exists a hyperplane $L'\subseteq \Z^2$ such that $\psi(L\cap U')\cap \Int_{\Z^2}(V',\beta(C+C'))\subseteq B(L',\beta(C+C'))$. Conversely, for every hyperplane $L'\subseteq\Z^2$, there exists a hyperplane $L\subseteq \Z^2$ such that $\psi^{-1} (L'\cap \Int_{\Z^2}(V',\beta (C+C')) )\subseteq B(L,\beta (C+C'))$. 
\end{lemma}
 Since limit sets do not play a key role in this article, we present here its definition as a special case of \cite{wortman2006quasiflats}.
\begin{definition}\label{da.1}
    Suppose $\delta\in(0,1)$, $x\in \X$ and $\psi\:\R^2\to \X$. We call $\X_x(\delta)=\bigl(\bigcup_{L\in \bL_x}L[\delta]\bigr)^c$ the \emph{$\delta$-nondegenerate space} from $x$, where $\bL_x$ denotes the set of hyperplanes passing through $x$. Then a \emph{$\delta$-limit point} of $\psi$ from $x$ is a boundary point $\fC\in \widehat{\X}$ such that there exists a path $\gamma\:[0,+\infty)\to \psi^{-1}(\X_x(\delta))$ escaping every compact set such that $\lim_{t\to+\infty}\psi(\gamma(t))=\fC$. The set of all $\delta$-limit points of $\psi$ from $x$, denoted by $\fL_{\psi,x}(\delta)$, is called the \emph{$\delta$-limit set} of $\psi$ from $x$.
\end{definition}
\begin{rem}\label{ra.1}
    Note that for each pair of opposite points in $\widehat{\X}$, there exists a unique flat $F$ containing them up to Hausdorff equivalence. It is shown in \cite[Section~5]{wortman2006quasiflats} that the set of flats in Theorem~\ref{t3.4} (i.e., \cite[Theorem~1.2]{wortman2006quasiflats}) can be obtained from the pairs of opposite points in $\fL_{\psi,x}(\delta)$ in this way.
\end{rem}

\bibliographystyle{alpha}
\bibliography{refs}

@article{wortman2006quasiflats,
  title={Quasiflats with holes in reductive groups},
  author={Wortman, Kevin},
  journal={Algebr. Geom. Topol.},
  volume={6},
  pages={91--117},
  year={2006},
  publisher={Mathematical Sciences Publishers}
}

@article{eskin1997quasi,
  title={Quasi-flats and rigidity in higher rank symmetric spaces},
  author={Eskin, Alex and Farb, Benson},
  journal={J. Amer. Math. Soc.},
  volume={10},
  pages={653--692},
  year={1997},
  publisher={JSTOR}
}

@article{basu2014lipschitz,
  title={Lipschitz embeddings of random sequences},
  author={Basu, Riddhipratim and Sly, Allan},
  journal={Probab. Theory Related Fields},
  volume={159},
  pages={721--775},
  year={2014},
  publisher={Springer}
}

@article{basu2018lipschitz,
  title={Lipschitz embeddings of random fields},
  author={Basu, Riddhipratim and Sidoravicius, Vladas and Sly, Allan},
  journal={Probab. Theory Related Fields},
  volume={172},
  pages={1121--1179},
  year={2018},
  publisher={Springer}
}

@article{kleiner1997rigidity,
  title={Rigidity of quasi-isometries for symmetric spaces and {E}uclidean buildings},
  author={Kleiner, Bruce and Leeb, Bernhard},
  journal={Publ. Math. Inst. Hautes \'Etudes Sci.},
  volume={86},
  pages={115--197},
  year={1997}
}

@misc{abert2010,
  author = {Ab\'ert, Mikl\'os},
  title = {Some questions},
  year = 2010,
  note = {\url{https://users.renyi.hu/~abert/questions.pdf} [Accessed: (2010.10.02)]}
}

@article {AbertWeiss,
    AUTHOR = {Ab\'ert, Mikl\'os and Weiss, Benjamin},
     TITLE = {Bernoulli actions are weakly contained in any free action},
   JOURNAL = {Ergodic Theory Dynam. Systems},
  FJOURNAL = {Ergodic Theory and Dynamical Systems},
    VOLUME = {33},
      YEAR = {2013},
     PAGES = {323--333},
      ISSN = {0143-3857,1469-4417},
   MRCLASS = {37A05},
  MRNUMBER = {3035287},
MRREVIEWER = {Sophie\ Grivaux},
       DOI = {10.1017/S0143385711000988},
       URL = {https://doi.org/10.1017/S0143385711000988},
}

@article {AbertMellick,
    AUTHOR = {Ab\'ert, Mikl\'os and Mellick, Sam},
     TITLE = {Point processes, cost, and the growth of rank in locally
              compact groups},
   JOURNAL = {Israel J. Math.},
  FJOURNAL = {Israel Journal of Mathematics},
    VOLUME = {251},
      YEAR = {2022},
     PAGES = {47--154},
      ISSN = {0021-2172,1565-8511},
   MRCLASS = {37J37 (22D40 60G55)},
  MRNUMBER = {4555892},
       DOI = {10.1007/s11856-022-2445-9},
       URL = {https://doi.org/10.1007/s11856-022-2445-9},
}

@book{Durrett_2019, 
  place={Cambridge}, 
  edition={5}, 
  series={Cambridge Series in Statistical and Probabilistic Mathematics}, 
  title={Probability: Theory and Examples}, 
  publisher={Cambridge Univ. Press}, 
  author={Durrett, Rick}, 
  year={2019}, 
  collection={Cambridge Series in Statistical and Probabilistic Mathematics}}

@article{eskin1998quasi,
  title={Quasi-isometric rigidity of nonuniform lattices in higher rank symmetric spaces},
  author={Eskin, Alex},
  journal={J. Amer. Math. Soc.},
  volume={11},
  pages={321--361},
  year={1998}
}

@incollection{gromov1987hyper,
  title={Hyperbolic groups},
  author={Gromov, Mikhael},
  booktitle={Essays in group theory},
  journal={J. Amer. Math. Soc.},
  publisher={Springer New York},
  pages={75--264},
  year={1987}
}

@incollection {gromov1991asym,
    AUTHOR = {Gromov, Mikhael},
     TITLE = {Asymptotic invariants of infinite groups},
 BOOKTITLE = {Geometric group theory, {V}ol.\ 2 ({S}ussex, 1991)},
    SERIES = {London Math. Soc. Lecture Note Ser.},
    VOLUME = {182},
     PAGES = {1--295},
 PUBLISHER = {Cambridge Univ. Press, Cambridge},
      YEAR = {1993},
      ISBN = {0-521-44680-5},
   MRCLASS = {20F32 (57M07)},
  MRNUMBER = {1253544},
}

@article {pansu1989,
    AUTHOR = {Pansu, Pierre},
     TITLE = {M\'etriques de {C}arnot-{C}arath\'eodory et quasiisom\'etries
              des espaces sym\'etriques de rang un},
   JOURNAL = {Ann. of Math. (2)},
  FJOURNAL = {Annals of Mathematics. Second Series},
    VOLUME = {129},
      YEAR = {1989},
     PAGES = {1--60},
      ISSN = {0003-486X,1939-8980},
   MRCLASS = {53C20 (22E40)},
  MRNUMBER = {979599},
MRREVIEWER = {Gudlaugur\ Thorbergsson},
       DOI = {10.2307/1971484},
       URL = {https://doi.org/10.2307/1971484},
}

@article {schwartz1995,
    AUTHOR = {Schwartz, Richard Evan},
     TITLE = {The quasi-isometry classification of rank one lattices},
   JOURNAL = {Inst. Hautes \'Etudes Sci. Publ. Math.},
  FJOURNAL = {Institut des Hautes \'Etudes Scientifiques. Publications
              Math\'ematiques},
    NUMBER = {82},
      YEAR = {1995},
     PAGES = {133--168},
      ISSN = {0073-8301,1618-1913},
   MRCLASS = {22E40 (22E46)},
  MRNUMBER = {1383215},
MRREVIEWER = {Alexander\ Starkov},
       URL = {http://www.numdam.org/item?id=PMIHES_1995__82__133_0},
}

@book{grimmett2012percolation,
  author={Grimmett, Geoffrey},
  title={Percolation},
  year={2012},
  publisher={Springer}
  }

@article {Farb1997Survey,
    AUTHOR = {Farb, Benson},
     TITLE = {The quasi-isometry classification of lattices in semisimple
              {L}ie groups},
   JOURNAL = {Math. Res. Lett.},
  FJOURNAL = {Mathematical Research Letters},
    VOLUME = {4},
      YEAR = {1997},
     PAGES = {705--717},
      ISSN = {1073-2780},
   MRCLASS = {22E40 (20F32 22-02)},
  MRNUMBER = {1484701},
MRREVIEWER = {Lee\ Mosher},
       DOI = {10.4310/MRL.1997.v4.n5.a8},
       URL = {https://doi.org/10.4310/MRL.1997.v4.n5.a8},
}

@article {farbschwartz,
    AUTHOR = {Farb, Benson and Schwartz, Richard},
     TITLE = {The large-scale geometry of {H}ilbert modular groups},
   JOURNAL = {J. Differential Geom.},
  FJOURNAL = {Journal of Differential Geometry},
    VOLUME = {44},
      YEAR = {1996},
     PAGES = {435--478},
      ISSN = {0022-040X,1945-743X},
   MRCLASS = {22E40 (11F41 22E46)},
  MRNUMBER = {1431001},
MRREVIEWER = {Alexander\ Starkov},
       URL = {http://projecteuclid.org/euclid.jdg/1214459217},
}

@article{fraczyk2023,
      title={Poisson-Voronoi tessellations and fixed price in higher rank}, 
      author={Mikolaj Fraczyk and Sam Mellick and Amanda Wilkens},
JOURNAL = {to appear in Ann. of Math.},      
year={2023},
      eprint={2307.01194},
      archivePrefix={arXiv},
      primaryClass={math.GT},
      url={https://arxiv.org/abs/2307.01194}, 
}

@article {Furman1999ME,
    AUTHOR = {Furman, Alex},
     TITLE = {Gromov's measure equivalence and rigidity of higher rank
              lattices},
   JOURNAL = {Ann. of Math. (2)},
  FJOURNAL = {Annals of Mathematics. Second Series},
    VOLUME = {150},
      YEAR = {1999},
    NUMBER = {3},
     PAGES = {1059--1081},
      ISSN = {0003-486X,1939-8980},
   MRCLASS = {22F10 (22E40 28D15 37A15 53C24)},
  MRNUMBER = {1740986},
MRREVIEWER = {Scot\ Adams},
       DOI = {10.2307/121062},
       URL = {https://doi.org/10.2307/121062},
}

@inproceedings {GaboriauICM2010,
    AUTHOR = {Gaboriau, Damien},
     TITLE = {Orbit equivalence and measured group theory},
 BOOKTITLE = {Proceedings of the {I}nternational {C}ongress of
              {M}athematicians. {V}olume {III}},
     PAGES = {1501--1527},
 PUBLISHER = {Hindustan Book Agency, New Delhi},
      YEAR = {2010},
      ISBN = {978-81-85931-08-3; 978-981-4324-33-5; 981-4324-33-7},
   MRCLASS = {37A20 (46L10)},
  MRNUMBER = {2827853},
MRREVIEWER = {Claire\ Anantharaman-Delaroche},
}

@article {Peled2010,
    AUTHOR = {Peled, Ron},
     TITLE = {On rough isometries of {P}oisson processes on the line},
   JOURNAL = {Ann. Appl. Probab.},
  FJOURNAL = {The Annals of Applied Probability},
    VOLUME = {20},
      YEAR = {2010},
     PAGES = {462--494},
      ISSN = {1050-5164,2168-8737},
   MRCLASS = {60D05 (60K35)},
  MRNUMBER = {2650039},
MRREVIEWER = {Dimitri\ Petritis},
       DOI = {10.1214/09-AAP624},
       URL = {https://doi.org/10.1214/09-AAP624},
}
\end{document}